\documentclass[aop]{imsart}

\RequirePackage{amsthm,amsmath,amsfonts,amssymb}
\RequirePackage[numbers]{natbib}

\usepackage{txfonts}
\usepackage{graphicx}
\usepackage{wrapfig}
\usepackage{pifont}
\usepackage{color}

\newcommand{\red}[1]{{#1}}
\newcommand{\blue}[1]{{#1}}

\usepackage[colorlinks=true, pdfstartview=FitV, linkcolor=blue,
  citecolor=blue, urlcolor=blue,pagebackref=false]{hyperref}

\newtheorem{dfn}{Definition}[section]
\newtheorem{thm}[dfn]{Theorem}
\newtheorem{lem}[dfn]{Lemma}
\newtheorem{cor}[dfn]{Corollary}
\newtheorem{rem}[dfn]{Remark}
\newtheorem*{rem*}{Remark}

\newtheorem{open}[dfn]{Open problem}
\newtheorem{prop}[dfn]{Proposition}\makeatletter

\newcommand{\abbr}[1]{{\sc{\lowercase{#1}}}}

\newcommand{\ca}{{\rm Cap}}
\newcommand{\Z}{{\mathbb Z}}
\newcommand{\RR}{{\mathbb R}}
\newcommand{\BB}{{\mathbb B}}
\newcommand{\R}{{\mathcal R}}
\newcommand{\F}{{\mathcal F}}
\newcommand{\G}{{\mathcal G}}
\newcommand{\N}{{\mathbb N}}

\newcommand{\C}{{\mathcal C}}

\newcommand{\sss}{{\mathbb S}}
\newcommand{\D}{{\mathcal D}}
\newcommand{\crn}{{R_n}}
\newcommand{\ocrn}{{\overline{R}_n}}
\newcommand{\ovg}{\hat{g}}
\newcommand{\barT}{\hat T}
\newcommand{\barh}{\tilde h}

\startlocaldefs
\endlocaldefs
\numberwithin{equation}{section}
\begin{document}

\begin{frontmatter}
\title{Capacity of the range of random walk:\\ The law of the iterated logarithm}
%\title{A sample article title with some additional note\thanksref{T1}}
\runtitle{LIL of capacity of the random walk range}
%\thankstext{T1}{A sample of additional note to the title.}

\begin{aug}
%%%%%%%%%%%%%%%%%%%%%%%%%%%%%%%%%%%%%%%%%%%%%%%
%% Only one address is permitted per author. %%
%% Only division, organization and e-mail is %%
%% included in the address.                  %%
%% Additional information can be included in %%
\author[A]{\fnms{Amir}~\snm{Dembo}\ead[label=e1]{adembo@stanford.edu}}
%\author[B]{\fnms{Second}~\snm{Author}\ead[label=e2]{second@somewhere.com}\orcid{0000-0000-0000-0000}}
\and
\author[B]{\fnms{Izumi}~\snm{Okada}\ead[label=e2]{iokada@math.s.chiba-u.ac.jp}}
%%%%%%%%%%%%%%%%%%%%%%%%%%%%%%%%%%%%%%%%%%%%%%
%% Addresses                                %%
%%%%%%%%%%%%%%%%%%%%%%%%%%%%%%%%%%%%%%%%%%%%%%
\address[A]{Mathematics Department and Statistics Department, Stanford University\printead[presep={,\ }]{e1}}

\address[B]{Department of Mathematics and Informatics, Faculty of Science, Chiba University\printead[presep={,\ }]{e2}}
\end{aug}

\begin{abstract}
We establish both  the $\limsup$ and the $\liminf$ law of the iterated logarithm (\abbr{LIL}), for 
the capacity of the range of a simple random walk in any dimension $d\ge 3$. While for $d \ge 4$, 
the order of growth in $n$ of such \abbr{LIL} at dimension $d$ matches 
that for the volume of the random walk range in dimension $d-2$, somewhat 
surprisingly this correspondence breaks down for the capacity of the range at $d=3$.
We further establish such \abbr{LIL} for the Brownian capacity of a $3$-dimensional 
Brownian sample path and novel, sharp moderate deviations bounds for the capacity 
of the range of a $4$-dimensional simple random walk.
\end{abstract}



\end{frontmatter}

\section{Introduction and Main Results}
Let $\tau_A$ denote the first positive hitting time of a finite set $A$ by a simple random walk (\abbr{srw}) 
on $\Z^d$, denoted hereafter $(S_m)_{m \ge 0}$. Recall that the corresponding (Newtonian) capacity 
%of a finite set $A$, 
is given for $d \ge 3$, by
\begin{align*}
\ca (A)
:= \sum_{x \in A} 
P^x(\tau_A =\infty)
=\lim_{|z|\to \infty}\frac{P^z(\tau_A <\infty)}{\red{G(0,z)}}  
\end{align*}
\red{(where $G(x,y)$ denotes the Green's function of the walk).}
The asymptotics of the capacity $R_n:=\ca(\R_n)$ of the random walk range 
$\R_n :=  \{S_1,\ldots,S_n\}$ is relatively trivial for $d=2$ 
(for then $R_n=\frac{2+o(1)}{\pi} \log$ (diam$\R_n$), see \cite[Lemma 2.3.5]{LA91}).
In contrast, for $d \ge 3$ \blue{such asymptotics} is of an on going interest. 
% applying ergodic theory, over 50 years ago Jain and Orey \cite{JO68} established 
Indeed, the strong law
\begin{align*}
\lim_{n\to \infty}\frac{\crn}{n}=\alpha_d 
\quad \text{ a.s.,} \qquad \hbox{for all} \quad d \ge 3 \,,
\end{align*}
\blue{is an immediate consequence of the subadditive ergodic theorem, with}
$\alpha_d>0$ iff $d\ge 5$ \blue{(as shown in \cite{JO68}).} Recall Green's function for the \blue{$d$}-dimensional Brownian motion 
\begin{equation}\label{eq:green4}
G_B(x,y)
:=\int_0^\infty (2\pi t)^{-d/2} e^{-|x-y|^2/(2t)} dt
= \begin{cases} \frac{1}{2\pi}|x-y|^{-1} \,, & \qquad d=3 \,, \\
 \frac{1}{2\pi^2} |x-y|^{-2} \,, & \qquad d=4 \,, \end{cases}
\end{equation}
and the corresponding Brownian capacity of $D \subset \RR^{\blue d}$,
\begin{align*}
\ca_B(D)^{-1}
:= \inf \bigg\{ \int\int G_B(x,y) \mu(dx) \mu(dy): \mu(D)=1 \bigg\} \,.
\end{align*}
More recently, 
%with $\ca_B (B[0,t])$ denoting the $d$-dimensional Brownian capacity of the Brownian sample path up to time $t \ge 0$, 
Chang \cite{Ch17} showed that for $d=3$,
\begin{align*}
\frac{ \crn}{\sqrt{n}} \;
\stackrel{\D}{\Longrightarrow} \;
\frac{1}{3\sqrt{3}}\ca_B (B[0,1]),
\end{align*}
whereas Asselah et al \cite{As3} showed that in this case further, for some $C$ finite,
\begin{align*}
C^{-1} \sqrt{n} \le E[ \crn ]\le C \sqrt{n}\,.
\end{align*}
We denote throughout by $\overline{X}$ the centering $\overline{X} = X-E[X]$ 
of a generic random variable $X$. In higher dimensions $d \ge 4$, the centered 
capacity $\ocrn$, converges after proper scaling to a non-degenerate limit law, 
which is Gaussian iff $d \ge 5$ (see \cite{As5} for $d=4$ and 
\cite{As3, S20} for $d \ge 5$). For $d \ge 5$,
estimates of the corresponding large and moderate deviations are provided in \cite{As1} 
(but they are not sharp enough to imply a \abbr{lil}), 
while the central limit theorem (\abbr{clt}) is further established in \cite{WNS22} for $\crn$ and 
a class of symmetric $\alpha$-stable walks, provided $d > 5 \alpha/2$. 
\blue{We note in passing that similar questions for critical branching random walk on $\Z^d$ conditioned to have total population $n$ have also been studied in \cite{BW*,BW**,BW22}. }
%We also note in passing the asymptotics in \cite{BW22,BW*,BW**} of $R_n$ for a critical branching random walk on $\Z^d$, $d \ge 3$, conditional to have total population size $n$. 

In view of these works, a natural question, which we fully resolve here, is to determine 
the almost sure fluctuations of $n \mapsto \crn$ for the \abbr{srw}, in the form of some \abbr{LIL} 
(possibly after centering $\crn$ when $d \ge 4$). Specifically, using hereafter
$\log_k a=\log (\log_{k-1} a)$ for $k\ge 2$, with $\log_1 a$ for the usual logarithm,
here is our first main result,
about the \abbr{srw} in $\mathbb{Z}^3$.
\begin{thm}\label{m1} For $d=3$, almost surely,
\begin{align}\label{eq:lil-3d}
\limsup_{n\to \infty}\frac{ \crn}{h_3(n)} =1 \,, \qquad
\liminf_{n\to \infty}\frac{ \crn}{\hat{h}_3(n)} =1, \,
\end{align}
where 
\begin{align}\label{h3:def}
h_3(n):=\frac{\sqrt{6}\pi}{9}  (\log_3 n)^{-1} \sqrt{n \log_2 n}, \quad
\hat{h}_3(n):= \frac{\sqrt{6} \pi^2}{9} \sqrt{ n (\log_2 n)^{-1}} .
\end{align}
\end{thm}
Utilizing \eqref{q0}, we also get from Theorem \ref{m1} the following 
consequence about the Brownian capacity of the $3$-dimensional (Brownian) sample path.
\begin{cor}\label{BMm1} For $d=3$, almost surely,
\begin{align*}
\limsup_{n\to \infty}\frac{\ca_B (B[0,n])}{3 \sqrt{3}h_3(n)}=1 \,, \qquad
\liminf_{n\to \infty}\frac{\ca_B (B[0,n])}{3 \sqrt{3}\hat{h}_3(n)}=1 \,.
\end{align*}
\end{cor}
\red{
\begin{rem*}
From the variational characterization of $\ca_B (B[0,n])$, with $\mu(\cdot)$ 
the push-forward of the uniform law on $[0,n]$ by the Brownian path $t \mapsto B_t$, we get that 
\[
\frac{\pi n^2}{\ca_B(B[0,n])} \le  \int_0^n dt \int_0^t |B_t-B_s|^{-1} ds =: \eta([0,n]^2_<) \,.
\]
It thus follows from the $\limsup$-\abbr{LIL} of \cite[Thm. 1.2]{CR} for $\eta([0,n]^2_<)$, that almost surely,
\[
\liminf_{n\to \infty}\frac{\ca_B (B[0,n])}{3 \sqrt{3}\hat{h}_3(n)} \ge \frac{3 \sqrt{3}}{8 \sqrt{2} \pi \, \rho} \,,
\]
where $\rho$ is given by \cite[formula (1.15) for $d=3$, $\sigma=1$, $\psi(\lambda)=\lambda^2/2$]{CR}.
\end{rem*}}
We next provide the \abbr{LIL} for the centered capacity 
$\ocrn$ of the range, first in case of the \abbr{srw} on $\Z^4$ and then for
\abbr{srw} on $\mathbb{Z}^d$, $d \ge 5$.
\begin{thm}\label{m2}
For $d=4$,  almost surely,
\begin{align}\label{eq:lil-centered}
 \limsup_{n\to \infty}  \frac{ \ocrn}{h_d(n)}=1, 
 \quad
  \liminf_{n\to \infty}  \frac{\ocrn}{\hat{h}_d(n)}=-1, 
\end{align}
where for some non-random $0<c_\star<\infty$,
\begin{align}\label{h4:def}
h_4(n):= \frac{\pi^2}{8}\frac{ n \log_3 n}{ (\log n)^2} \,,
\qquad
\hat{h}_4(n):= c_\star \frac{ n \log_2 n}{ (\log n)^2} \,.
\end{align}
\end{thm}
\begin{thm}\label{m3}
For any $d \ge 5$, the \abbr{LIL}-s \eqref{eq:lil-centered} hold almost surely, now with 
\begin{align}\label{hd:def}
h_d(n)=\hat{h}_d(n) := \sigma_d \sqrt{2 n (1 + 1_{\{d=5\}}\log n ) \log_2 n} \,, \quad d \ge 5, 
\end{align}
where the non-random, finite $\sigma_d^2>0$ are given by the leading asymptotic of
${\rm var}(\crn)$ 
 (c.f. \cite[Thm. A]{S20} for $\sigma_5$ and \cite[Thm. 1.1]{As3} for $\sigma_d$, $d \ge 6$). 
\end{thm}
\begin{rem} Our proof of Theorem \ref{m3} via Skorokhod embedding, also yields Strassen's \abbr{LIL} 
for the a.s. set of limit points in $C([0,1])$ of the functions 
$\{ t \mapsto h_d(n)^{-1} \overline{R}_{tn} \}$, for any $d \ge 5$. 
\end{rem} 
\begin{rem} 
\red{The moment generating function of the 
%(non-Gaussian)
limit in law of $- ((\log n)^2/n) \overline{R}_n$ for the \abbr{srw} on $\Z^4$, blows-up 
at a finite, positive $\lambda$.  The value of $\lambda$ is identified in \cite[Theorem 1.3]{CR}.
In}
% formally 
%defined as 
%\begin{align*}
%\gamma_G([0,1]^2) = \int_0^1 \int_0^1 G_\beta (\beta_s, \beta_t) \text{d}s \text{d}t
%- E\Big[ \int_0^1 \int_0^1 G_\beta (\beta_s, \beta_t) \text{d}s \text{d}t \Big] 
%\end{align*}
%\blue{(see \eqref{dfn:ga-G} for their precise definition). It is further shown in \cite[(1.3)]{As5} that} 
%(\abbr{mgf}) 
%\begin{equation}\label{def:mgf}
%M_G(\lambda) := \,,
%\end{equation}
%is infinite at all large enough $\lambda$. 
%\blue{Adapting the proof of \cite[Thm. 1]{LG94} we}
%establish en route to Theorem \ref{m2} the following complementary
%exponential tail result for $\gamma_G([0,1]^2)$, which is potentially of independent interest
Lemma \ref{bkk} we establish the uniform in $n$ boundedness 
of the moment generating function of $\blue{-}((\log n)^2/n) \overline{R}_n$ for a small enough argument.
%\begin{prop}\label{prop:exptail}
%For $M_G(\cdot)$ of \eqref{def:mgf}, the value of  
%\[
%\lambda_o = \sup\{ \lambda : M_G(\lambda) < \infty \}\,
%\]
%is strictly positive and finite.
%\end{prop}
\end{rem}
%\begin{rem} 
%Second author with another author improved Theorem \ref{m2}. We 
%\end{rem} 

We note that  $\crn \approx n E[\hat{P}^{S_{n/2}}(\hat{\tau}_{\R_n} =\infty)]$ at any fixed $n \gg 1$
and $d \ge 3$, where $\hat{P}$ and $\hat{\tau}_A$ denote the law and the first hitting time 
by an i.i.d. copy 
%$(\hat{S}_m)$ 
of the \abbr{srw}.
% on $\mathbb{Z}^d$.
 Similarly, the volume 
of $\R_n$ in any dimension ($d \ge 1$), is approximately $n P(\tau_0 >n)$. It has been 
observed before, see for example \cite[Sec. 6]{As3}, that the typical order of growth of 
$E[\hat{P}^{S_{n/2}}(\hat{\tau}_{\R_n} =\infty )]$ at any $d \ge 3$, matches that 
of $P(\tau_0 > n)$ at $d'=d-2$, yielding the same order of 
growth in $n$ for $\crn$ at $d \ge 3$ and for the volume of $\R_n$ at $d'=d-2$. 
In Theorems \ref{m2} and \ref{m3}, our \abbr{LIL} for $d \ge  4$ adheres to such a match
with the scale for the \abbr{LIL} of the volume of $\R_n$ at $d'=d-2$ 
(see \cite{BCR, BK02, NW75} for the latter \abbr{LIL} at any $d' \ge 2$, as well as
the limit distribution results for the volume of the Wiener sausage
at $d' \ge 2$, and the corresponding \abbr{LIL} at $d' \ge 3$, in \cite{LG88} and 
\cite{CH07,WG11}, respectively). In contrast, 
this relation \emph{breaks down at the $\limsup$ \abbr{LIL} for $d=3$,} with the 
appearance of the novel factor $(\log_3 n)^{-1}$ in Theorem \ref{m1}. 
Nevertheless, even at $d=3$ the relevant deviations of $\crn$ are due to  
those in the diameter of $\R_n$, except that the upper tails for these two variables differ
in their growth rates. Specifically, our proofs in Sections \ref{ioo1} and \ref{ioo2} yield the 
following (sharper) result. 
\begin{prop}\label{m1+}
Let $M_n:=\max_{1\le i\le n}|S_i|$. For \abbr{srw} of $\mathbb{Z}^3$ and any $\epsilon>0$, 
\begin{align*}
P( \{\crn \ge (1-\epsilon)h_3(n) \}\cap \{M_n \ge (1-\epsilon)\psi(n) \} \quad \text{i.o.}) &=1, \\
P( \{\crn \le (1+\epsilon)\hat{h}_3(n) \}\cap \{M_n \le (1+\epsilon)\hat{\psi}(n) \}\quad \text{i.o.}) &=1,
\end{align*}
where $\psi(n):=\sqrt{(2/3) n \log_2 n}$ and $\hat{\psi}(n):=\pi\sqrt{(1/6) n (\log_2 n)^{-1}}$. 
\end{prop}
We note in passing that Proposition \ref{m1+} is a rotation-invariant result,  and in particular
it applies also under any (fixed) rotation of the \abbr{srw} lattice $\mathbb{Z}^3$. Further, 
Proposition \ref{m1+} implies that almost surely, the $\limsup$ (resp. $\liminf$) 
of $\crn$ are essentially attained simultaneously with those for $M_n$, since for $d=3$, 
almost surely, 
\begin{align}\label{eq:lil-Mn}
\limsup_{n\to \infty} \frac{M_n}{\psi(n)}=1\,, \qquad 
\liminf_{n\to \infty} \frac{M_n}{\hat{\psi}(n)}=1.
\end{align}
Indeed, by the invariance principle it suffices for proving \eqref{eq:lil-Mn} to show 
the equivalent a.s. statement 
for $3$-dimensional Bessel process, and the latter follows by mimicking the proof 
of Chung's one-dimensional \abbr{LIL} (see \cite{Ch48}), starting with the estimate
(\ref{chung*}). We note in passing that while the scaling $\psi(n)$ of the 
upper fluctuation of $M_n$ is the same as that for a single coordinate of our \abbr{srw},
this is not true about the scaling $\hat{\psi}(n)$, which is about the tail probability 
of confining the walk to stay within an Euclidean ball in $\RR^3$. 

In a followup work,  \cite[Corollary 1.2]{AO} determines
the value of $c_\star$ of Theorem \ref{m2} 
%$c_\star = \frac{1}{2} (\pi \widetilde{\kappa}(4,2))^4$,  
in terms of the best constant
% $\widetilde{\kappa}(4,2)$ of \cite[(6)]{FY}.
in a generalized Gagliardo-Nirenberg inequality.  In contrast, 
the following analog of Proposition \ref{m1+} 
in case $d \ge 4$,  is still open. 
\begin{open}
Consider the \abbr{srw} $S_i=(S_i^1,\ldots,S_i^d) \in \mathbb{Z}^d$, $d\ge4$.
For $d'=d-2$, let $\hat{S}^{d'}_i=(S^1_i,\ldots ,S^{d'}_i,0,0)$ and
$V_{d'} (n) =|\{\hat{S}^{d'}_1, \ldots, \hat{S}^{d'}_n\}|$. Pick 
non-random $\psi_{d'}(n)$ such that a.s.
\begin{align*}
\limsup_{n\to \infty} \frac{\overline{V}_{d'}(n)}{\psi_{d'}(n)}=1 \,.
\end{align*}
We then conjecture that for $h_d(n)$ of Theorems \ref{m2}-\ref{m3} and any $\epsilon>0$, 
\begin{align*}
P( \{\ocrn \ge (1-\epsilon)h_d(n) \}\cap \{\overline{V}_{d'}(n) \ge (1-\epsilon)\psi_{d'}(n) \} \quad \text{i.o.})
=1.
\end{align*}
\end{open}

While we consider throughout only the discrete time \abbr{srw} whose
increments are the $2d$ neighbors of the origin in $\Z^d$, due to sharp 
concentration of Poisson variables, all our results apply also for the continuous time 
\abbr{srw} with i.i.d. Exponential($1$) clocks and up to the scaling $n \mapsto (1-\rho) n$,
also to the $\rho$-lazy discrete time \abbr{srw}. By definition of $R_n$, our results apply 
to any random walk on a group with a finite symmetric set of generators, whose words 
are isomorphic to those of the \abbr{srw} (e.g. an invariance of our results under any non-random, 
invertible affine transformation of the walk). We note in passing the 
recent work \cite{MSS21} on the strong law for any symmetric random walk on a group of 
growth index $d$ and the corresponding \abbr{clt} in case $d \ge 6$, suggesting 
the possibility of a future extension of our \abbr{lil}-s in this context.

Beyond the intrinsic interest in $\crn$, 
%due to the connection between hitting probabilities and capacity, the 
its asymptotic is also relevant for the study of intersections between two independent random walks 
(e.g., see \cite[Ch. 3]{LA91}). Similarly, \cite{As2*,As2**} utilize bounds on 
$\crn$ to gain insights about the so called Swiss cheese picture for $d=3$. Further, 
to understand Sznitman's \cite{SZ10} random interlacement model, one may use
moment estimates for the capacity of the union of ranges  (c.f. \cite{Ch17} and 
the references therein). Finally, the capacity equals the summation of all 
entries of the inverse of the (positive definite) Green's function matrix (see \eqref{gr1}),
a point of view which \cite{OK} uses, for $d=2$,
%some property concerning the summation of all entries of the inverse of positive definite matrix 
to estimate the geometry of late points of the walk.

As for the organization of this paper, we prove Theorem \ref{m1} in Section \ref{sec:m1}, 
relying on \blue{certain} relations between the capacity and Green's function which we explore in 
Section \ref{sec-Green}. Our proof of Theorem \ref{m1} 
further indicates that the $\limsup$-\abbr{lil} is due to 
exceptional time where $\R_n$ has a cylinder-like shape, with one dimension 
being about $\psi(n)$ while the other two are $O(\psi(n)/(\log_2 n))$
(see Lemma \ref{cy1} and Section \ref{subsec:limsup-3u}). In contrast, the $\liminf$-\abbr{lil}
seems to be due to times where the shape of $\R_n$ is close enough to a ball of radius 
$\hat{\psi}(n)$ to approximately match the capacity of such a ball (see \eqref{q2b} and \eqref{q3}).

Sections \ref{sec:m2-d4} and \ref{sec:m2-dge5} are devoted to the proofs of 
Theorems \ref{m2} and \ref{m3}, respectively. Our proofs rely on 
the decomposition \eqref{eq:decomp4} of $R_{n_k}$ as the sum of $k$ independent 
variables $\{U_j\}$ which are the capacities of the walk restricted to the $k$ 
parts of a partition of $[1,n_k]$, minus some random $\Delta_{n_k,k} \ge 0$ (which ties
all these parts together). For any $d \ge 5$ the effect of $\Delta_{n_k,k}$ on the \abbr{lil} is 
negligible, so upon coupling $R_{n_k}$ with a one-dimensional Brownian motion, 
we immediately get the \abbr{lil} for the former out of the standard \abbr{lil} for the latter.
As seen in Section \ref{sec:m2-d4}, the situation is \emph{way more delicate for $d=4$,}
where $E \Delta_{n,k} \approx h_4(n)$ dominates for a suitable slowly growing 
$k=k_n$ the fluctuations of the i.i.d. $\{U_j\}$. 
The $\limsup$-\abbr{lil} is then due to the exceptional (random) sequence $\{n_k\}$ where 
$\Delta_{n_k,k} = o(h_4(n_k))$, while the $\liminf$-\abbr{lil} is due to the 
exceptional $\{n_k\}$ for which $\Delta_{n_k,k} \approx \hat{h}_4(n_k) \gg E \Delta_{n_k,k}$. 
Indeed, whereas Theorem \ref{m2} is proved via the framework developed in 
\cite[Section 4]{BK02} for the \abbr{LIL} for the volume of $\R_n$ 
in the planar case ($d'=2$), special care is needed here in order to
establish tight control on the moderate deviations of $\Delta_{n,k}$ and $U_j$ in case $d=4$
(c.f. Lemmas \ref{extail:lem}, \ref{moment*}, \ref{bkk} and \ref{sd1}, which
may be of independent interest).

\subsection*{Acknowledgments}
This research was supported in part by NSF grant DMS-1954337 (A.D.), 
by JSPS KAKENHI grant-in-aid for early career scientists JP20K14329 (I.O.) 
and by a JSPS overseas research fellowship (I.O.).
\blue{We also thank the anonymous referees for their detailed feedback which greatly improved 
the exposition of this work.}

\section{Capacity geometry and Green's function}\label{sec-Green}

The following asymptotic for the $3$-dimensional capacity of cylinder-like domains
(which we prove at the end of this section), 
is behind the factor $(\log_3 n)^{-1}$ in the $\limsup$-\abbr{lil} of \eqref{eq:lil-3d}. 
\begin{lem}\label{cy1}
%We define the following cylinder domain. 
For $m \ge 1$ and $r \ge k \in \N$, let
\begin{align*}
{\C}_m(\ell,r):=& (\ell \Z)^3 \cap \{(x_1,x_2,x_3) : x_1^2+x_2^2 \le r^2, 1 \le x_3 \le m \}.
\end{align*}
Fix $b<2/3$, $r_m=o(m)$, $r_m \uparrow \infty$. If 
$\C_m(1,r_m) \supseteq \C_m \supseteq \C_m(\ell,r_m)$ for some $\ell \le r_m^{b}$, then
% one has that 
\begin{align}\label{eq:cyl-cap}
\lim_{m\to \infty}\frac{ \ca({\C}_m)}{m (\log (m/r_m))^{-1}}
=\frac{\pi}{3}.
\end{align}
\end{lem}
\begin{rem}\label{for-ubd-sec3} In the sequel we prove a stronger result, namely that the upper bound
in \eqref{eq:cyl-cap} holds as soon as $\C_m$ is contained in a union $\C_m^\star (r_m)$ of at most 
$m/r_m$ balls $\BB(z_i,r_m)$ of radius $r_m$ in $\Z^3$, of centers such that $|z_{i+1}-z_i| \le r_m$ 
for $1 \le i < m/r_m$ (where $\C_m(1,r)$ is merely one possible choice for 
$\C_m^\star(r)$).
\end{rem}

Indeed, in Section \ref{sec:m1} we will see that $\limsup$ of $\crn$ is roughly attained on the 
event $\{S^i_n \ge \psi(n)\}$ for $\psi(n)$ of Proposition \ref{m1+}, 
with $\R_n$ then having approximately the shape of such ${\C}_m$ for $m=\psi(n)$, 
% $\ell_m= \lceil n^{1/6} \rceil$
and $r_m=c m/\log_2 n$, hence 
% Since eventually $\ell_m \le \sqrt{r_m}$, we get 
%where $\lceil a \rceil$ means the smallest integer which is larger than $a$.  
from Lemma \ref{cy1} we find that  
\begin{align*}
\crn \approx
\ca( {\C}_{m} )  \approx \frac{\pi}{3} m (\log (m/r_m))^{-1} \approx \frac{\pi}{3}
 \psi(n) (\log_3 n)^{-1} \,,
\end{align*}
which is precisely $h_3(n)$ of Theorem \ref{m1}.

We proceed with two lemmas relating the capacity of \abbr{srw} with its Green's function,  
\begin{align*}
G(x,y)=\sum_{i=0}^\infty P^x(S_i=y) \,.
\end{align*} 
% (which we shall use throughout this paper). 
To this end, partition $\Omega$ by the last time the walk visits $X=\{x_1 \ne x_2 \cdots  \ne x_j\}$, 
to get that
\begin{align}\label{gr1}
1=\sum_{\ell=1}^j G(x_i,x_\ell) P^{x_\ell}(\tau_X=\infty), \qquad \forall  1 \le i \le j \,.
\end{align} 
\begin{lem}\label{mat1}
For any set $X=\{x_1,\ldots ,x_j\}$, with $\{x_j\}$ not necessarily \blue{distinct}, 
\begin{align}\label{eq:grf}
\frac{j}{\max_{1\le \ell \le j} \{\sum_{i=1}^j G(x_i,x_\ell)\}}
\le \ca (X)
\le \frac{j}{\min_{1\le \ell \le j} \{\sum_{i=1}^j G(x_i,x_\ell)\}}.
\end{align} 
\end{lem} 
\begin{proof} The set $X$ of size $|X|=k \le j$ consists \abbr{wlog} of \blue{distinct} points 
$\hat{X}=\{\hat{x}_1 \ne \hat{x}_2 \cdots \ne \hat{x}_k\}$, where $\hat{x}_v$ 
appears $m_v \ge 1$ times in $X$ (and $\sum_{v \le k} m_v = j$). 
\blue{Though \eqref{eq:grf} follows from the characterization of $\ca(\hat{X})$ as in \cite[Lemma 2.2(i)]{JO73}, we proceed instead with a direct, short and elementary proof of these bounds. Specifically, setting}
$v(\ell)$ for the index of $x_\ell$ in $\hat{X}$ and 
$q_v := (m_v)^{-1} P^{\hat{x}_v}(\tau_{\hat{X}}=\infty)$, we see that 
\begin{equation}\label{eq:cap-id}
\ca(X)=\ca(\hat{X}) = \sum_{v=1}^k P^{\hat{x}_v} (\tau_{\hat{X}}=\infty) = \sum_{\ell=1}^j q_{v(\ell)} 
\end{equation}
and moreover, summing \eqref{gr1} over $i \le j$, we get that
\begin{align}\label{eq:jhat}
j = \sum_{i=1}^j \sum_{v=1}^k G(x_i,\hat{x}_v) P^{\hat{x}_v}(\tau_{\hat{X}}=\infty) = 
\sum_{\ell=1}^j q_{v(\ell)} \sum_{i=1}^j G(x_i,x_\ell) \,.
\end{align} 
The bounds of \eqref{eq:grf} are an immediate consequence of \eqref{eq:cap-id} and \eqref{eq:jhat}.
\end{proof}

\begin{lem}\label{mat2}
For $Z_1=\{x_1,\ldots ,x_{j_1}\}$, $Z_2=\{x_{j_1+1},\ldots ,x_{j_1+j_2}\}$ with $\{x_i\}$  
not necessarily \blue{distinct},
\begin{align*}
\ca (Z_1 \cup Z_2) \le\ca(Z_2) + 
\frac{j_1+j_2}{\min_{x\in Z_1 \setminus Z_2} \{ \sum_{i=1}^{j_1+j_2} G(x_i,x) \} } . 
\end{align*} 
\end{lem} 

\begin{proof} Since $\tau_{Z_1 \cup Z_2} \le \tau_{ Z_2}$, it follows that 
\begin{align}
\ca (Z_1 \cup Z_2) \le \ca (Z_2) + \sum_{x\in Z_1 \setminus Z_2} P^x (\tau_{Z_1 \cup Z_2}=\infty) \,.
\label{for1}
\end{align} 
For $\hat{X}$ enumerating the \emph{\blue{distinct}} points in $Z_1 \cup Z_2$, 
$v(\ell)$, $q_v$ as in 
%the proof of 
Lemma \ref{mat1}, we have that 
 \begin{align*}
 \sum_{x\in Z_1 \setminus Z_2} P^x (\tau_{Z_1 \cup Z_2}=\infty) &= \sum_{\ell=1}^{j_1+j_2} q_{v(\ell)} 
 1_{\{\hat{x}_{v(\ell)} \in Z_1 \setminus Z_2\}} \,, \\
j_1+j_2 &= \sum_{\ell=1}^{j_1+j_2} q_{v(\ell)} \sum_{i=1}^{j_1+j_2} G(x_i,x_\ell) \,.
\end{align*}
Combining these identities with \eqref{for1} yields the stated upper bound.
\end{proof}
\begin{rem}\label{green*} In particular, applying  Lemma \ref{mat2} for  
\[
Z_1=\bigcup_{i \in (j, J-j]} \hat{Z}_i, \qquad Z_2=\bigcup_{i \in [1,j] \cup (J-j, J]} \hat{Z}_i \,,
\]
we have that for any $\hat{Z}_i \subset \Z^d$, $2 j<J$,  
\begin{align*}
\ca (Z_1 \cup Z_2) \le \ca(Z_2) + 
 \frac{\sum_{i=1}^J |\hat{Z}_i|}{\min_{x\in Z_1}\sum_{i=1}^J\sum_{y \in \hat{Z}_i}  G(x,y)} . 
\end{align*} 
\end{rem} 

\begin{proof}[Proof of Lemma \ref{cy1}] 
By the monotonicity of $A \mapsto \ca(A)$, it 
suffices to provide a uniform in $\ell \le r_m^b$ lower bound on $\ca(\C_m(\ell,r_m))$ 
and a matching upper bound on $\ca(C_m^\star(r_m))$, \blue{valid for 
any union $\C_m^\star (r_m)$ of at most 
$m/r_m$ balls $\BB(z_i,r_m)$ of radius $r_m$ in $\Z^3$ and 
centers such that $|z_{i+1}-z_i| \le r_m$ for $1 \le i < m/r_m$.} With 
$|{\C}_m(\ell,r_m)| = (1+o(1))\pi m r_m^2 \ell^{-3}$,  we get such a lower bound 
from Lemma \ref{mat1}, upon showing that for \abbr{srw} on $\Z^3$, 
%uniformly  
\begin{equation}\label{eq:ub-GCn}
\sum_{y \in {\C}_m(\ell,r_m)} G(x,y) \le 3 (1+o(1))  r_m^2 \ell^{-3} \log (m/r_m) \,, \qquad \forall \ell \le r_m^b, \;\; \forall 
x \in \C_m(\ell,r_m) \,.
\end{equation}
\blue{Fixing $b < 2/3$, since} the right-side of \eqref{eq:ub-GCn} diverges in $m$, 
\blue{uniformly in} $\ell \le r_m^b$, \blue{we can ignore any bounded contribution 
to its left-side. In particular,}
with $m/r_m \uparrow \infty$ and $G(x,y)$ bounded, it suffices to sum in \eqref{eq:ub-GCn} 
\blue{only over $y \in \C_m(\ell,r_m)$} with 
$|x-y| \ge r_m^{2/3} \gg \ell$ and use the asymptotics
\begin{align}\label{grf}
G(x,y)=\frac{3+o(1)}{2\pi} |x-y|^{-1}
\end{align} 
(for example, see \cite[Thm. 1.5.4]{LA91}). Setting $u_m=m$, we have
for any $v_m \uparrow \infty$ and $r \in [r_m^{2/3}, v_m r_m]$, at most $C r^2 \ell^{-3}$
points $y \in \C_m(\ell,r_m)$ with $|x-y| \in [r,r+1]$, while for each $r \in [v_m r_m, u_m]$ 
there are at most $(2 \pi +o(1)) r_m^2 \ell^{-3}$ such points in $\C_m(\ell,r_m)$. Thus, 
taking $v_m^2 \ll \log (m/r_m)$ yields
\begin{align}
\sum_{y \in {\C}_m\red{(\ell,r_m)}} G (x,y) & \le \frac{3+o(1)}{2\pi \ell^3} \Big[
 C \int_{r_m^{2/3}}^{v_m r_m} r dr + 2 \pi r_m^2 \int_{v_m r_m}^{u_m} r^{-1} d r \Big ] \nonumber \\
& =  (3+o(1)) r_m^2 \ell^{-3} \log (u_m/(r_m v_m)) \,, 
\label{eq:basic-bd}
\end{align}
from which \eqref{eq:ub-GCn} immediately follows. Turning to upper bound on $\ca(\C_m^\star(r_m))$, 
take now $u_m := (m/r_m)^{1-\epsilon_m}$ and $v_m := (m/r_m)^{\epsilon_m} \uparrow \infty$
for some $\epsilon_m \to 0$, splitting 
$\C_m^\star(r_m)$ to $Q_1 \cup Q_2$, where 
\begin{align*}
Q_1:=\bigcup_{i \in (u_m,(m/r_m)- u_m)} \BB(z_i,r_m) , \qquad 
Q_2 := \bigcup_{i \notin (u_m,(m/r_m)-u_m)} \BB(z_i,r_m) .
\end{align*}
Note that $\C_m^\star(r_m)$ has at most $(4\pi/3 + o(1)) r_m^2 m$, possibly overlapping, points. 
Thus, combining Lemma \ref{mat2} with the upper bound of Lemma \ref{mat1}, we get  
the upper bound of \eqref{eq:cyl-cap}, once we show that for some $\delta_m \to 0$,  
\begin{align}\label{eq:lb-GCn}
\sum_{y \in {\C}_m^\star(r_m)} G(x,y) & \ge (4+\delta_m)  r_m^2 \log (m/r_m) \,, \qquad \forall x \in Q_1 \,,\\
\label{eq:lb-Q2}
\sum_{y \in Q_2} G(x,y) &\ge \frac{|Q_2|}{m \, \delta_m} \log (m/r_m) \,,
\qquad \qquad\;  \forall x \in Q_2 \,.
\end{align}
Fixing $x \in \BB(z_i,r_m) \subset Q_1$, consider only the contribution to the \abbr{lhs} 
of \eqref{eq:lb-GCn} from all points $y \in B(z_j,r_m)$ with $|j-i| \in [v_m,u_m]$.
For such a pair $|y-x| \le (|j-i|+3) r_m$, hence by \eqref{grf},
\[
G(x,y) \ge \frac{3+o(1)}{2\pi} r_m^{-1} |j-i|^{-1} \,,
\]
resulting with 
\[
\sum_{y \in {\C}_m^\star(r_m)} G(x,y) \ge 2 \frac{3+o(1)}{2\pi} \frac{|\BB(0,r_m)|}{r_m} 
\sum_{j=v_m}^{u_m} j^{-1} = (4+o(1)) r_m^2 \log (u_m/v_m) \,,
\]
which for our choices of $u_m$ and $v_m$ is as stated in \eqref{eq:lb-GCn} (for some $\delta_m \to 0$,
uniformly over $x \in Q_1$).
Further, $Q_2$ consists of two sets with an equal number of elements, each of
diameter at most $(1+o(1)) u_m r_m$. Thus, we get by \eqref{grf} that for some $c>0$, 
\begin{align*}
 \sum_{y\in  Q_2}  G(x,y) \ge  & c \, |Q_2| \, (u_m r_m)^{-1} \,, \qquad \forall x \in Q_2 \,,
\end{align*}
and \eqref{eq:lb-Q2} follows upon choosing $\epsilon_m \to 0$ slow enough so
that $(m/r_m)^{\epsilon_m} \gg \log (m/r_m)$.
\end{proof}

%%%%%%%%%%%%%%%%%%%%%%%%%%%%%%%%%%%%%%%%%%%%%%%%%%%%%%%%%%%%%%%%%%

%%%%%%%%%%%%%%%%%%%%%%%%%%%%%%%%%%%%%%%%%%%%%%%%%%%%%%%%%%%%%%%%%%

\section{\abbr{LIL} for \abbr{srw} on $\Z^3$: Proof of  Theorem \ref{m1}}\label{sec:m1}

To ease the presentation we omit hereafter the integer-part symbol $\lceil \cdot \rceil$ and
divide the section to four parts, establishing the lower and then upper bounds, first for the 
$\limsup$-\abbr{LIL} of \eqref{eq:lil-3d} and then for the $\liminf$-\abbr{LIL} of \eqref{eq:lil-3d}.
\blue{Our proofs for the $\limsup$-\abbr{LIL} are \emph{discrete} in nature, relying on Remark 
\ref{for-ubd-sec3}, Lemma \ref{mat1} and the direct evaluation of certain sharp tail probabilities for \abbr{srw}. In contrast, we prove the $\liminf$-\abbr{LIL} by first replacing $\crn$ by the corresponding
quantity about the Brownian capacity of the range of the 3D Brownian motion.}

%%%%%%%%%%%%%%%%%%%%%%%%%%%%%%%%%%%%%%%%%%%%%%%%%%%%%%%%%%%%%%%%%%
\subsection{The lower bound in the limsup-\abbr{LIL}}\label{ioo1}
Recall 
\begin{align*}
\psi(t)
:=\sqrt{(2/3) t \log_2 t} \,, \qquad h_3(t):= \frac{\pi}{3} \psi(t) (\log_3 t)^{-1} \,,
\end{align*}
of Proposition \ref{m1+} and Theorem \ref{m1}, respectively, and for the \abbr{srw}  $(S_m)$ on $\Z^3$ 
\blue{set} 
\begin{align*}
A_t := \{\, S^1_t  \ge \psi(t) \, \} \,, \qquad
V_{I} :=1_{A_t} \sum_{\ell \in I}  G(0,S_\ell) \,, \qquad 
\; I \subseteq [0,t] \cap \Z \,.
\end{align*} 
\blue{The lower bound in our $\limsup$-\abbr{LIL} is attained via partial sums on disjoint intervals 
$I_n$ of suitably growing length $t_n$ and controllable Green function values (utilizing the 
\abbr{lhs} of \eqref{eq:grf}). The key for this is} our next lemma
\blue{(whose proof is deferred to the end of this sub-section), 
showing that $A_{t_n}$ yields the appropriate bound on the sum $V_{[0,t_n]}$ 
of Green function values.}
\begin{lem}\label{gn2+gn3}
Fixing $\delta \in (0,1)$, for $\gamma_t := t (\log_2 t)^{-1} (\log_3 t)^{3/2}$
and some $\zeta_t \to 0$ when $t \to \infty$,
\begin{align}
\label{eq:gn3}
P\Big(  V_{[0,\gamma_t]}
\ge \frac{\delta t}{h_3(t)} \Big)  & \le \zeta_t \, P(A_t)\,, \\
\label{eq:gn2}
P\Big( V_{(\gamma_t,t]} \ge \frac{(1+ 2 \delta) t}{2 h_3(t)}  \Big)  &\le \zeta_t \, P(A_t)\,.
\end{align} 
\end{lem}
\blue{Indeed, from \eqref{eq:gn3}-\eqref{eq:gn2} we see that for} 
% \abbr{srw} $(S_m)$ on $\Z^3$ and 
any $\delta>0$, there exists $\zeta_t \to 0$ as $t \to \infty$, such that 
\begin{equation}\label{eq:34+35}
P \,  (  A^\star_t  ) \le 2 \zeta_t P(A_t) \,, \qquad 
A^\star_t:= \Big\{ \blue{V_{[0,t]}} \ge \frac{(1+ 4 \delta)}{2} \frac{t}{h_3(t)} \Big\} \,.
\end{equation}
\blue{Proceeding to deduce the lower bound in the limsup-\abbr{LIL} out of \eqref{eq:34+35}, given} 
$\epsilon>0$, we choose $q>1$ large and $\delta>0$ small so for $t_n:=q^n-q^{n-1}$ 
and all large $n$, 
\begin{equation}\label{eq:def-q-delta}
\frac{1-\delta}{1+ 4 \delta} h_3(t_n) \ge (1-\epsilon) h_3(\blue{q^n}) \,.
\end{equation}
We then partition $\Z_+$ to disjoint intervals $I_n:=(q^{n-1},q^n] \cap \Z$ of length \blue{$t_n$} and set the events 
\begin{align*}
\hat{A}_n  :=\{  S^1_{q^n}-S^1_{q^{n-1}} \ge \psi(t_n) \}, \qquad 
\hat{H}_n(i) := \Big\{ \hat{V}_n(i) \ge (1+4 \delta) \frac{t_n}{h_3(t_n)} \Big\} \,, 
\end{align*}
where
\[
\hat{V}_n(i) := \sum_{\ell \in I_n}  G(S_i,S_\ell) \,.
\]
Setting 
\[
H_t(i) := \Big\{ \sum_{\ell=1}^t G(S_i,S_\ell) \ge (1+4 \delta) \frac{t}{h_3(t)} \Big\} \,, \qquad i \in [1,t] \,,
\]
note that by the stationarity of the \abbr{srw} increments,  
$(\hat{A}_n,\hat{H}_n(q^{n-1}+i)) \stackrel{d}{=} (A_{t_n},H_{t_n}(i))$. Further, the event $A_{t}$ 
is invariant to any permutation of the \abbr{srw} increments $\{X_j, j \le t\}$,  whereas given
$A_t$ the event $H_t(i)$ depends only on $\{S_\ell-S_i, \ell \in [1,t]\}$. Further, the 
permuted increments $\hat{X}_j = X_{(i+j)\!\!\!\mod\!\!(t)}$ result with $\hat{S}_\ell=S_{\ell+i}-S_i$
for all $\ell \in [1,t-i]$, whereas 
%the permuted 
$\hat{X}_j = X_{(i+1-j)\!\!\!\mod\!\!(t)}$ result
with $\hat{S}_\ell=S_i-S_{i-\ell}$ for all $\ell \in [1,i]$. Since $G(x,y)=G(-x,-y) \ge 0$, it thus follows that
conditional on $A_t$ the random sum in each of the events $H_t(i)$ is stochastically dominated by twice 
the random sum in the event $A^\star_t$ \blue{of \eqref{eq:34+35}.} 
Consequently, \eqref{eq:34+35} yields that 
\begin{equation}\label{eq:lem31}
\max_{i \in I_n} \{ P \,  ( \, \hat{H}_n  (i) \, \big| \, \hat{A}_n ) \} \le \blue{4} \zeta_{t_n} \to 0 \,.
\end{equation} 
Next, consider the independent events $G_n  :=\big\{ |\Lambda_n| \ge (1-\delta) t_n \big\}$, 
where $\Lambda_n$ is the subset of all those $i \in I_n$ for which $\hat{H}_n(i)$ does not hold.
From Markov's inequality and \eqref{eq:lem31} it follows that 
\begin{equation}\label{eq:lem31-app}
P(G_n^c \,|\, \hat{A}_n) = P(|I_n|-|\Lambda_n| \ge \delta t_n \, | \, \hat{A}_n) 
\le \frac{1}{\delta t_n} \sum_{i \in I_n} 
P \,  \big( \, \hat{H}_n(i) \, \big| \, \hat{A}_n \big) \le 
\frac{\blue{4} \zeta_{t_n}}{\delta} \to 0 \,.
\end{equation}
Setting $x(t):= \sqrt{2 \log_2 t}$, note that  
$(S^1_m)$ is the partial sum of $\{-1,0,1\}$-valued, zero-mean, i.i.d. variables of variance $1/3$. 
By the asymptotic normality of the moderate deviations for such partial sums (see 
\cite[Thm. VIII.2.1]{Petrov}), we have that \blue{for some $o_{\sf t}(1) \to 0$ as ${\sf t} \to \infty$, 
uniformly over ${\sf x}/{\sf t}^{1/6}$ small,
\begin{equation}\label{eq:lclt}
P ( ({\sf t}/3)^{-1/2} \, S_{\sf t}^1 \ge {\sf x} \Big) = (1+o_{\sf t}(1)) \overline{\Phi} ({\sf x}) \,, \qquad
\overline{\Phi}({\sf x}) := \int_{\sf x}^\infty \frac{e^{-u^2/2}}{\sqrt{2\pi}} du \,.
\end{equation}
In particular,}
\begin{align*}
P(\hat{A}_n) = P(A_{t_n})  &= P \Big( (t_n/3)^{-1/2} \, S_{t_n}^1 \ge x(t_n) \Big) \nonumber \\ & 
= (1+o(1)) \blue{\overline{\Phi}}(x(t_n))  \ge c_1 x(t_n)^{-1} e^{-x(t_n)^2/2} \ge  \frac{c_2}{n \sqrt{\log n}}\,,
%\label{eq:lclt}
\end{align*}
for some positive $c_1$ and $c_2=c_2(q)$. Hence, by \eqref{eq:lem31-app}, for all $n$ large enough,
\begin{align*}
P(G_n) \ge P(\hat{A}_n \cap G_n) \ge \frac{1}{2} P(\hat{A}_n) \ge \frac{c_2}{2 n \sqrt{\log n}}. 
\end{align*} 
Having $\{G_n\}$ independent with $\sum_n P(G_n) = \infty$, we deduce 
by the second Borel-Cantelli lemma that a.s. the events $G_n$ hold for infinitely many values of $n$.
Since 
\begin{align*}
\hat{\R}_{q^n} 
:=\{S_i\}_{i \in \Lambda_n} \subseteq \R_{q^n} \,,
\end{align*} 
we have by the monotonicity of $A \mapsto \ca(A)$, the non-negativity of $G(x,y)$, Lemma \ref{mat1} 
and the definition of $\Lambda_n \subseteq I_n$, that 
\begin{align*}
R_{q^n} \ge \ca (\hat{\R}_{q^n} )
&\ge \frac{|\Lambda_n|}{ \max_{i \in \Lambda_n} \{ \hat{V}_n(i) \} }  
\ge \frac{|\Lambda_n| h_3(t_n)}{(1+ 4 \delta) t_n} \,.
\end{align*}  
Consequently, in view of \eqref{eq:def-q-delta} we have on the event $G_n$ that 
\[
R_{q^n} \ge \frac{1-\delta}{1+ 4 \delta} h_3(t_n) \ge (1-\epsilon) h_3(\red{q^n}) \,,
\]  
which since $G_n$ holds infinity often, \blue{yields} the lower bound in the $\limsup$-\abbr{LIL} 
(along the sub-sequence $q^n$, and with $\R_{q^n} \supseteq \hat{\R}_{q^n}$ of roughly 
the shape of $\C_{\psi(t_n)}$ of Lemma \ref{cy1}).

\blue{Turning to the task of proving Lemma \ref{gn2+gn3}, we give two}
sharp tail estimates for the path of the \emph{one-dimensional walk} 
$(S_m^1)$, \blue{that we will use later for} 
proving \eqref{eq:gn3} and \eqref{eq:gn2}, respectively.
\begin{lem}\label{gn1-1+gn1-2}
Fixing $\delta \in (0,1)$, for some $C<\infty$ and all $t$ large enough,
\begin{align}\label{sd}
\sup_{\ell \le \gamma_t} \big\{ P(S_{t-\ell}^1 \ge \psi(t)-\sqrt{\ell} ) \big\} &\le C P(A_t) \,, \\
\label{eq:gn2-prep}
\blue{P\big(L^c_t | A_t )} & \le 4 \zeta_t \to 0 \,, \qquad 
L_t := \bigcap_{\ell \in (\gamma_t, t]} \Big\{ S_\ell^1 \ge \frac{\psi(t) \ell}{(1+\delta) t} \Big\} \,.
\end{align}
\end{lem}
\blue{
\begin{rem} In \eqref{sd} we claim that the decay $t \mapsto P(A_t)$ of our
moderate deviations upper-tail event $A_t$ is within a universal factor of that for such an event 
with a granted (free) upper fluctuation of $\sqrt{\ell}$ in the first $\ell \le \gamma_t$ 
steps of $S_\ell^1$. 
%our 1D projection $S_t^1$ of the \abbr{srw} in $\Z^3$. 
The event $L_t$ requires $S_\ell^1$ 
to stay above a linear slope which is $1/(1+\delta)<1$ of the slope $\psi(t)/t$ of $A_t$. Thus, 
if $\ell \mapsto S_\ell^1$ was a \abbr{srw} on $\Z$, we could have applied en-route to \eqref{eq:gn2-prep}
a ballot theorem (after conditioning on $S_{\gamma_t}^1$ and $S_t^1$). It is not so here, due to the additional randomness in number 
of steps of the \abbr{srw} $S_\ell \in \Z^3$ along the other two coordinate axis. We thus resort 
%instead 
to proving \eqref{eq:gn2-prep} via Gaussian approximations.
% (see \eqref{eq:lclt}). 
\end{rem}} 
% $\zeta_t = C \exp(-c \log_3 t)$ in \eqref{eq:gn2} and $\zeta_t =  C(\log_3 t)^{-1/2}$ in \eqref{eq:gn3}. We only need that both go to zero).}
%Recall \eqref{eq:lclt} that $A_t$ 
\begin{proof}[Proof of Lemma \ref{gn1-1+gn1-2}] 
Let $\eta_t := \gamma_t/t = (\log_3 t)^{3/2}/(\log_2 t)$,
setting $x(t,r):= (x(t) - \sqrt{3 r})/\sqrt{1-r}$ for $r \le \eta_t$ and $x(t)=x(t,0) :=\sqrt{2 \log_2 t}$.
Then, in view of the uniform Gaussian approximation of \eqref{eq:lclt},
we get \eqref{sd} once we show that uniformly in
$r \le \eta_t$, the standard \blue{Gaussian} measure of $[x(t,r),\infty)$, is at most 
$C$ times the \blue{Gaussian} measure of $[x(t),\infty)$. Note that $\eta_t \to 0$ and $x(t) \to \infty$
as $t \to \infty$, hence $x(t,r)/x(t) \to 1$ uniformly in $r \le \eta_t$. It thus remains only to 
show that for some $C<\infty$ and all $t$ large enough,
\begin{equation}\label{eq:quad-bdd}
\inf_{r \le \eta_t} \{ x(t,r)^2 - x(t)^2 \} \ge - 2 \log C \,.
\end{equation}
Next, 
%setting $u=x(t) \sqrt{r}$, 
our expression for $x(t,r)$ is such 
\[
(1-r) [ x(t,r)^2 - x(t)^2 ] = \blue{\big(x(t) \sqrt{r}  - \sqrt{3}\big)^2  + 3 r - 3} \ge - 3 \,,
\]
yielding \eqref{eq:quad-bdd} and thereby also \eqref{sd}.

Next, setting $s_j := j t/(\log_2 t)$, we partition $(\gamma_t,t]$ into the disjoint intervals 
$J_j=(s_j,s_{j+1}]$, $j \in [(\log_3 t)^{3/2},\log_2 t)$. We likewise partition the events
$A_t \cap L_t^c$ according to whether the 
stopping time $\tau_t := \inf\{ \ell \ge \gamma_t : S_\ell^1 < \frac{\psi(t) \ell}{(1+\delta) t} \}$
equals $\gamma_t$ or alternatively which interval $J_j$ contains $\tau_t$.
Note that if $\tau_t > s_j$ then $S_{s_j}^1 \ge \frac{\psi(t) s_j}{(1+\delta) t}$
and conditioning on the \abbr{srw} filtration at $\tau_t \in J_j$, we get by the
strong Markov property (and i.i.d. increments), of the \abbr{srw}, that 
\[
P(A_t \cap \{\tau_t \in J_j\}) \le q_t (s_j,0) \sup_{s \in J_j} p_t (t-s,0) \,,
\] 
where 
\[
q_t(s_j,y) := P\Big( \big( S_{s_j}^1\big)_+ \ge \frac{\psi(t) s_j}{(1+\delta) t} - y \Big)  \,, \qquad 
p_t(r,y) := P \Big( S_{r}^1 \ge \frac{\psi(t) (\delta t + r)}{(1+\delta)t} + y \Big) \,
\]
and $ \big( S_{s_j}^1\big)_+:= S_{s_j}^1 \vee 0$.  The only other way for the event $L_t^c$ to occur, is by having 
\[
S^1_{\gamma_t} < \frac{\psi(t) \eta_t}{1+\delta} \blue{\eqqcolon} \Delta_t \log_2 t \,.
\]
Partitioning 
%this event 
to $\{S^1_{\gamma_t} \in I_i\}$, for
$I_i :=[(\log_2 t - i) \Delta_t-\Delta_t,(\log_2 t-i) \Delta_t)$ when $0 \le i < \log_2 t$,
and $I_{\log_2 t} :=(-\infty,0)$, yields that
% and 
% $\Delta_t := \sqrt{\frac{(2/3) t}{(1+\delta)^2}} \Big(\frac{\log_3 t}{\log_2 t}\Big)^{3/2}$. 
\begin{align*}
P(A_t \cap \{ S^1_{\gamma_t} \in I_i\}) \le q_t(\gamma_t, i \Delta_t+\Delta_t) p_t(t-\gamma_t,i \Delta_t) 
\end{align*}
and consequently,
\begin{align*}
P(A_t \cap L_t^c) 
& \le \sum_{i} P(A_t \cap \{ S^1_{\gamma_t} \in I_i\})  
+ \sum_{j} P(A_t \cap \{ \tau_t \in J_j \} ) \\
& \le 
(1+\log_2 t) \Big[\sup_{y  \in [0,\Delta_t \log_2 t]
} \{ q_t(\gamma_t,y+\Delta_t) p_t(t-\gamma_t,y)  \}
     +\sup_{s \in J_j, j \ge (\log_3 t)^{3/2}} q_t(s_j,0) p_t(t-s,0) \Big] \,.
\end{align*}
As $P(A_t)=p_t(t,0)$, we thus get \eqref{eq:gn2-prep}, if
for some $\zeta_t \to 0$ 
\begin{align}\label{eq:bd-L}
q_t(s_j,0) p_t(t-s,0) &\le \frac{\zeta_t}{\log_2 t} p_t(t,0) \,,  \qquad \forall j \ge (\log_3 t)^{3/2}, s \in J_j\,, \\
q_t(\gamma_t,y+\Delta_t) p_t(t-\gamma_t,y) &\le \frac{2 \zeta_t}{\log_2 t} p_t(t,0) \,, \qquad \forall 
y \red{\in [0, \Delta_t \log_2 t]}  \,.
\label{eq:bd-Y}
\end{align}
Proceeding to verify \eqref{eq:bd-L} and \eqref{eq:bd-Y}, we rely on \eqref{eq:lclt} to replace both $q_t(\cdot)$
and $p_t(\cdot)$ by the Gaussian measure of the corresponding intervals.  We claim that when doing so,  all 
partial sums appearing in  \eqref{eq:bd-L} and \eqref{eq:bd-Y}
be at time index ${\sf t} \ge \delta \psi(t)/2 \to \infty$.  Indeed,  note that 
$p_t(r,y)=0$ for $y \ge 0$,  unless the time index $r$ is at least $\delta \psi(t)/2$,  whereas in all the terms 
$q_t(\cdot,\cdot)$ that appear there,  such time indices are $s_j \ge \gamma_t \ge \delta \psi(t)/2$. 
Further,  the argument ${\sf x}$ of $\overline{\Phi}(\cdot)$
in such Gaussian approximations of the probabilities 
$q_t(\cdot)$ and $p_t(\cdot)$ that appear in \eqref{eq:bd-L},  is $\frac{x(t)}{1+\delta}\sqrt{s_j/t}$ 
and $\frac{x(t)}{1+\delta} \frac{1+\delta-s/t}{\sqrt{1-s/t+\xi_t}}$ at $\xi_t=0$,  respectively (where 
$x(t)=\sqrt{3}\psi(t)/\sqrt{t}$). Taking instead
$\xi_t := (\log_2 t) t^{-1/6}$ guarantees that all such
space arguments ${\sf x}$ be uniformly of $o({\sf t}^{1/6})$. The
%Thus, a uniform normal approximation applies for both $q_t(s_j,0)$ and $p_t(t-s,0)$, and with all
%arguments  
Gaussian approximations then hold, as in \eqref{eq:lclt}, with \blue{the same}
$o(1)$ \emph{relative} error \blue{for all the terms, which we thus ignore hereafter.
Note also that} $s/t \in [\eta_t,1)$ with $s_j/t \ge \blue{s/t} - \varepsilon_t$
for $\varepsilon_t:=(\log_2 t)^{-1}$.  In conclusion,  by the preceding it suffices for \eqref{eq:bd-L} 
to show that
\begin{equation}\label{eq:bd-L2}
\sup_{u \in [\eta_t,1)} \Big\{
\overline{\Phi} \Big( \frac{x(t)}{1+\delta}\sqrt{u-\varepsilon_t} \, \Big) 
\overline{\Phi} \Big(\frac{x(t) (1+\delta-u)}{(1+\delta)\sqrt{1-u\red{+\xi_t}}} \Big) \Big\}
 \le \frac{\zeta_t}{\log_2 t} \overline{\Phi}(x(t))\,.
\end{equation}
\blue{The arguments of $\overline{\Phi}(\cdot)$ in \eqref{eq:bd-L2} grow to infinity with $t$, 
uniformly over $u \in [\eta_t,1)$. Thus,} 
recalling that $|\log \overline{\Phi}(y) +\log y + y^2/2|$ is bounded at $y \to \infty$, 
upon taking the logarithm of both sides of \eqref{eq:bd-L2}, noting that $\eta_t \ge 2\varepsilon_t$
\blue{and ignoring all the uniformly bounded terms, such as} 
$x(t)^2 \varepsilon_t$ \blue{and $\log x(t) - \frac{1}{2} \log_3 t$, it suffices to show that} 
for some $\tilde{\zeta}_t \to 0$ as $t \to \infty$, \blue{
\[
\sup_{u \in [\eta_t,1)} \Big\{
\frac{1}{2} x(t)^2 -\frac{1}{2} \frac{x(t)^2 u}{(1+\delta)^2}
-\frac{1}{2} \frac{x(t)^2 (1+\delta-u)^2}{(1+\delta)^2 (1-u \red{+\xi_t})} 
-\frac{1}{2} \log u
 \Big\}
 \le \log \tilde{\zeta}_t - \frac{1}{2} \log_3 t \,.
\]
With $\eta_t \gg \xi_t$, it is easy to verify that for any $u \ge \eta_t$}
\[
\frac{u}{(1+\delta)^2} + \frac{(1+\delta-u)^2}{(1+\delta)^2 (1-u\red{+\xi_t})} -1 \ge  
\frac{u}{(1-u\red{+\xi_t})} \frac{\delta^2}{2 (1+\delta)^2} \eqqcolon \red{\theta_t}(u) \,.
\]
\blue{Hence, substituting $x(t)^2 = 2 \log_2 t$ in the preceding, we arrive} (after some algebra) at
\begin{equation}\label{eq:bd-L3}
\sup_{u \in [\eta_t,1)} \Big\{ -(\log_2 t) \red{\theta_t} (u) - \frac{1}{2} \red{\log u} \Big\}
\le  \log \tilde{\zeta}_t - \frac{1}{2} \log_3 t  \,.
\end{equation}
Since $u \mapsto \theta_t(u)$ is non-decreasing, the supremum on the 
left-side of \eqref{eq:bd-L3} is attained at
$u=\eta_t$, where for large $t$ it is at most $-\frac{\delta^2}{\red{10}} (\log_3 t)^{3/2}$. This is more
than enough for \eqref{eq:bd-L3}  to hold, thereby establishing \eqref{eq:bd-L}. We next turn to 
\eqref{eq:bd-Y}, \blue{where the Gaussian approximation of \eqref{eq:lclt} again applies for 
all three probabilities, with uniform, hence negligible, $o(1)$ relative errors. Thus,} 
analogously to \eqref{eq:bd-L2}, the bound \eqref{eq:bd-Y} is a consequence of \blue{having}
\begin{equation}\label{eq:bd-Y2}
\sup_{v \in [0,1]} \Big\{
\overline{\Phi} \Big( \frac{x(t)}{1+\delta} \sqrt{\eta_t} (v-\varepsilon_t) \, \Big) 
\overline{\Phi} \Big(\frac{x(t) (1+\delta-v \,\eta_t)}{(1+\delta)\sqrt{1-\eta_t}} \Big) \Big\}
 \le \frac{\zeta_t}{\log_2 t} \overline{\Phi}(x(t))
\end{equation}
\red{(temporarily setting $\overline{\Phi}({\sf x})=1$ wherever ${\sf x}<0$). Considering first
$v \le 1/2$,} we bound the left-most term \red{of \eqref{eq:bd-Y2}}
by one, \blue{and as before,} take the logarithm of both sides, \blue{replace 
$\log \overline{\Phi}(y)$ by $-\log y - y^2/2$ and eliminate all uniformly 
bounded terms} 
to find that \eqref{eq:bd-Y2} holds because
\[
% \frac{\eta_t (v-\varepsilon_t)^2}{(1+\delta)^2} 
\blue{\inf_{\nu \in [0,\frac{1}{2} ]}} \Big\{ \frac{(1+\delta-v \eta_t)^2}{(1+\delta)^2 (1-\eta_t)} \Big\} - 1 \ge \frac{\eta_t  \delta}{1+\delta} 
\]
and $x(t)^2 \eta_t = 2 (\log_3 t)^{3/2} \gg \log_3 t$. To complete the proof of \eqref{eq:bd-Y2}, 
it thus suffices (similarly to \eqref{eq:bd-L3}), to have for some 
$\tilde{\zeta}_t \to 0$, that 
\begin{equation}\label{eq:bd-Y3}
\sup_{v \in [1/2,1]} \Big\{ -(\log_2 t) \theta_t (v) \Big\} 
\le  \log \tilde{\zeta}_t - \log_3 t  
% + \frac{3}{4} \log_4 t 
\,,
\end{equation}
where, \blue{recalling that $\eta_t \gg \varepsilon_t$,} 
it is easy to check that for any $\nu \in [\frac{1}{2},1]$, 
\[
\theta_t(v) := \frac{\eta_t \red{(v-\varepsilon_t)}^2}{(1+\delta)^2} +
\frac{(1+\delta-v \eta_t)^2}{(1+\delta)^2 (1-\eta_t)} - 1 \ge \theta_t(1) \ge 
\frac{\eta_t \delta^2}{(1+\delta)^2} \,.
\]
The preceding suffices for \eqref{eq:bd-Y3} and thereby for \eqref{eq:bd-Y},
thus completing the proof.
% of the lemma.
\end{proof}

\begin{proof}[Proof of Lemma \ref{gn2+gn3}] Starting with \eqref{eq:gn3}, note that for some $C_1<\infty$ 
and any $\ell \ge 0$,
\[
E[G(0,S_\ell)] = \sum_{i = \ell}^\infty P^0(S_i=0) \le C_1 \big(1+\sqrt{\ell}\big)^{-1} \,.
\]
Further, recall from \eqref{grf} that 
$G(0,y) \le C_2/(1+|y|)$ for some $C_2<\infty$ and all $y \in \Z^3$. Hence,
\begin{equation}\label{eq:bd-At}
E[G(0,S_\ell) 1_{A_t}] \le C_2 (1+\sqrt{\ell})^{-1} P(A_t) +  
\sum_{|y| \le \sqrt{\ell}} G(0,y) P(\{S_\ell=y\} \cap A_t) \,.
\end{equation}
With $(S_\ell,S_t-S_\ell) \stackrel{d}{=} (S_\ell,\hat{S}_{t-\ell})$ for
\abbr{srw} $(\hat{S}_m)$ which is independent of $(S_m)$, if $y \in \Z^3$ 
is such that $|y^1| \le |y| \le \sqrt{\ell}$, then 
\[
P(\{S_\ell = y \} \cap A_t) \le P\big( S_\ell = y\big) P\big( \hat{S}^1_{t-\ell} \ge \psi(t) - \sqrt{\ell} \big) \,.
\]
Therefore, thanks to \eqref{sd}, for any $\ell \le \gamma_t$ the right-most term in \eqref{eq:bd-At} 
is at most,
\[
% \sum_{|y| \le \sqrt{\ell}} G(0,y) P(A_t \cap \{S_\ell=y\}) \le 
E[ G(0,S_\ell)] P(\hat{S}^1_{t-\ell} \ge \psi(t) - \sqrt{\ell}) \le 
C_1 (1+\sqrt{\ell})^{-1} C P(A_t) \,.
\]
We thus deduce from \eqref{eq:bd-At}, that for some $C_3,
%=C_2 + C_1 C$ and some $
C_4<\infty$, 
\[
E[V_{[0,\gamma_t]}] = \sum_{\ell=0}^{\gamma_t} E[G(0,S_\ell) 1_{A_t} ] \le 
C_3 P(A_t) \sum_{\ell=0}^{\gamma_t} (1+\sqrt{\ell})^{-1} \le C_4 P(A_t) \sqrt{\gamma_t} \,.
\]
Consequently, by Markov's inequality and our choice of $\gamma_t$, 
\[
P\Big(  V_{[0,\gamma_t]} \ge \frac{\delta t}{h_3(t)} \Big)  \le \zeta_t P(A_t) \,,
\]
where our choice of $\gamma_t$ results with  
$\zeta_t := C_4 h_3(t) \sqrt{\gamma_t}/(\delta t)\to 0$ as $t \to \infty$.

Turning to \eqref{eq:gn2}, recall our choice of $\gamma_t$ implying that
\[
\sum_{\ell \in (\gamma_t,t]} \frac{1}{\ell} \le \log (t/\gamma_t) \le \log_3 t \,.
\]
Further, with $h_3(t)=(\pi/3) \psi(t) (\log_3 t)^{-1}$, it follows that 
$t/(2 h_3(t))=\frac{3}{2\pi} (t/\psi(t)) (\log_3 t)$. Consequently, for all $t$ large enough,
the event $L_t$ of \eqref{eq:gn2-prep} implies, thanks to \eqref{grf}, that 
\[
V_{(\gamma_t,t]} \blue{\le} \sum_{\ell \in (\gamma_t,t]} G(0,S_\ell) \le \frac{3+o(1)}{2\pi} 
\sum_{\ell \in (\gamma_t,t]} \red{|}S_\ell^1\red{|}^{-1} \le \frac{3+o(1)}{2\pi} \frac{(1+\delta) t}{\psi(t)} (\log_3 t)
\le \frac{(1+2\delta) t}{2 h_3(t)} \,.
\]
Since the same conclusion applies when $A_t^c$ holds (in which case $V_{(\gamma_t,t]}=0$), we
see that \eqref{eq:gn2} is an immediate consequence of \eqref{eq:gn2-prep}.
\end{proof}

%%%%%%%%%%%%%%%%%%%%%%%%%%%%%%%%%%%%%%%%%%%%%%%%%%%%%%%%%%%%%%%%%%

\subsection{The upper bound in the limsup-\abbr{LIL}}\label{subsec:limsup-3u}
Recall that $\crn$ is non-decreasing. Further, for any $t_n = q^n$, $q>1$ 
we have that eventually $h_3(t_n)/h_3(t_{n-1}) \le q$. It thus suffices to prove the upper bound 
in our $\limsup$-\abbr{LIL} only along each such sequence $t_n$ (thereafter 
taking $q \downarrow 1$ to complete the proof). To this end, fix $q>1$ and $\delta,\eta > 0$, 
\blue{and set} $m:=(1+\delta)^3 \psi(t_n)$, \blue{with} $r_m = m/b_m \uparrow \infty$ for
$b_m :=2 \eta (1+\delta)^6 (\log_2 t_n)$ (so $r_m:= \psi(t_n)/(2 \eta (1+\delta)^3 \log_2 t_n)$).
We aim to cover $\R_{t_n}$ by the union $\C_m^\star(r_m)$ of $b_m$ balls of radius 
$r_m$ each, with the centers of consecutive balls at most $r_m$ apart.
Indeed, as shown in Section \ref{sec-Green} 
(see Remark \ref{for-ubd-sec3}), this would yield  
$R_{t_n} \le (1+\delta+o(1))^3 h_3(t_n)$, so we then conclude 
%the proof
by taking $\delta \downarrow 0$ and $q \downarrow 1$.

Specifically, starting at $T_0=0$, set the increasing stopping times  
\begin{align*}
T_i :=\inf \{k>T_{i-1} :  |S_k-S_{T_{i-1}}| > r_m-1 \}, \qquad \forall i \ge 1 \,,
\end{align*}  
noting that the event $\{T_{b_m} \ge t_n\}$ implies the aforementioned containment 
$\R_{t_n} \subseteq \C_m^\star(r_m)$. Further, with 
$\exp(-(1+\delta) \log_2 t_n) \le C n^{-(1+\delta)}$ summable, upon employing 
the first Borel-Cantelli lemma, it remains only to establish the following key lemma.
\begin{lem}\label{bl1}
For any $q>1$ and small $\delta > 0$, there exist $\eta>0$ and $C<\infty$ such that 
\begin{align}\label{eq:Tb-bd}
P(T_{b_m}<t_n) \le C\exp(-(1+\delta) \log_2 t_n) \qquad \forall n \,.
\end{align}
\end{lem}
\begin{proof} By the strong Markov property and the independence of increments of the walk,
we see that $T_{b_m}$ is the sum of $b_m$ i.i.d. copies of the first exit time $T_1$ 
of the (discrete) ball $\BB(0,r_m-1)$, by the 3D-\abbr{srw}. \blue{As $r_m \uparrow \infty$,  
Skorokhod's embedding implies, see \cite[Lemma 3.2]{LL13}, that for some
$m_o,c<\infty$ and $\varepsilon>0$ (depending only on $\delta>0$), 
all $m \ge m_o$ and $u \ge 0$,
\[
P(T_1 < u) \le c e^{-r_m^\varepsilon} +  P(3 \underline{T} < u) = 
c e^{-r_m^\varepsilon} + P(3 (1+\delta)^{-1} r_m^2 \barT_1 < u ) \,,
\]}
with $\barT_1 := \inf \{ t \ge 0 : |B_t| \ge 1 \}$
the Brownian hitting time of the unit sphere $\sss^2$
\blue{and $\underline{T} := \inf \{ t \ge 0: |B_t| \ge r_m / \sqrt{1+\delta} \}$ 
(so the identity above is merely Brownian scaling). Further, here 
$b_m e^{-r_m^\varepsilon} \ll \exp(-2 \log_2 t_n)$ and} $t_n = 3 r_m^2\, \eta \,b_m$ 
\blue{with} $b_m/(2\eta) =3 m^2/(2 t_n) = (1+\delta)^6 (\log_2 t_n)$. \blue{It thus} 
suffices to show that for 
some $\eta=\eta(\delta)>0$ and all $m$,
\begin{equation}\label{eq:bm-ubd}
P\Big(\frac{1}{b_m} \sum_{i=1}^{b_m} \barT_i < \eta \Big) \le e^{-(1-\delta)^3 b_m/(2 \eta)} \,,
\end{equation}
where $\barT_i$ are i.i.d. copies of $\barT_1$. To this end, covering $\sss^2$ by $c_\delta$ 
balls of radius $\delta$ each, centered at some $\theta_i \in \sss^2$, we have by the triangle 
inequality that 
\[
\max_i \{ \langle \theta_i, B_{\barT_1} \rangle \} \ge  1-\delta \,.
\]
Hence, fixing $\lambda>0$, we get upon applying Doob's optional stopping theorem 
for the martingale $M_t = \sum_i \exp(\lambda \langle \theta_i, B_t \rangle - \lambda^2 t/2)$ 
at the stopping time $\barT_1$, that 
\[
c_\delta = M_0 = E[M_{\barT_1}] \ge e^{\lambda (1-\delta)} E[e^{-\lambda^2 \barT_1/2} ] \,.
\]
Consequently, by Markov's inequality, we have for any $\eta,\lambda,\delta>0$ and integer $b \ge 1$, that
\begin{equation}\label{eq:mgf}
P(\frac{1}{b} \sum_{i=1}^{b} \barT_i < \eta) \le e^{\lambda^2 b \eta/2} E[e^{-\lambda^2 \barT_1/2} ]^b 
\le \Big( c_\delta e^{\lambda^2 \eta/2-\lambda (1-\delta)} \Big)^b \,.
\end{equation}
Taking the optimal $\lambda=(1-\delta)/\eta$ it is easy to check that for 
$\eta \le \eta(\delta) = \delta (1-\delta)^2/(2 \log c_\delta)$, the \abbr{lhs} of \eqref{eq:mgf} is 
at most $\exp(-(1-\delta)^3 b/(2\eta))$. We thus got 
\eqref{eq:bm-ubd} for any $\delta>0$ provided $\eta \le \eta(\delta)$, thereby completing the proof.
\end{proof}

\begin{rem} One has for any $\delta>0$ small and all $t$ large enough, the classical bound
\begin{align*}
P( \max_{1\le k\le t}  |S_k|\ge (1+\delta)^3 \psi(t)  ) \le C e^{-(1+\delta) \log_2 t} \,.
\end{align*}
We need in \eqref{eq:Tb-bd} a stronger result, since for any $b_m$ and $r_m \gg 1$,
 \begin{align*}
\{ \max_{1\le k\le t_n}  |S_k|\ge b_m r_m \} \subset \{ T_{b_m} < t_n \},
\end{align*}
and while $b_m r_m = (1+\delta)^3\psi(t_n)$, our crude use of $\delta$-cover of $\sss^2$ in 
proving Lemma \ref{bl1} requires us to also have $b_m/(\log_2 t_n) \to 0$ as $\delta \downarrow 0$. 
\end{rem}

%%%%%%%%%%%%%%%%%%%%%%%%%%%%%%%%%%%%%%%%%%%%%%%%%%%%%%%%%%%%%%%%%%

\subsection{The upper bound in the liminf-\abbr{LIL}}\label{ioo2}
For any $A \subset \RR^3$ and $r> 0$, let 
\begin{align*}
\mathrm{Nbd}(A,r):= \bigcup_{x \in A} \BB(x,r) 
\end{align*}
denote the $r$-blowup of $A$. Utilizing \cite{Ch17}, we first relate $\crn$ with a suitable 
Brownian capacity, as stated next.
\begin{lem}
We can couple the \abbr{srw} with a 3D Brownian motion $(B_t, t \ge 0)$, such that 
\begin{align}\label{q0}
 \lim_{n\to \infty}\frac{\crn }{\ca_{B} (B[0,n/3] )}=\frac{1}{3}
 \quad\text{ a.s.}  
\end{align}
and for any $\delta \in (0,1/2)$, 
\begin{align}\label{q1}
 \lim_{n\to \infty}\frac{\crn }{\ca_{B} (\mathrm{Nbd}(B[0,n/3], n^{1/2-\delta}) )}=\frac{1}{3}
 \quad\text{ a.s.}  
\end{align}
\end{lem}
\begin{proof}
The results were essentially shown in  \cite{Ch17}. Indeed, \cite[(4.15)]{Ch17} shows that 
\eqref{q0} holds when each ratio is restricted to the events $E_n$, while it is also shown 
that a.s. $E_n$ holds for all sufficiently large $n$ (combine \cite[(4.2)]{Ch17} with Borel-Cantelli).
Hence \eqref{q0} also holds without such a restriction. Turning to 
show \eqref{q1}, let $\tilde{P}$ denote the probability of an independent Brownian motion 
$(\tilde{B}_t)$. 
Fixing $\delta \in (0,1/2)$ and some 
$y_{n} \in \Z^3$ such that $|y_{n}|=n^{1/2+\delta}$, we similarly obtain (after dispensing 
of events $E_n$), that by the same argument as in \cite[(4.4)]{Ch17}, a.s. one has for all large $n$,
\begin{align*}
 \ca_{B} & (\mathrm{Nbd}(B[0,n/3],  n^{1/2-\delta}) )\\
& =(2\pi+o(1))
n^{1/2+\delta} 
\tilde{P}(\mathrm{Nbd}(B[0,n/3], n^{1/2-\delta}) \cap (y_n+\tilde{B}[0, \infty) )\neq \emptyset | B[0,n/3] ) \,.
\end{align*}  
By \cite[(4.13) \& (4.4)]{Ch17} , a.s., the latter expression is for all $n$ large 
$(1+o(1)) \ca_{B}(B[0,n/3])$. 
Thus,
\begin{align*}
 \lim_{n\to \infty}\frac{\ca_{B} (B[0,n/3] )}{\ca_{B} (\mathrm{Nbd}(B[0,n/3], n^{1/2-\delta}) )}=1
 \quad\text{ a.s.}  \,,
\end{align*}
which in view of (\ref{q0}) completes the proof of \eqref{q1}.  
\end{proof}

Proceeding to show the upper bound in the $\liminf$-\abbr{LIL}, let 
$r(n) := \pi\sqrt{n/(2 \log_2 n)}$, that is $r(n) = \sqrt{3} \, \hat{\psi} (n)$ for 
$\hat{\psi}(\cdot)$ of Proposition \ref{m1+}. Recall that  
by \cite[Lemma 1.1]{HH16} or \cite{K80}, and Brownian scaling, for some $c>0$ and all $t,r>0$, 
\begin{align}\label{eq:Tx-lb}
 P(\sup_{s\in [0,t)} \{ |B_s| \} \le r) \ge  2c e^{-\frac{\pi^2 t}{2 r^2}} \,.
% \ge c \exp(-2^{-1} j^2_{1/2,1} t),
\end{align}
We have used in \eqref{eq:Tx-lb} also that the largest eigenvalue of the Dirichlet Laplacian in the 
unit ball in $\RR^d$ is $-j^2$, where $j=j_{(d-2)/2,1}$ denotes the first positive 
zero of the Bessel function of the first kind with index $(d-2)/2$, and in 
particular that $j_{1/2,1}=\pi$ (see \cite[Page 490]{W44}). Considering 
\eqref{eq:Tx-lb} for $r=r(s_n)$, $s_n=n^n$
and the Brownian increments in the disjoint intervals $[s_{n-1},s_n)$ of length $s_n-s_{n-1}$,
result with
\begin{align}\label{chung*}
 P\bigg(\sup_{t \in \blue{[s_{n-1},s_n)}} \{ |B_t-B_{s_{n-1}}| \} \le  r(s_n) \bigg)
 \ge c \exp(-  \log_2 s_n) = \frac{c}{n \log n} \,.
\end{align}
Thus, thanks to the independence of Brownian increments on \blue{these disjoint intervals,}
we get from the second Borel-Cantelli lemma that 
\begin{align}\label{lil1}
 P\bigg(\liminf_{n \to \infty} \,
 (r(s_n)^{-1} \sup_{t \in \blue{[s_{n-1},s_n)}} \{ |B_t-B_{s_{n-1}}|\}  ) \le  1 \bigg)
=1.
\end{align}
Further, as $\sqrt{ s_{n-1} (\log_2 s_{n-1})} = o(r(s_n))$, 
by Kinchin's \abbr{LIL} for the Brownian motion, 
\begin{align}\label{lil2}
 P\bigg(\limsup_{n \to \infty} (r(s_n)^{-1} \sup_{t\le s_{n-1}} |B_t| ) =0 \bigg) =1.
\end{align}
Combining (\ref{lil1}) and (\ref{lil2}), we deduce that 
\[
 P\bigg( \liminf_{n \to \infty} (r(s_n)^{-1} \sup_{t<s_n} |B_t|)  \le  1 \bigg) =1. 
\]
This of course implies that also 
\begin{align}\label{lil3}
 P\bigg( \liminf_{n \to \infty} (r(n)^{-1} \sup_{t<n} |B_t|)  \le  1 \bigg) =1. 
\end{align}
Recall that for any $r>0$ one has that $r^{-1} \ca_B (\BB(0,r))=\ca_B(\BB(0,1)) = 2\pi$ 
($=\kappa_1$ on \cite[Page 356]{BBH01}).
By \eqref{lil3}, for any $\epsilon>0$, a.s. $B[0,n] \subset \BB(0,(1+\epsilon) r(n))$ for
infinitely many values of $n$, in which case also $\ca_B(B[0,n]) \le 2\pi (1+\epsilon) r(n)$. That is,
 \begin{align}\label{eq:ubd-cB}
 P\bigg(   \ca_B(B[0,n]) \le  2 \pi (1+\epsilon) r(n)  \quad \text{ i.o. } \bigg)=1.
\end{align}
By Brownian scaling, the sequence $\{\sqrt{3} \ca_B(B[0,n/3])\}$ has the same law as
the sequence $\{\ca_B(B[0,n])\}$. Thus, in view of (\ref{q0}), we can also 
construct a coupling so that for any $\epsilon>0$, we have that a.s.
\begin{align*}
\crn  \le & \frac{(1+\epsilon)}{3\sqrt{3}} \ca_B(B[0,n]) \,,
\end{align*}
 for all $n$ large enough. With $r(n)=\frac{3 \sqrt{3}}{2\pi} \hat{h}_3(n)$, it thus follows from \eqref{eq:ubd-cB} that 
 \begin{align*}
 P\Big( \, \crn  \le (1+\epsilon)^2 \hat{h}_3(n) \quad \text{ i.o. } \Big) =1 \,,
\end{align*}
and taking $\epsilon \downarrow 0$ establishes the stated upper bound in our $\liminf$-\abbr{LIL}.
%%%%%%%%%%%%%%%%%%%%%%%%%%%%%%%%%%%%%%%%%%%%%%%%%%%%%%%%%%%%%%%%%%

%\subsection{Proof of the lower bound in Theorem \ref{m2}}
\subsection{The lower bound of the liminf-\abbr{LIL}}
Fixing $a>0$, we have by \cite[(1.4)]{BBH01}, 
that for any $f(t) \uparrow \infty$ such that $f(t)=o(t^{2/3})$,
\begin{align}\label{large1}
\lim_{t\to \infty} 
\frac{f(t)}{t} 
\log P \bigg( | \mathrm{Nbd}(B[0,t], a) | \le \bigg(\frac{\pi}{\sqrt{2}} \bigg)^3 f(t)^{3/2}   \omega_3 \bigg)=-1, 
\end{align}
where $\omega_3$ denotes the volume of the unit ball (using here that 
the largest eigenvalue of the Dirichlet Laplacian in the \emph{unit volume ball} in $\RR^3$ 
is $-\omega_3^{2/3} \pi^2$). Fixing $\delta \in (0,1/2)$ as in \eqref{q1},
Brownian scaling by time factor $3 n^{2 \delta-1}$ yields equality in distribution between the sequences
\begin{align*}
|\mathrm{Nbd}(B[0,n/3], n^{1/2-\delta})|
\stackrel{d}{=} 3^{-3/2} n^{3/2-3\delta} |\mathrm{Nbd}(B[0, n^{2\delta}], 3^{-1/2})| \,.
\end{align*} 
Thus, considering  (\ref{large1}) for 
$a=3^{-1/2}$, $t=n^{2\delta}$ and $f(t)=(1-\epsilon)^3 n^{2\delta}(\log_2 n)^{-1}$, we arrive at
\begin{align*}
& P\bigg(|\mathrm{Nbd}(B[0,n/3], n^{1/2-\delta}) | 
 \le (1-\epsilon)^{2} \hat{\psi} (n)^3 \omega_3  \bigg)\\
 = & P\bigg(| \mathrm{Nbd}(B[0, n^{2\delta}], 3^{-1/2}) | 
 \le (1-\epsilon)^{2} \bigg(\frac{\pi}{\sqrt{2}} \bigg)^3   (n^{2\delta}(\log_2 n)^{-1})^{3/2} \omega_3  \bigg)\\
\le & C \exp(-(1-\epsilon)^{-2}(\log_2 n)).
\end{align*}
Considering $n_k=q^k$, we get by the first Borel-Cantelli lemma, that for fixed 
$q>1$ and $\epsilon>0$, 
\begin{align}\label{q2}
\liminf_{k \to \infty } \frac{ | \mathrm{Nbd}(B[0,n_k/3], n_k^{1/2-\delta}) | }{ \omega_3 \, 
\hat{\psi} (n_k)^3 }
\ge (1-\epsilon)^2
 \quad\text{ a.s.}  
\end{align}
With $n \mapsto |\mathrm{Nbd}(B[0,n/3], n^{1/2-\delta})|$ monotone increasing and 
$\hat{\psi}(q^k)/\hat{\psi}(q^{k-1}) \to 1$ as $k \to \infty$ followed by $q \downarrow 1$, we deduce from  
\eqref{q2} that 
\begin{align}\label{q2b}
\liminf_{n \to \infty } \frac{ | \mathrm{Nbd}(B[0,n/3], n^{1/2-\delta}) | }{  \omega_3 \,  
\hat{\psi}(n)^3 }
\ge 1 \quad\text{ a.s.}  
\end{align}
Next, recall the Poincar\'{e}-Carleman-Szeg\"{o}  theorem \cite{PS51}, that for any $r>0$, 
\begin{align}\label{q3}
\inf_{ |A|=\omega_3 r^3} \{\ca_B (A)\} =\ca_B (\BB(0,r) ) = r \, \ca_B(\BB(0,1)) = 2 \pi r \, . 
\end{align}
Recall that $\hat{h}_3(n)=\frac{2\pi}{3} \hat{\psi}(n)$. Hence, by 
\eqref{q3} for $A=\mathrm{Nbd}(B[0,n/3], n^{1/2-\delta})$, we have in view of (\ref{q2}) that 
\begin{align*}
 \liminf_{n\to \infty}\frac{ \ca_B \big( \mathrm{Nbd}(B[0,n/3], n^{1/2-\delta}) \big) }{ 3 \hat{h}_3(n)  } \ge 1 \,,
 \quad\text{ a.s.,}  
\end{align*}
which together with (\ref{q1}) yields the stated lower bound for the $\liminf$-\abbr{LIL}
of $\crn$ in $\Z^3$.

%%%%%%%%%%%%%%%%%%%%%%%%%%%%%%%%%%%%%%%%%%%%%%%%%%%%%%%%%%%%%%%%%%%%

\section{\abbr{LIL} for \abbr{srw} on $\Z^4$: Proof of Theorem \ref{m2}}\label{sec:m2-d4}

Hereafter we consider for integers $0 \le a \le b \le c$, the random variables
\begin{equation}\label{def:Rab}
R_{a,b} := \ca(\R(a,b]) \,, \qquad V_{a,b,c} := R_{a,b} + R_{b,c} - R_{a,c} \ge 0 \,.
\end{equation}
Note that by shift invariance $R_{a,b} \stackrel{d}{=} R_{0,b-a}=R_{b-a}$ and
$R_{a,b}$ is independent of $R_{b,c}$ (due to the independence of increments).  In particular, 
for any increasing $\{n_k\}$ starting at $n_0=0$, one has the decomposition 
\begin{equation}\label{eq:decomp4}
R_{n_k} 
%= \sum_{j=1}^k \overline{R}_{n_{j-1},n_j} - \sum_{j=1}^{k-1} \overline{V}_{n_{j-1},n_j,n_k} 
:= \sum_{j=1}^k U_j - \Delta_{n_k,k} \,,
\end{equation}
in terms of the independent variables $U_j :=R_{n_{j-1},n_j}$ and the sum of non-negative 
variables 
\begin{equation}\label{def:Delta}
\Delta_{n_k,k} = \sum_{j=1}^{k-1} V_{n_{j-1},n_j,n_k} = \sum_{j=1}^{k-1} V_{0,n_j,n_{j+1}} \,.
\end{equation}
\blue{We use different non-random sub-sequences $(n_k)$ for $d=4$, for $d=5$ and for $d \ge 6$.
Also, the subsequences used for the $\limsup$-\abbr{LIL} and for the $\liminf$ \abbr{LIL}-s be different.}
As we shall see, \blue{for such suitable $n_k$,} the fluctuation in $\sum_j U_j$ is negligible for the \abbr{LIL}-s of 
Theorem \ref{m2} where $E \Delta_{n_k,k} \approx h_4(n_k)$, the $\limsup$-\abbr{LIL} 
being due to the exceptional times with $\Delta_{n_k,k} = o(E \Delta_{n_k,k})$, while the 
$\liminf$-\abbr{LIL} is due to the exceptional times where 
$\Delta_{n_k,k} \approx \hat{h}_4(n_k) \gg E \Delta_{n_k,k}$. 
In contrast, we show in Section \ref{sec:m2-dge5} that $\Delta_{n_k,k}$ has a negligible effect
when $d \ge 5$, where the \abbr{LIL} follows the usual pattern for 
sums of independent variables (namely, that of the \abbr{LIL} for a Brownian motion).
\blue{We take a relatively small $k=O(\log_2 n_k)$ 
for the $\limsup$-\abbr{LIL} and $d=4$, with larger $k=O((\log n_k)^\alpha)$ for
the $\liminf$-\abbr{LIL} and for any $d \ge 5$.}

\subsection{The lower bound in the limsup-\abbr{LIL}}

We start with \blue{the statement of a key lemma about the   
\abbr{srw} on $\Z^4$} \red{(which is related to \cite[Thm. 2.2]{CR} 
in the $4D$ Brownian motion case).}
\begin{lem}\label{extail:lem}
Suppose $(S_m)$ and $(\tilde{S}_m)$ are two independent \abbr{srw} on $\Z^4$.
% while $(\beta_s, s \ge 0)$, $(\tilde{\beta}_s, s \ge 0)$ are two independent, standard 
% $4$-dimensional Brownian motions. 
%\blue{With $G_\beta(\cdot)$ as in \eqref{eq:green4}, 
Let
\begin{align*}
X_n:=\frac{1}{n} \sum_{i,\ell \in [1,n]}
 G(S_i,\tilde{S}_\ell).
% \qquad Y:=\int_0^1\int_0^1 G_\beta (\beta_s,\tilde{\beta}_t) dsdt\,.
\end{align*}
Then, for some $C<\infty$ and any $p,n \in \N$,
 \begin{align}\label{exptail*}
E[X_n^p] \le C^p p!\,.
%\qquad E[Y^p] \le C^p p! .
\end{align}
\end{lem}
One immediate consequence of \eqref{exptail*} is that for any $c < 1/C$,
 \begin{align}\label{exptail}
\sup_n E[e^{c X_n}] < \infty \,.
% \qquad E[e^{cY} ] < \infty\,.
\end{align}

%\begin{proof}[Proof of Proposition \ref{prop:exptail}] This is 
%a direct consequence of Lemma \ref{extail:lem}.
%Indeed, recall from \cite[Prop. 4.5]{As5}, that the limit in $L^2$,
%\begin{align}\label{dfn:ga-G}
%\gamma_G([0,1]^2) := 2 \sum_{n \ge 1} \sum_{k=1}^{2^{n-1}} \big(\alpha (A_k^n) - E \alpha(A_k^n) 
%\big)\,
%\end{align}
%exists, where for any $n \ge 1$ and $k \le 2^{n-1}$,
%\begin{align}\label{alpha-def}
%\alpha (A_k^n) & := \int_{A_k^n} G_{\beta}(\beta_s, \beta_t) ds dt \,, \\
%A_k^n & := [(2k-2)2^{-n},(2k-1)2^{-n}) \times ((2k-1)2^{-n},2k2^{-n}] \,. \nonumber
%\end{align}
%Note that $\{\alpha(A_k^n)\}_{k \ge 1}$ are i.i.d. for any fixed $n \ge 1$, with 
%$\alpha(A_k^n) \stackrel{d}{=} 2^{-n} Y$. That is, $\{\alpha(A_k^n)\}$ satisfies
%\cite[properties (i) and (ii)]{LG94}. Upon replacing the upper bound of \cite[Lemma 2]{LG94} 
%by \eqref{exptail*}, it is not hard to check that the proof of \cite[Thm. 1]{LG94} yields here that 
%the moment generating function $M_G(\lambda)$ of \eqref{def:mgf} is finite for sufficiently 
%small $\lambda>0$. \blue{Further, as argued in \cite[Remark (a)]{LG94}, the same applies for 
%all $\lambda \le 0$.} By H\"older's inequality, and in view of \cite[Thm. 1.2]{As5}, 
%it follows that $\lambda_o \in (0,\infty)$, as claimed.
%\end{proof}
  
\blue{The proof of Lemma \ref{extail:lem} follows the same scheme as
that of \cite[Lemma 2]{LG94}. However, \cite{LG94} crucially relies on 
an explicit representation of the moments of the Brownian self-intersection local 
time via variances of linear combinations of Brownian increments.
Lacking any such tool here, our more involved proof of Lemma \ref{extail:lem} 
relies instead on the} following elementary bounds,
\red{where \eqref{bine1}-\eqref{bine3} implies also that \eqref{exptail*}-\eqref{exptail} hold
for $\int_0^1\int_0^1 |\beta_s - \tilde{\beta}_t|^{-2} dsdt$ and the independent, 
standard $4$-dimensional Brownian motions
$(\beta_s, s \ge 0)$, $(\tilde{\beta}_t, t \ge 0)$ (which is 
an improvement over the upper bound of \cite[Prop. 4.1]{As5}).}
\begin{lem}\label{mo22}
There exists $C<\infty$ such that for any $t>0$ and $x,y \in \RR^4$,  
\begin{align}\label{bine1}
&E[ | \beta_t-x |^{-2}]
\le C \min\{ t^{-1}, |x|^{-2} \} \le C t^{-1/2} |x|^{-1},\\ 
\label{bine2}
& E[|\beta_t-x |^{-1} |\beta_t - y|^{-1}]
\le  C t^{-1/2} (|x| \vee |y|)^{-1} \le 2 C t^{-1/2} |y-x|^{-1} ,\\ 
\label{bine3}
&E[ |\beta_t -x|^{-2} |\beta_t  - y |^{-1}]
\le 2 C t^{-1/2} |x|^{-1} |y-x|^{-1}.
\end{align}
Similarly, for $|\cdot|_+=|\cdot| \vee 1$, any $i\ge0$ and $x,y \in \Z^4$,  
\begin{align}\label{ine1}
&E[ | S_i-x |_+^{-2}]
\le C \min\{ |i|_+^{-1}, |x|_+^{-2} \} \le C |i|_+^{-1/2} |x|_+^{-1} ,\\
\label{ine2}
& E[|S_i-x |_+^{-1} |S_i - y |_+^{-1}]
\le  C |i|_+^{-1/2} (|x|_+ \vee |y|_+)^{-1} \le 2 C |i|_+^{-1/2} |y-x|_+^{-1} ,\\
\label{ine3}
&E[ |S_i -x|_+^{-2} |S_i  - y|_+^{-1}]
\le 2 C |i|_+^{-1/2} |x|_+^{-1} |y-x|_+^{-1}.
\end{align}
%\begin{align*}
%E[ |S_i  - y |^{-1}]
%\le C  i^{-1/2}.
%\end{align*}
\end{lem}

\begin{proof} Denoting by $\phi_s (x):=(2\pi s)^{-2}\exp(-\frac{|x|^2}{2s})$ the density
at $x$ of the Gaussian law of $\beta_s$, we get from \eqref{eq:green4} 
after change of variables that
\begin{align*}
E[ | \beta_t-x |^{-2}]
 =2\pi^2 E[ G_{\red B}( \beta_t, x )]
=2 \pi^2 \int_t^\infty \phi_s(x) ds = t^{-1} \varphi_1(|x|^2/t) = |x|^{-2} \varphi_2(t/|x|^2) 
\,,  
\end{align*}
for the finite decreasing functions $\varphi_1(r):=\frac{1}{2} \int_1^\infty u^{-2} e^{-r/(2u)} du$,
$\varphi_2(r):=\frac{1}{2} \int_r^\infty u^{-2} e^{-1/(2u)} du$. Thus, \eqref{bine1} holds for any
$C \ge \varphi_1(0) \vee \varphi_2(0)$. Next, by Cauchy-Schwarz and (\ref{bine1}),
\begin{align*}
E( |\beta_t-x |^{-1} |\beta_t - y |^{-1} )
\le  (E [ |\beta_t-x |^{-2} ])^{1/2} (E [ |\beta_t-y|^{-2} ])^{1/2}
\le C t^{-1/2} |y|^{-1} \,.
\end{align*}
Exchanging $x$ with $y$ yields the first inequality in \eqref{bine2}, whereby
the second inequality follows (as $|x-y| \le |x|+|y| \le 2 |x| \vee |y|$). Now, by the triangle inequality, for 
$\beta_t \ne x \ne y$,
\begin{align*}
|y-x|\, |\beta_t  -x|^{-2} |\beta_t- y|^{-1} \le |\beta_t -x|^{-1} (  |\beta_t-x|^{-1}+|\beta_t- y|^{-1}),
\end{align*}
so taking the expectation and using \eqref{bine1} and \eqref{bine2} to bound 
the \abbr{rhs}, results with \eqref{bine3}.
  
With $S_0=0$, clearly \eqref{ine1} holds at $i=0$, while for $i \ge 1$, 
recall \cite[Thm. 1.2.1 and 1.5.4]{LA91} that for some $C$ finite and any $x \in \Z^4$,
\begin{align}\label{gre2}
P(S_i=x) &\le  C i^{-2} \Big[ e^{-2|x|^2/i} + (|x|^2 \vee i)^{-1} \Big], \, \\
C^{-1} |x|_+^{-2} &\le G(0,x) \le C |x|_+^{-2} \,.
\label{gre1}
\end{align}
By \eqref{gre1},
\begin{align*}
E[ | S_i-x |_+^{-2}]
\le CE[ G( S_i, x )]
=C\sum_{\ell=i}^\infty P(S_\ell=x). 
\end{align*}
Further, for some $C$ finite and all $i \ge 1$, $x \in \Z^4$,
\begin{align*}
\sum_{\ell=i}^\infty \ell^{-2} (|x|^2 \vee \ell)^{-1} \le  C i^{-1} (|x|^2 \vee i)^{-1} \le  C \min\{i^{-1}, |x|_+^{-2} \} .
\end{align*}
Thus, in view of \eqref{gre2}, the same computation as for (\ref{bine1}) yields the first inequality of (\ref{ine1}). 
The second inequality of \eqref{ine1} follows (as $a \wedge b \le \sqrt{a b}$ for any $a,b>0$),
and since the strictly positive $|\cdot|_+$ satisfies the triangle inequality on $\Z^4$, we get 
first \eqref{ine2}, and then \eqref{ine3}, by the same reasoning that led to \eqref{bine2}
and \eqref{bine3}, respectively. 
\end{proof}

\begin{proof}[Proof of Lemma \ref{extail:lem}]
From (\ref{ine3}), for any $p \ge 2$, $1 \le s_1 \le \cdots \le s_p$ 
and $\{y_1,\ldots, y_p\} \subset \Z^4$,  
\begin{align*}
&E\Big[\prod_{i=1}^{p-1} |S_{s_i}-y_i|_+^{-2}  | S_{s_{p-1}}-y_p |_+^{-1} \Big]\\
= &     E\Big[ \prod_{i=1}^{p-2} |S_{s_i} -y_i |_+^{-2} |S_{s_{p-1}}-S_{s_{p-2}} -(y_{p-1}-S_{s_{p-2}})|_+^{-2}  | S_{s_{p-1}}-S_{s_{p-2}} -(y_p-S_{s_{p-2}} ) |_+^{-1} \Big]\\
\le &  C    | s_{p-1}-s_{p-2} |_+^{-1/2} |y_p-y_{p-1}|_+^{-1}  
E\Big[\prod_{i=1}^{p-2} |S_{s_i} -y_i|_+^{-2}  |S_{s_{p-2}}-y_{p-1} |_+^{-1} \Big] \,,
\end{align*}
where we set throughout $s_0=0$ and $y_0=0$. 
Thus, by induction on $p \ge 1$, 
\begin{equation}\label{eq:induct}
E\Big[\prod_{i=1}^{p-1} |S_{s_i} -y_i |_+^{-2}  | S_{s_{p-1}}-y_p |_+^{-1} \Big] 
\le C^{p-1} \prod_{i=1}^{p-1} |s_i-s_{i-1} |_+^{-1/2}  \prod_{i=1}^p  |y_i-y_{i-1}|_+^{-1} \,.
\end{equation}
Next, note that 
%since \min(a,b) \le \sqrt{ab} 
by \eqref{ine1}, for any $p \ge 1$, 
\begin{align*}
 E\Big[\prod_{i=1}^p |S_{s_i}-y_i|_+^{-2} \Big]
= &  E\Big[\prod_{i=1}^{p-1} |S_{s_i}-y_i|_+^{-2} |S_{s_p} - S_{s_{p-1}} -( y_p-S_{s_{p-1}} )|_+^{-2} \Big]\\
\le &  C | s_p-s_{p-1} |_+^{-1/2}  E\Big[\prod_{i=1}^{p-1} |S_{s_i} -y_i |_+^{-2}  | S_{s_{p-1}}-y_p |_+^{-1} \Big]
\,.
\end{align*}
Hence, setting $t_0=0$, in view of \eqref{eq:induct} and the independence of $(S_i)$ and $(\tilde{S}_i)$,
\begin{equation}\label{eq:bdG1}
E\Big[ \prod_{i=1}^p |S_{s_i}-\tilde{S}_{t_i}|_+^{-2} \Big] \le C^p \prod_{i=1}^p |s_i-s_{i-1} |_+^{-1/2} 
E \Big[ \prod_{i=1}^p |\tilde{S}_{t_i}-\tilde{S}_{t_{i-1}}|_+^{-1} \Big] \,.
\end{equation}
Suppose that $t_{\sigma(1)} \le \cdots \le t_{\sigma(p)}$ for some 
permutation $\sigma$ of $\{1,\ldots,p\}$. Then, conditioning on $(\tilde{S}_i, i \le t_{\sigma(p-1)})$, 
we get by (\ref{ine2}) and the independence of increments, that when $\sigma(p)=\ell$, 
\begin{align*}
E\Big[ \prod_{i=1}^p  |\tilde{S}_{t_i}-\tilde{S}_{t_{i-1}} |_+^{-1} \Big] 
\le C  |t_{\sigma(p)}-t_{\sigma(p-1)} |_+^{-1/2} 
E \Big[ \prod_{i=1}^{\ell-1} |\tilde{S}_{t_i}-\tilde{S}_{t_{i-1}} |_+^{-1} 
|\tilde{S}_{t_{\ell+1}}-\tilde{S}_{t_{\ell-1}}|_+^{-1} \prod_{i=\ell+2}^p 
|\tilde{S}_{t_i}-\tilde{S}_{t_{i-1}} |_+^{-1}
%\substack{1\le i \le \ell-1,\\ \ell+2 \le i \le p}}  |\tilde{S}_{t_i}-\tilde{S}_{t_{i-1}} |_+^{-1} 
%|\tilde{S}_{t_{\ell+1}}-\tilde{S}_{t_{\ell-1}}|_+^{-1}
\Big].
\end{align*} 
\blue{Any permutation} $\sigma$ \blue{of $\{1,\ldots,p\}$ with $\sigma(p)=\ell$ must be} 
a bijection from $\{1,\dots, p-1 \}$ to $\{ 1,\ldots,\ell-1,\ell+1,\ldots, p\}$. \blue{We can thus
further bound the right-side of the preceding inequality,}
inductively according to the values of $\sigma(j)$, for $j=p-1,\ldots,1$, 
\blue{and thereby arrive at}
\begin{equation}\label{eq:bdG2}
E\Big[ \prod_{i=1}^p  |\tilde{S}_{t_i}-\tilde{S}_{t_{i-1}} |_+^{-1} \Big]  \le C^p  \prod_{j=1}^p |t_{\sigma(j)}-t_{\sigma(j-1)}|_+^{-1/2} \,, \qquad \blue{\sigma(0):=0} \,.
\end{equation}
Combining \eqref{eq:bdG1} and \eqref{eq:bdG2}, we conclude that for any non-decreasing $(s_i)$ 
and $(t_{\sigma(j)})$,
\begin{align}\label{mult}
E \Big[ \prod_{i=1}^p |S_{s_i} -\tilde{S}_{t_i}|_+^{-2} \Big]
\le C^{2p} \prod_{i=1}^p  |s_i-s_{i-1} |_+^{-1/2}  \prod_{j=1}^p |t_{\sigma(j)}-t_{\sigma(j-1)} |_+^{-1/2} \,.
\end{align}
\blue{We next bound $E[X_n^p]$ by \eqref{gre1} and enumerate over all words 
$(s_i)$ and $(t_i)$ of length $p$ with symbols from $[1,n]$, according to the
numbers $k$ and $k'$ of distinct symbols in each word. Recalling that $|0|_+=1$ and having} 
at most $k^p$ words of length $p$ composed of \blue{given (fixed)} $k$ distinguished symbols,  we 
\blue{thus} deduce from \eqref{mult} that for any $n,p \in \N$, 
\begin{align}\label{eq:bdXn}
&E[X_n^p] 
\le \frac{1}{n^p} C^p E \blue{\bigg[\sum_{s_i,t_i \in [1,n]} \prod_{i=1}^p |S_{s_i}-\tilde{S}_{t_i} |_+^{-2}  \bigg]}
\le C^{2p} \Big[\sum_{k=1}^{p} k^p n^{-(p-k)/2} J_{k,n} \Big]^2, \\
&J_{k,n}:=\frac{1}{n^k} \sum_{1 \le s_1 < \cdots < s_k \le n} \prod_{i=1}^k  \Big(\frac{s_i}{n}-\frac{s_{i-1}}{n}\Big)^{-1/2} \,. \nonumber
\end{align}
Considering $\theta_i=s_i/n$, we see that $\{J_{k,n}\}$ are for each $k \in \N$, the Riemann sums of
% the integral  
\begin{align}\label{eq:int-bd}
 J_k := \int_{0=\theta_0 < \theta_1 < \ldots < \theta_k \le 1}  \prod_{i=1}^k  (\theta_i-\theta_{i-1})^{-1/2} 
 d\theta_1 \cdots d\theta_k  \le C^k (k!)^{-1/2}  
\end{align}
(for the inequality, see e.g. \cite[proof of Lemma 2]{LG94}). 
Note that $J_k = (Q^k {\bf 1})(0)$ for
the positive linear operator $(Q f)(x) =\int_0^{1-x} y^{-1/2} f(x+y) dy$ on $C([0,1])$.
Setting $(y)_n=\lceil y n \rceil/n$,  we have that 
$J_{k,n} = (Q_n^k {\bf 1})(0)$ for the positive linear operators 
$(Q_n f)(x)=\int_0^{1-x} (y)_n^{-1/2} f((x)_n+(y)_n) dy$. It is easy to see that
$(Q_n f) \le (Q f)$ are both non-increasing whenever $f(\cdot)$ is non-increasing.
By induction on $k  \ge 0$, we thus have that $Q_n^k {\bf 1} \le Q^k {\bf 1}$, pointwise, 
so in particular $J_{k,n} \le J_k$
for any $k, n \in \N$. Further, $k! \ge (k/e)^k$ and $J_{k,n}=0$ unless $k \le n$. Hence,
in view of \eqref{eq:bdXn} and \eqref{eq:int-bd} we find that 
\[
E[X_n^p] \le C^{2p} \Big[\sum_{k=1}^{p} k^p k^{-(p-k)/2} J_{k} \Big]^2 \le C^{2p} 
\Big[ \sum_{k=1}^p k^{(p+k)/2} C^k (k/e)^{-k/2} \Big]^2 \le
p^2 C^{4p} e^p p^{p}  \,,
\]
and the uniform moment bounds \eqref{exptail*} on $X_n$ follow. 
%Turning to the moments of $Y$, note that
%\[
%E[Y^p] \le p!  
%\int_{0 \le s_1 \le \ldots \le s_p \le 1} 
%\int_{0 \le t_1 \le \ldots \le t_p \le 1} 
%E\Big[ \sum_{\sigma} \prod_{i=1}^p |\beta_{s_i}-\tilde{\beta}_{t_{\sigma(i)}}|^{-2} \Big]
%ds_1 \cdots ds_p dt_1 \cdots dt_p \,.
%\]
%In view of \eqref{eq:int-bd}, it thus suffices to show that for some $C<\infty$ and any $p \in \N$,
%\begin{equation}\label{eq:bdY}
%E\Big[ \sum_{\sigma} \prod_{i=1}^p |\beta_{s_i}-\tilde{\beta}_{t_{\sigma(i)}}|^{-2} \Big] 
%\le C^p p! \prod_{i=1}^p  (s_i-s_{i-1})^{-1/2} \prod_{j=1}^p (t_j-t_{j-1})^{-1/2} \,.
%\end{equation}
%We can assume \abbr{wlog} that $\{\beta_{s_i},\tilde{\beta}_{t_j}, i,j \le p\}$ are all
%distinct, as are the times $\{s_i\}_{i \le p}$ and $\{t_j\}_{j \le p}$. 
%Thus, following our proof of \eqref{mult}, while using \eqref{bine1}--\eqref{bine3}
%instead of \eqref{ine1}--\eqref{ine3}, yields \eqref{eq:bdY}, thereby completing the proof of the lemma.
\end{proof}

For any interval $I$, consider the range
$\R_I=\{S_i\}_{i \in I}$ and $\tilde{\R}_I=\{\tilde{S}_i\}_{i \in I}$ of independent \abbr{srw}-s.
\blue{Fixing $\alpha>0$,} let $n_\alpha:=n(\log n)^{-\alpha}$ and 
denoting by $\hat{P}$ and $\hat{\tau}_A$ the probability and the hitting time 
to a set $A$ by another independent \abbr{srw} $(\hat{S}_i)$, \blue{set} for 
$i\in [n_{2\alpha}, n-n_{2\alpha}]$, 
\begin{align}\label{def:gna}
g_{n,\alpha}(i) &:=  \quad 1_{\{S_i \notin \R(i,i+n_{2\alpha}]\}} 
\hat{P}^{S_i}(\hat{\tau}_{\R(i-n_{2\alpha}, i+n_{2\alpha}]} =\infty) ,
\end{align}
with $g_{n,\alpha}(i)=1$ for $i \in [0,n_{2\alpha}) \cup (n-n_{2\alpha},n]$,  
and $\tilde{g}_{n,\alpha}(i)$ defined analogously for the \abbr{srw} $(\tilde{S}_i)$.
\blue{In particular, for any $i \in [n_{2\alpha}, n - n_{2 \alpha}]$,}
\begin{align}
\blue{\ovg}_{n,\alpha} &:= E[
1_{\{S_{n_{2\alpha} } \notin \R(n_{2\alpha},2n_{2\alpha}] \}}
\hat{P}^{S_{n_{2\alpha} }}(\hat{\tau}_{\R_{2n_{2\alpha}}} =\infty)] = E[g_{n,\alpha}(i)] \,.
 \label{def:bgna}
\end{align}
The next lemma, whose proof is deferred to the end of this section, allows us to complete
the proof of the limsup-\abbr{LIL} lower bound for \abbr{srw} on $\Z^4$, \blue{by comparing} 
 \begin{align*}
Y_{n,m} := \sum_{\substack{i \in [1,n],\\ \ell \in [1,m]}}  g_{n,\alpha}(i)  G(S_i,\tilde{S}_\ell )\tilde{g}_{m,\alpha}(\ell) ,
\end{align*}
\blue{with the simpler to analyze}
\[
\underline{Y}_{n,m} := \ovg_{n,\alpha} \ovg_{m,\alpha}\sum_{\substack{i \in [1,n],\\ \ell \in [1,m]}} G(S_i,\tilde{S}_\ell )\,,
\]
\blue{whose moments we bound in Lemma \ref{extail:lem} (for $m=n$, see also 
\eqref{ee1} for the decay of $\ovg_{n,\alpha}$).}
\begin{lem}\label{moment*} 
We have that $EY_{n,m} \le  C \sqrt{n m}$ for some $C$ finite and all $m, n \in \N$, 
and if $n^{\epsilon} \le m \le n$, then for some $C=C(\epsilon,\alpha)<\infty$ and 
any $\epsilon>0$, $\alpha > 4$, 
 \begin{align}\label{moment*+}
E Y_{n,m} &\le C \sqrt{nm}(\log n)^{-2} \,, \\
E \Big[ (Y_{n,m}- \underline{Y}_{n,m})^2 \Big] 
&\le C nm (\log n)^{-\alpha/2}\,.
\label{moment*++}
\end{align}
\end{lem}
\begin{proof}[\blue{Proof of the limsup-\abbr{LIL} lower bound}]
Equipped with Lemma \ref{extail:lem} and Lemma \ref{moment*}, we derive the 
limsup-\abbr{LIL} lower bound for \abbr{srw} on $\Z^4$, \blue{and} with
$n \mapsto \crn$ having a similar structure as $|\R_n|$ for the \abbr{srw} on $\Z^2$,
we adapt the proof in \cite[Prop. 4.4]{BK02} of the limsup-\abbr{lil} lower-bound for the latter sequence.
Specifically, set $p=[(-\kappa + \log_3 n)/(\log 2)]$ with $\kappa<\infty$ large and 
$k = 2^p$ (so $k=\gamma \log_2 n$ for small $\gamma \le e^{-\kappa}$).
Centering both sides of \eqref{eq:decomp4} for $n_j=j m$, $m=n/k$ (assumed for simplicity 
to be integer), we have that for i.i.d. $U_j:=R_{(j-1)m,jm}$ and 
the non-random $\varphi_n := E R_n$,
\begin{align}\label{eq:decom}
\overline{R}_n
=k \varphi_{n/k} - \varphi_n + \sum_{j=1}^k \overline{U}_j -\sum_{j=1}^{k-1} V_{0,jm,(j+1)m} \,.
\end{align}
Further, denoting by $\theta$ the time shift $S_i \mapsto S_{i+1}$, we set 
\begin{align*}
\chi(A,B) & := \sum_{y\in A} \sum_{z\in B} P^y(\tau_{A\cup B}=\infty) G(y,z) P^z(\tau_{B}=\infty), 
\end{align*}
and for all $i\in (0,n]$,
\begin{align}\label{def:hni}
h_n(i) &: = 1_{\{S_i \notin \R(i,n] \}} \hat{P}^{S_i}(\hat{\tau}_{\R_n} =\infty) \le g_{n, \alpha}(i) 
\end{align}
(see \eqref{def:gna}), recalling from \eqref{def:Rab} and \cite[Prop. 1.6]{As5} that
\begin{align} 
0 \le V_{0,j m,(j+1)m} & \le \chi (\R_{jm}, \R(jm,(j+1)m])+ \chi (\R(jm,(j+1)m],\R_{jm}) \nonumber \\
% =&  \sum_{y \in \R_{jm}}  \sum_{z \in \R(jm,(j+1)m] } 
%P^y(\tau_{ \R_{jm} \cup \R(jm,(j+1)m]}=\infty ) \; G(y,z ) \; P^z(\tau_{ \R(jm,(j+1)m]}=\infty )\\
%+ &   \sum_{y \in \R_{jm}}  \sum_{z \in \R(jm,(j+1)m] } 
% P^y(\tau_{ \R_{jm}}=\infty ) \; G(y,z ) \; P^z(\tau_{ \R_{jm} \cup  \R(jm,(j+1)m]} =\infty ) \\
\le &  2 \sum_{y \in \R_{jm}}  \sum_{z \in \R(jm,(j+1)m] } 
P^y(\tau_{ \R_{jm}}=\infty ) \; G(y,z) \; P^z(\tau_{ \R(jm,(j+1)m]}=\infty ) \nonumber \\
= & 2 \sum_{i=1}^{jm}  \sum_{\ell=1}^{m} 
h_{jm}(i) \; G(S_i,S_{jm+\ell}) \; h_m (\ell)\circ \theta_{jm} \nonumber \\
\le & 2 \sum_{i=1}^{jm}  \sum_{\ell=1}^{m} 
g_{jm,\alpha}(i) \; G(S_i,S_{jm+\ell}) \; g_{m,\alpha}(\ell)\circ \theta_{jm}
:= 2 W_j \,.
\label{eq:Wjbd}
\end{align}
Setting in addition 
\begin{align}\label{def:Wu}
%& \hat{W}_j:=  \sum_{i=1}^{(j-1)M}  \sum_{\ell=jM+1}^{(j+1)M} 
%g_{jM,\alpha}(i-(j-1)M)\circ \theta_{(j-1)M} \; G(S_i,S_\ell ) \; g_{M,\alpha}(\ell-jM)\circ \theta_{jM}, \\
& \underline{W}_j:=  \ovg_{jm,\alpha}  \ovg_{m,\alpha} \sum_{i=0}^{m-1}  
\sum_{\ell=1}^{m} G(S_{jm-i},S_{jm+\ell} ), \\
& \underline{\hat{W}}_j:=  \ovg_{jm,\alpha} \ovg_{m,\alpha} \sum_{i=m}^{jm-1}  
\sum_{\ell=1}^{m} G(S_{j m-i} ,S_{jm+\ell}),
\label{def:Wh}
\end{align}
we see that for any fixed $j$, 
\begin{align}\label{eq:Y-W}
 W_j- \underline{W}_j-\underline{\hat{W}}_j \stackrel{d}{=} Y_{jm,m}-\underline{Y}_{jm,m} \,. 
\end{align}
\blue{Next, following \cite{As5}, we let 
\[
\chi_n(i,j):=\chi(\R_n^{(i,2j-1)}, \R_n^{(i,2j)})+\chi(\R_n^{(i,2j)}, \R_n^{(i,2j-1)})
\quad \text{ for }  \quad \R_n^{(i,j)}:=\R[(j-1)2^{-i}n, j2^{-i}n]\,,
\]
and take} the expected value in \cite[Prop. 2.3]{As5}, \blue{to arrive at}
\begin{align}\label{eq:23As5**}
k \varphi_{n/k} - \varphi_n 
% = \sum_{j=1}^k EU_j -ER_n
= \sum_{i=1}^p \sum_{j=1}^{2^{i-1}} E[ \chi_n(i,j)]-\sum_{i=1}^p \sum_{j=1}^{2^{i-1}} E[ \epsilon_n(i,j)] \,.
\end{align}
\blue{In \cite[Prop. 2.3]{As5} it is shown that for $p$ fixed, the non-negative right-most sum is at most 
$C (\log n)^2$. The same applies for our choice of growing $p=p(n)$. Indeed,}
as each $\epsilon_n(i,j)$ is bounded by the intersection of the ranges 
of two independent \abbr{srw}-s of length $n/2^i$, we have that $\max_{i,j} E[ \epsilon_n(i,j)] \le \log n$
 (see \cite[Section 3.4]{LA91}), and with at most $2^p \le C \log n$ such terms, we conclude that 
\begin{align}\label{eq:23As5}
\sum_{i=1}^p \sum_{j=1}^{2^{i-1}} E[ \chi_n(i,j)]- C(\log n)^2
\le k \varphi_{n/k}- \varphi_n
\le \sum_{i=1}^p \sum_{j=1}^{2^{i-1}} E[\chi_n(i,j)]. 
\end{align}
\blue{Recall} \cite[Prop. 6.1]{As5}, that 
\[
 \lim_{n \to \infty} \frac{2 (\log n)^2}{\pi^4 n}  E[\chi_{n}(1,1)] 
 %= E[\alpha(A_1^1)] 
  = \int_{A_1^1} E [ G_\beta (\beta_s, \beta_t) ] ds dt \,, \quad \blue{ \text{for} \quad 
  A_1^1 = [0,2^{-1}) \times (2^{-1},1].}
\] 
\blue{Now,} using \eqref{eq:green4} \blue{(at $d=4$), with}   
$\int_{A_1^1}  |t-s|^{-1} dsdt=\log 2$ and $E |\beta_1|^{-2}=\frac{1}{2}$, \blue{we see that}
\[
\int_{A_1^1} E [ G_\beta (\beta_s, \beta_t) ] ds dt 
= \frac{1}{2\pi^2} E[|\beta_1|^{-2}] \int_{A_1^1} |t - s|^{-1} dt ds = \frac{\log 2}{4 \pi^2} \,.
\]
\blue{By definition it follows that} $E[\chi_n(i,j)]=E[\chi_{n'_i} (1,1)]$ for $n'_i : =n 2^{1-i}$ 
and all $i$, $j$. \blue{So, in view of the 
 %(see \cite[Prop. 2.3]{As5}), 
preceding, we deduce that for any $p = o(\log n)$,}  as $n \to \infty$,
\begin{align*}  
 \blue{\max_{i \le p} \Big\{ \big|} \frac{2 (\log n_i')^2}{\pi^4 n_i'}  E[\chi_{n'_i}(1,1)]  - \frac{\log 2}{4 \pi^2} 
 \blue{\big| \Big\} \to 0} \,.
\end{align*} 
\blue{Recall also} (see \eqref{h4:def}), that 
\[
p = (1+o(1)) \frac{\log_3 n}{\log 2} \,, \qquad 
h_4(n)= \frac{\pi^2}{8} \frac{ n \log_3 n}{ (\log n)^2} \,.
\]
It thus follows that 
\begin{align} \label{eq:chi-exp}
 \sum_{i=1}^p \sum_{j=1}^{2^{i-1}} E[\chi_n(i,j)]
 =(1+o(1)) p  \frac{\log 2}{4 \pi^2} \frac{\pi^4 n} {2 (\log  n)^2} = (1+o(1)) h_4(n) \,,
\end{align}
and combining \eqref{eq:23As5} and \eqref{eq:chi-exp} we arrive at
\begin{align}\label{log3}
k \varphi_{n/k}- \varphi_n = (1+o(1)) h_4(n) .
\end{align}

In view of \eqref{eq:decom}, \eqref{eq:Wjbd} and \eqref{log3}, we get our limsup-\abbr{lil} lower bound, precisely
as in \cite[Proof of Prop. 4.4]{BK02}, once we find for any $\varepsilon>0$, 
constants $c_1<\infty$, $c_2>0$ and for any \blue{$k=2^p$, $m=n/k$, $p$} as above, \blue{some} events $G_k$ 
such that $P(G_k) \ge \frac{1}{4} c_2^k$ and
\begin{align}\label{def:Bj}
G_k & \subseteq \bigcap_{j=1}^k \Big\{ \overline{U}_j \ge - \frac{c_1 m}{(\log m)^2} \Big\}  \blue{\eqqcolon} \bigcap_{j=1}^k B_j \,, \\
\label{bd:sumHj}
G_k & \blue{\subseteq \Big\{} \sum_{j=1}^{k-1} W_j \le 3  \varepsilon \frac{n \log_3 n}{(\log m)^2} \Big\} \,.
\end{align}
To this end, it suffices to construct events $F_k$ such that for some $c_3 < \infty$,
\begin{equation}\label{eq:Fk}
P(F_k) \ge c_2^k\,, \qquad \qquad F_k \subseteq 
\Big\{ \max_{j < k} \{ \underline{\hat{W}}_j \} \le \frac{c_3 m}{(\log m)^2} \Big\} \blue{\bigcap_{j=1}^k B_j} \,.
\end{equation}
Indeed, we shall see that $P(\C_i) \le \frac{1}{4} c_2^k$ for 
%\blue{$0 \le i \le 3$}, 
$k \le \gamma \log_2 n$, 
$\gamma>0$ small and $n \to \infty$, where 
\begin{align*}
\C_1 := \bigg\{\sum_{j\text{ odd}} \underline{W}_j
> & \varepsilon \frac{n \log_3 n}{(\log m)^2} \bigg\},
\qquad \qquad \quad 
\C_2 :=\bigg\{\sum_{j\text{ even}} \underline{W}_j > \varepsilon \frac{n \log_3 n}{(\log m)^2} \bigg\},\\
\C_3 & := \bigg\{ \max_{j< k} \{ W_j - \underline{W}_j-\underline{\hat{W}}_j \} > \frac{m}{(\log m)^2} 
\bigg\}\,.
\end{align*}
Taking $G_k := F_k \cap_{i=1}^3 \C_i^c$ this would imply that $P(G_k) \ge \frac{1}{4} c_2^k$
for large $k$, and it is easy to check, as stated, that \blue{both \eqref{def:Bj} and} \eqref{bd:sumHj} then hold \blue{for} such $G_k$.

Next, utilizing the union bound, \eqref{eq:Y-W}, Markov's inequality and 
\eqref{moment*++}, we get that  
\begin{align*}
P(\C_3) & \le 
\sum_{j=1}^{k-1} P\bigg( W_j - \underline{W}_j -\underline{\hat{W}}_j \ge \frac{m}{(\log m)^2}\bigg)
= \sum_{j=1}^{k-1} P\bigg( Y_{jm,m} - \underline{Y}_{jm,m} \ge \frac{m}{(\log m)^2} \bigg)
\\
& \le  \frac{(\log m)^4}{m^2} \sum_{j=1}^{k-1} E \Big[ (Y_{jm,m}- \underline{Y}_{jm,m})^2 \Big]
\le  C (\log m)^{4-\alpha/2} \sum_{j=1}^{k-1} j \le C k^2 (\log m)^{4-\alpha/2} \,. 
\end{align*} 
In particular, for $\alpha>8 + 2\gamma \log (1/c_2)$, $k$ as above and $n=m k \ge n_0$, 
the preceding bound implies that $P(\C_3) \le \frac{1}{4} c_2^k$. Turning to deal with $\underline{W}_j$
and $\underline{\hat{W}}_j$, upon expressing \eqref{def:bgna} via the independent \abbr{srw}-s
$\hat{S}_i$, $S^+_i := S_{n'+i}-S_{n'}$ and $S^-_i := S_{n'-i}-S_{n'}$, at $n'=n_{2 \alpha}$, 
it follows from \cite[(1.4)]{As3} that 
\begin{align}\label{ee1}
\ovg_{n,\alpha} 
& = E[
1_{\{0 \notin \R^+_{n'}\}}
\hat{P}^{0}(\hat{\tau}_{\R^+_{n'} \cup \R^-([0,n'-1])} =\infty)] \nonumber \\
& = P(0 \notin \R^+_{n'}, \hat{R}_\infty \cap (\R^+_{n'} \cup \R^-([0,n'-1]) ) = \emptyset) 
= (1+o(1))\frac{\pi^2}{8} (\log n)^{-1}
\end{align} 
(note that $\log n' = (1+o(1)) \log n$). In view of \eqref{def:Wu}, we note that 
$\{m^{-1} (\ovg_{m,\alpha})^{-2} \underline{W}_j \}$ are, for 
odd $j$, independent copies of $X_m$ of Lemma \ref{extail:lem} 
(except for now including also $\ell=0$ in $X_m$). It thus follows
from \eqref{exptail} and \eqref{ee1} that for some $c>0$, and all $k$, $m$,
\[
E\Big[ \exp\Big( c \,m^{-1} (\log m)^2 \sum_{j\text{ odd}} \underline{W}_j\Big) \Big] \le \exp(k/c) \,.
\]
Hence, for $n = m k$, one has by Markov's inequality that 
\begin{align*}
P(\C_1 ) \le \exp(-\varepsilon c k \log_3 n) \exp(k/c) 
\end{align*} 
decays as $n \to \infty$, faster than $\frac{1}{4} c_2^k$. By the same reasoning, this 
applies also for $\C_2$. 

Finally, in view of \eqref{gre1}, \eqref{def:Wh} and \eqref{ee1}, for some $C<\infty$ and any $m$, $j$,
\begin{equation}\label{eq:wh-bd}
\underline{\hat{W}}_j \le \frac{C m^2}{(\log m)^2} \sum_{s=1}^{j-1} {\rm dist}(\R((s-1) m,s m]), \R((jm,(j+1)m]))^{-2} \,.
% \sum_{\ell=1}^{m} |S_{j m-i}-S_{jm+\ell}|^{-2} ,
\end{equation}
As  in \cite[Proof of Prop. 4.4]{BK02}, fixing
a unit vector $\bf u$, we let $F_j:=\cap_{i=1}^j (A_i \cap B_i)$, \blue{while taking here $B_i$ of
\eqref{def:Bj}, and}
\begin{align*}
A_i :=\Big\{ S_{i m}\subset \BB(i \sqrt{m} \, {\bf u}, \sqrt{m}/8),  \,\,
\R((i-1)m,im]) \subset \BB\Big( (i-\frac{1}{2}) \sqrt{m} \, {\bf u}, \frac{3}{4} \sqrt{m}\Big) \Big\}.
\end{align*}
The event $F_k$ guarantees that \blue{for any $s<j$,} the distance of $\R((s-1) m,s m])$ from
$\R((jm,(j+1)m])$ be at least $(j-s-1/2)\sqrt{m}$, so \eqref{eq:wh-bd} results
with the \abbr{rhs} of \eqref{eq:Fk} (for $c_3 = C \sum (r-1/2)^{-2}$ finite).
As for the \abbr{lhs} of \eqref{eq:Fk}, \blue{recall \cite[Thm. 1.2]{As5} that 
$\{\frac{(\log m)^2}{m} \overline{R}_m\}$ converges in law, hence is a uniformly tight sequence.
In particular,} for any $\delta>0$ there exists $c_1=c_1(\delta)$ finite, such that
$P(B_1^c) \le \delta$ \blue{for $B_1$ of \eqref{def:Bj},} uniformly in $m$. Further, $\{B_j, j \ge 1\}$ are i.i.d. and 
by the invariance principle, there exists $c_2>0$ such that,
\[
\lim_{m \to \infty} \inf_{S_0 \in \BB(0, \sqrt{m}/8)} \{ P^{S_0} (A_1) \} =
\inf_{\beta_0 \in \BB(0,1/4)} \{ 
P(\, |\beta_1-2 {\bf u}| < 1/4, \sup_{t \in [0,1]} |\beta_t - {\bf u} | < 3/2) \} \ge 2 c_2 \,.
\]
As $F_j$ is measurable on $\sigma(S_i, i \le jm)$, by the Markov property of the \abbr{srw} 
and its independent, stationary increments, for any $j \ge 1$, 
\[
P(A_j  \cap B_j | F_{j-1} ) \ge \inf_{S_0 \in \BB(0, \sqrt{m}/8)} \{  P^{S_0}(A_1) \} - P(B_1^c) \ge c_2 \,,
\]
provided $\delta>0$ is small enough and $m \ge m_0$ finite, thereby establishing the \abbr{lhs}
of \eqref{eq:Fk}.
\end{proof}

\medskip
We conclude this sub-section by proving Lemma \ref{moment*}. \blue{To this end, note that}
\begin{align}\label{eq:gG-dec}
E \Big[ (Y_{n,m}-\underline{Y}_{n,m})^2 \Big]
 = & \frac{1}{2} \sum_\pi \sum_{\substack{(i_1,i_2)\in [1,n]^2,\\(\ell_1,\ell_2) \in [1,m]^2}}  E [g\cdot \G]\,,
\end{align}
where the sum is over the two permutations $\pi$ of $\{1,2\}$ and 
\begin{align}\label{dfn:g}
g &:=\Big( g_{n,\alpha}(i_1)\tilde{g}_{m,\alpha}(\ell_{\pi_1}) -\ovg_{n,\alpha} \ovg_{m,\alpha} \Big) 
\Big( g_{n,\alpha}(i_2)  \tilde{g}_{m,\alpha}(\ell_{\pi_2}) -\ovg_{n,\alpha} \ovg_{m,\alpha} \Big) 
\,,\\
\G &:= G(S_{i_1},\tilde{S}_{\ell_{\pi_{1}}}) G(S_{i_2},\tilde{S}_{\ell_{\pi_{2}}}) \,.
\label{dfn:sG}
\end{align}
Setting for $\alpha>0$,
\[
I_\alpha (n):=[n_\alpha, n-n_\alpha]^2 \cap \{(i,j): j-i \ge n_{\alpha} \} \,,
\]
\blue{the key to \eqref{moment*++} is to bound $|E[g \cdot \G]|/E[\G]$ uniformly 
over $(i_1,i_2) \in I_\alpha(n)$ and $(\ell_1,\ell_2) \in I_\alpha(m)$.} 
For $(i_1,i_2) \in I_\alpha(n)$, $n'=n_{2\alpha}$, \blue{we will show that 
the contribution from the complement of 
\begin{align*}
H^{(n)}_{i_1} &:=\{|S_{i_1}-S_{i_1-n'}| \le \sqrt{n'} \log n' \} \,,\\
H_{i_1,i_2}^{(n)}&:=H^{(n)}_{i_1} \cap \{ |S_{i_2}-S_{i_2-n'}+S_{{i_1}+n'}-S_{i_1}| \le \sqrt{n'} \log n'\}\,,
\end{align*}
is negligible and the same applies for the analogous events} 
$\widetilde{H}^{(m)}_{\ell_1}$, $\widetilde{H}^{(m)}_{\ell_1,\ell_2}$ 
defined in terms of the \abbr{srw} $(\tilde{S}_\ell)$, 
$(\ell_1,\ell_2)\in I_\alpha(m)$ and $m'=m_{2\alpha}$. \blue{Further, from
\eqref{def:bgna} it follows that $E[g]=0$ for such $(i_1,i_2)$ and $(\ell_1,\ell_2)$, allowing 
us to instead bound (in terms of $E[\G]$), the value of 
\[
\big| E[ g \,  \G 1_{H^{(n)}_{i_1,i_2} \cap \widetilde{H}^{(m)}_{\ell_1,\ell_2}}] -
 E[ g 1_{H^{(n)}_{i_1,i_2} \cap \widetilde{H}^{(m)}_{\ell_1,\ell_2}}] E[\G]\big|\,.
 \]
Decomposing the events $H^{(n)}_{i_1}$ and $H^{(n)}_{i_1,i_2}$ as}
\begin{align*}
H^{(n)}_{i_1} & = \bigcup_{|u| \le \sqrt{n'} \log n'} H^{(n')}_{i_1}(u) \,, \qquad
H_{i_1,i_2}^{(n)} =\bigcup_{|u|,|v| \le \sqrt{n'} \log n'} H^{(n')}_{i_1,i_2}(u,v), \\
H^{(n')}_{i_1} (u) & := \{S_{i_1}-S_{i_1-n'}=u\} \,, \qquad 
H^{(n')}_{i_1,i_2}(u,v) := H^{(n')}_{i_1}(u) \cap 
\{ S_{i_2}-S_{i_2-n'}+S_{{i_1}+n'}-S_{i_1} =v\}\,,
\end{align*}
and such decomposition for 
%$\widetilde{H}^{(m)}_{\ell_1}$ and 
$\widetilde{H}^{(m)}_{\ell_1,\ell_2}$,
\blue{we show in the sequel that given
$H^{(n')}_{i_1,i_2}(u,v) \cap \widetilde{H}^{(m')}_{\ell_1,\ell_2}(\tilde{u},\tilde{v})$,
makes $\G$ independent of $g$, whereby the following estimates shall be utilized.}
\begin{lem}\label{moment-}
Fix $\alpha > 2$, $\epsilon >0$ and a permutation $\pi$ of $\{1,2\}$. Then,  for $n^\epsilon \le m \le n$,
\begin{align}\label{bd:F1}
F_1(u,\tilde{u}) := & E[G(S_{i_1},\tilde{S}_{\ell_1})\, |\, H^{(n')}_{i_1}(u) \cap \blue{\widetilde{H}}^{(m')}_{\ell_1}(\tilde{u}) ] =
(1+O((\log n)^{2-\alpha})) E[G(S_{i_1},\tilde{S}_{\ell_1})] \,, \\
F_2(u,v,\tilde{u},\tilde{v}) := & E[G(S_{i_1},\tilde{S}_{\ell_{\pi_{1}}}) G(S_{i_2},\tilde{S}_{\ell_{\pi_{2}}}) | 
H^{(n')}_{i_1,i_2}(u,v) \cap \widetilde{H}^{(m')}_{\ell_1,\ell_2}(\tilde{u},\tilde{v})] \nonumber \\
 = & (1+O((\log n)^{2-\alpha}))E[G(S_{i_1},\tilde{S}_{\ell_{\pi_{1}}})G(S_{i_2},\tilde{S}_{\ell_{\pi_{2}}})]\,,
 \label{bd:F2}
\end{align}
uniformly over $(i_1,i_2) \in I_\alpha(n)$, 
$(\ell_1,\ell_2) \in I_\alpha(m)$,  $|u|,|v| \le \sqrt{n'} \log n'$  
and $|\tilde{u}|,|\tilde{v}| \le \sqrt{m'} \log m'$.
\end{lem}
\begin{proof} 
For $(i_1,i_2) \in I_\alpha(n)$, the law of $(S_{i_1},S_{i_2})$ given 
$H^{(n')}_{i_1,i_2}(u,v)$ is as $(u+S^{(1)}_{i_1-n'},u+v+S^{(1)}_{i_2-3n'})$ 
for an independent \abbr{srw} $S^{(1)}_i$. Similarly, when $(\ell_1,\ell_2) \in I_\alpha(m)$, 
the law of $(\tilde{S}_{\ell_1},\tilde{S}_{\ell_2})$ given $\widetilde{H}^{(m')}_{\ell_1,\ell_2}(\tilde{u},\tilde{v})$
is as
$(\tilde{u}+\tilde{S}^{(1)}_{\ell_1-m'},\tilde{u}+\tilde{v}+\tilde{S}^{(1)}_{\ell_2-3m'})$.
Consequently,
\begin{align*}
F_1(u,\tilde{u}) &= E[G(u+S_{i_1-n'},\tilde{u}+\tilde{S}_{\ell_1-m'})] \,, \\
F_2(u,v,\tilde{u},\tilde{v}) &= 
\left\{\begin{array}{ll}
E[G(u+S_{i_1-n'},\tilde{u}+\tilde{S}_{\ell_1-m'})
G(u+v+S_{i_2-3n'},\tilde{u}+\tilde{v}+\tilde{S}_{\ell_2-3m'})\big] , & \;\textrm{if}\;\;  \pi_1=1, \\
E[G(u+S_{i_1-n'},\tilde{u}+\tilde{v}+\tilde{S}_{\ell_2-3m'})
G(u+v+S_{i_2-3n'},\tilde{u}+\tilde{S}_{\ell_1-m'})\big], & \;\textrm{if}\;\; \pi_1=2.
\end{array}\right.
\end{align*}
Note that for some $C=C(\epsilon)$ finite, $c=c(\epsilon)>0$, any $m \ge n^{\epsilon}$ and 
$(\ell_1,\ell_2)$, 
\begin{equation}\label{eq:rough-tail}
P((H^{(m)}_{\ell_1})^c) \le P( (H^{(m)}_{\ell_1,\ell_2})^c ) \le 
2 P(|S_{2m'}| > \sqrt{m'} \log m') \le C e^{- c (\log n)^2} \,,
\end{equation}
with the same bound applying also for $P((H^{(n)}_{i_1,i_2})^c)$. 
Now, by \eqref{gre1} and \eqref{mult} (at $p=1,2$), 
\begin{align}\label{eq:bd-G}
E[G(S_{i},\tilde{S}_{\ell})] &\le C i^{-1/2} \ell^{-1/2} \,, \\
E[G(S_{i_1},\tilde{S}_{\ell_{\pi_{1}}})G(S_{i_2},\tilde{S}_{\ell_{\pi_{2}}})] &\le 
C i_1^{-1/2}(i_2-i_1)^{-1/2} \ell_1^{-1/2}(\ell_2-\ell_1)^{-1/2} \,,
\label{eq:bd-cG}
\end{align}
with the \abbr{lhs} of \eqref{eq:bd-G} and \eqref{eq:bd-cG} being the 
expected values of $F_1(\cdot)$ and $F_2(\cdot)$ according to the joint law of the 
corresponding \abbr{srw} increments (for independent \abbr{srw} $S_i$ and $\tilde{S}_\ell$).
In view of \eqref{eq:rough-tail}, it thus suffices to bound the maximum fluctuation 
of $F_s(\cdot)$, $s=1,2$ over $|u|,|v| \le \sqrt{n'} \log n'$  and 
$|\tilde{u}|,|\tilde{v}| \le \sqrt{m'} \log m'$, by $C (\log n)^{2-\alpha}$ times the \abbr{rhs}
of \eqref{eq:bd-G} and \eqref{eq:bd-cG}, respectively.
To this end, since $F_1(\cdot)$ depends only on $u-\tilde{u}$ and $F_2(\cdot)$ depends only on 
$u-\tilde{u}$ and $v-\tilde{v}$ if $\pi_1=1$,
or on $u+v-\tilde{u}$ and $v+\tilde{v}$ if $\pi_1=2$, we may \abbr{wlog} fix 
$\tilde{u}=\tilde{v}=0$ and consider the maximum fluctuation of
\begin{align*}
F_1(u) &= E[G(u+S_{i_1-n'},\tilde{S}_{\ell_1-m'})] \,,\\
F_2(u,v) &= 
\left\{\begin{array}{ll}
E[G(u+S_{i_1-n'},\tilde{S}_{\ell_1-m'})
G(v+S_{i_2-3n'},\tilde{S}_{\ell_2-3m'})\big] , & \;\textrm{if}\;\;  \pi_1=1, \\
E[G(u+S_{i_1-n'},\tilde{S}_{\ell_2-3m'})
G(v+S_{i_2-3n'},\tilde{S}_{\ell_1-m'})\big], & \;\textrm{if}\;\; \red{\pi_1=2.}
\end{array}\right.
\end{align*}
over $|u|, |v| \le 3 \sqrt{n'} \log n'$. Further, with both $n_{\alpha}/n'$ and
$m_{\alpha}/m'$ diverging (as $(\log n)^\alpha$), it follows that uniformly over $(i_1,i_2) \in I_\alpha(n)$,
$(\ell_1,\ell_2) \in I_\alpha(m)$ and $m \ge n^\epsilon$, the \abbr{rhs} of \eqref{eq:bd-G} 
and \eqref{eq:bd-cG} also bound $F_1(0)$ and $F_2(0,0)$, respectively. Consequently, it 
suffices to show that for some $C$ finite and all $|u|, |v| \le 3 \sqrt{n'} \log n'$,
\begin{align}\label{eq:F1-bd}
|F_1(u)-F_1(0)| &\le C (\log n)^{2-\alpha} F_1(0) \,, \\
|F_2(u,v)-F_2(0,0)| &\le C (\log n)^{2-\alpha} F_2(0,0) \,.
\label{eq:F2-bd}
\end{align}
Turning to this task, since $t_2 := \ell_2-\ell_1 - 2 m' \ge 0$, $t_3 :=i_2 - i_1 -2 n' \ge 0$, 
$G(x,y)=G(y-x,0)$ and $S_i \stackrel{(d)}{=} -S_i$, we can further simplify the functions $F_s(\cdot)$ to be
\begin{align*}
F_1(u) &= E[G(S_{t_1},u)] \,, \\
F_2(u,v) &= 
\left\{\begin{array}{ll}
E[G(S_{t_1},u)
G(S_{t_1+t_2+t_3},v)\big] ,  \; &\textrm{if} \;\;  \pi_1=1, \\
E[G(S_{t_1+t_2},u)
G(S_{t_1+t_3},v)\big],  \; & \textrm{if}\;\; \pi_1=2,
\end{array}\right.
\end{align*}
where $t_1 = i_1+\ell_1-n'-m'$. Denoting by $p_j(u):=P(S_j=u)$, it is easy to check that 
\begin{equation}\label{eq:rep-Fs}
F_1(u) =\sum_{j_0>t_1} p_{j_0}(u) \,, \qquad
F_2(u,v)=\sum_{j_1,j_2} p_{j_1} (u) p_{j_2}(v) \,,
\end{equation}
where the sum is over $j_1>t_1$ and $j_2 > t_1+t_2+t_3$ in case $\pi=1$, and over
$j_1>t_1+t_2$, $j_2 > t_1+t_3$ when $\pi_1=2$. By the local \abbr{clt} for the \abbr{srw} on $\Z^4$, 
we have for some $C<\infty$ that  
\[
\Big| \frac{p_j(u)+p_{j+1}(u)}{p_j(0)+p_{j+1}(0)} - 1 \Big| \le \frac{C |u|^2}{9 j} 
\le C n' (\log n')^2 t_1^{-1}  \le 2 C (\log n)^{2-\alpha} 
\]
throughout the range of parameters considered here (utilizing the fact that 
$j_0,j_1,j_2 \ge t_1 \ge n_\alpha/2$).
The same bound applies with $v$ instead of $u$, and plugging these bounds in \eqref{eq:rep-Fs} 
results with \eqref{eq:F1-bd}-\eqref{eq:F2-bd}, thereby completing the proof of the lemma.
\end{proof}

\begin{proof}[Proof of Lemma \ref{moment*}]
\blue{First observe that $g$ of \eqref{dfn:g} can be written also as}
\begin{align*}
&g =
%\Big( g_{n,\alpha}(i_1)\tilde{g}_{m,\alpha}(\ell_{\pi_1}) -\ovg_{n,\alpha} \ovg_{m,\alpha} \Big)& 
%\Big( g_{n,\alpha}(i_2)  \tilde{g}_{m,\alpha}(\ell_{\pi_2}) -\ovg_{n,\alpha} \ovg_{m,\alpha} \Big) =
g_0-g_1-g_2+\ovg_{n,\alpha}^2 \ovg_{m,\alpha}^2 \,,\\
&g_0 := g_{n,\alpha}(i_1)\tilde{g}_{m,\alpha}(\ell_{\pi_1}) g_{n,\alpha}(i_2)  \tilde{g}_{m,\alpha}(\ell_{\pi_2}) ,\\
& g_1:= g_{n,\alpha}(i_1)\tilde{g}_{m,\alpha}(\ell_{\pi_1}) \ovg_{n,\alpha} \ovg_{m,\alpha}, 
\\
&g_2:=\ovg_{n,\alpha} \ovg_{m,\alpha} g_{n,\alpha}(i_2)\tilde{g}_{m,\alpha}(\ell_{\pi_2}).
\end{align*}
\blue{Proceeding} to show \eqref{moment*++}, note that $g_{n,\alpha}(i)$ 
and $\tilde{g}_{m,\alpha}(\ell)$ are measurable on $\F_{i} := (S_{i+j}-S_i, j \in (-n',n'])$,
$n':=n_{2\alpha}$, and  
$\tilde{\F}_{\ell} := (\tilde{S}_{\ell+j}-\tilde{S}_{\ell}, j \in (-m',m'])$, $m'=m_{2\alpha}$,
respectively.
Further, when $(i_1,i_2) \in I_\alpha(n)$ and $(\ell_1,\ell_2) \in I_\alpha(m)$, 
under the event $H^{(n')}_{i_1,i_2}(u,v) \cap \widetilde{H}^{(m')}_{\ell_1,\ell_2}(\tilde{u},\tilde{v})$ 
the law of $\G$ \blue{of \eqref{dfn:sG}}
given $\F_{i_1}$, $\F_{i_2}$, $\tilde{\F}_{\ell_1}$ and $\tilde{\F}_{\ell_2}$ is determined by 
$H^{(n')}_{i_1,i_2}(u,v) \cap \widetilde{H}^{(m')}_{\ell_1,\ell_2}(\tilde{u},\tilde{v})$. Thus, 
for $s=0,1,2$, and any such $(i_1,i_2)$, $(\ell_1,\ell_2)$,
\begin{align*}
E\big[
g_s \, \G 1_{H^{(n')}_{i_1,i_2}(u,v) \cap \widetilde{H}^{(m')}_{\ell_1,\ell_2}(\tilde{u},\tilde{v})}
\big]
= E\Big[ g_s 1_{H^{(n')}_{i_1,i_2}(u,v) \cap \widetilde{H}^{(m')}_{\ell_1,\ell_2}(\tilde{u},\tilde{v})
}  E[\G|
H^{(n')}_{i_1,i_2}(u,v) \cap \widetilde{H}^{(m')}_{\ell_1,\ell_2}(\tilde{u},\tilde{v})] \Big] \,.
\end{align*} 
With $g_s \ge 0$ and $E[g_s]=\ovg_{n,\alpha}^2 \ovg_{m,\alpha}^2$
whenever $(i_1,i_2) \in I_\alpha(n)$ and $(\ell_1,\ell_2) \in I_\alpha(m)$, 
we get from \eqref{bd:F2} (of Lemma \ref{moment-}), that for some universal $C<\infty$,
 \begin{align}
 \Big| E[ g_s \,  \G 1_{H^{(n)}_{i_1,i_2} \cap \widetilde{H}^{(m)}_{\ell_1,\ell_2}}] &-
 E[ g_s 1_{H^{(n)}_{i_1,i_2} \cap \widetilde{H}^{(m)}_{\ell_1,\ell_2}}] E[ \G ] \Big|
 \nonumber \\
\le & \!\!\!\!\! \sum_{\substack{|u|,|v| \le \sqrt{n'} \log n',\\ |\tilde{u}|, |\tilde{v}| \le \sqrt{m'}\log m'}}  
E\big[ g_s  1_{H^{(n')}_{i_1,i_2}(u,v) \cap \widetilde{H}^{(m')}_{\ell_1,\ell_2}(\tilde{u},\tilde{v})}  \big]
 \Big| E[\G|
H^{(n')}_{i_1,i_2}(u,v) \cap \widetilde{H}^{(m')}_{\ell_1,\ell_2}(\tilde{u},\tilde{v})] - E[\G] \Big| \nonumber \\
\le& C (\log n)^{2-\alpha} \,\ovg_{n,\alpha}^2 \ovg_{m,\alpha}^2 \, E[ \G].
\label{eq:bd-gs-calG}
\end{align}
In addition, with $g_s \in [0,1]$, $\G$ uniformly bounded and $(\log m)/(\log n) \ge \epsilon$, 
we have from \eqref{eq:rough-tail} that 
 \begin{align}
E\big[ g_s \,  \G 1_{(H^{(n)}_{i_1,i_2} \cap \widetilde{H}^{(m)}_{\ell_1,\ell_2})^c}\big]
+ E\big[ g_s 1_{(H^{(n)}_{i_1,i_2} \cap \widetilde{H}^{(m)}_{\ell_1,\ell_2})^c}\big] E[\G] 
&\le C_1(P((H^{(n)}_{i_1,i_2})^c) + P((H^{(m)}_{\ell_1,\ell_2})^c) ) \nonumber \\
&\le  2 C e^{-c(\log n)^2}. \label{eq:tail-gs-calG}
\end{align}
Combining \eqref{eq:bd-gs-calG} and \eqref{eq:tail-gs-calG} for $s=0,1,2$, we thus find for the zero-mean
\[
g =g_0-g_1-g_2+\ovg_{n,\alpha}^2 \ovg_{m,\alpha}^2 
%= 
% \Big( g_{n,\alpha}(i_1)\tilde{g}_{m,\alpha}(\ell_{\pi_1}) -\ovg_{n,\alpha} \ovg_{m,\alpha} \Big) 
% \Big( g_{n,\alpha}(i_2)  \tilde{g}_{m,\alpha}(\ell_{\pi_2}) -\ovg_{n,\alpha} \ovg_{m,\alpha} \Big)
\,,
\] 
that uniformly over $\pi$, $(i_1,i_2) \in I_\alpha(n)$ and $(\ell_1,\ell_2) \in I_\alpha(m)$,
\begin{align}\label{eq:bd-g-cG}
\big| E[ g \cdot \G] \big| \le \red{\sum_{s=0}^2 \big| E[g_s \G] - E[g_s] E[\G] \big| \le}
3 C (\log n)^{2-\alpha} \,\ovg_{n,\alpha}^2 \ovg_{m,\alpha}^2 \, E[ \G] 
+ 6 C e^{-c(\log n)^2} \,.
\end{align}
Further, as $|g| \le 1$, we have from \eqref{gre1} and \eqref{mult} at $p=2$, that for some $C_2<\infty$ 
and uniformly over all $(i_1,i_2) \in [1,n]^2$, $(\ell_1,\ell_2) \in [1,m]^2$,
\begin{align}\label{eq:bd-g-cG-out}
\big| E[ g \cdot \G] \big| \le \, E[ \G] \le C_2 
(i_1 \wedge i_2)^{-1/2} |i_2-i_1|_+^{-1/2} (\ell_1 \wedge \ell_2)^{-1/2} |\ell_2-\ell_1|_+^{-1/2} \,.
\end{align}
Next note that from \eqref{ee1} we have for some $C_3,C<\infty$,
\begin{align*}
\ovg_{n,\alpha}^2 \ovg_{m,\alpha}^2
\sum_\pi \sum_{\substack{(i_1,i_2)\in [1,n]^2,\\(\ell_1,\ell_2) \in [1,m]^2}} E[\G] 
& \le \frac{C_3}{(\log n)^{2} (\log m)^{2}} \sum_{\stackrel{1 \le i_1 \le i_2 \le n}{1 \le \ell_1 \le \ell_2 \le m}} 
i_1^{-1/2} |i_2-i_1|_+^{-1/2} \ell_1^{-1/2} |\ell_2-\ell_1|_+^{-1/2} \\
& \le \frac{C n m}{(\log n)^{2} (\log m)^{2}} \,.
\end{align*}
With the right most term of \eqref{eq:bd-g-cG} being $o(n^{-5})$, it follows that the overall contribution to  
\blue{the right-side of \eqref{eq:gG-dec}}
%\begin{align*}
%E \Big[ (Y_{n,m}-\underline{Y}_{n,m})^2 \Big]
% = & \frac{1}{2} \sum_\pi \sum_{\substack{(i_1,i_2)\in [1,n]^2,\\(\ell_1,\ell_2) \in [1,m]^2}}  E [g\cdot \G]
%\end{align*}
from $i_1,i_2 \in [n_\alpha,n-n_\alpha]$ with $|i_2-i_1| \ge n_\alpha$ and
$\ell_1,\ell_2 \in [m_\alpha,m-m_\alpha]$ with $|\ell_2-\ell_1| \ge m_\alpha$ 
is at most $O(nm (\log n)^{-2-\alpha})$, as specified in \eqref{moment*++}.
Further, the sum over the \abbr{rhs} of \eqref{eq:bd-g-cG-out} under any of the following three restrictions
\[ 
|i_2-i_1| < n_\alpha, \qquad i_1 \wedge i_2 < n_\alpha, \qquad i_1 \vee i_2 > n - n_\alpha \,,
\]
is at most $O(n m \sqrt{n_\alpha/n}) = O(n m (\log n)^{-\alpha/2})$. With
$(\log m)/(\log n) \ge \epsilon$, this applies also when
summing the \abbr{rhs} of \eqref{eq:bd-g-cG-out} under each of the analogous 
restrictions $|\ell_2-\ell_1| < m_\alpha$, $\ell_1 \wedge \ell_2 < m_\alpha$,
or $\ell_1 \vee \ell_2 > m-m_\alpha$. As $\alpha/2 < 2 + \alpha$, 
we have thus established \eqref{moment*++}.

Turning to \eqref{moment*+}, we similarly have from \eqref{bd:F1} of Lemma \ref{moment-}
that for some $C<\infty$ and $c>0$, uniformly over $m \in [n^\epsilon,n]$, 
$i \in [n_\alpha,n-n_\alpha]$ and $\ell \in [m_\alpha,m-m_\alpha]$, 
\begin{align}
E[g_{n,\alpha}(i) \tilde{g}_{m,\alpha}(\ell) G(S_i,\tilde{S}_\ell) ] &\le 
[1+ C (\log n)^{2-\alpha}] \, \ovg_{n,\alpha} \ovg_{m,\alpha} E[G(S_i,\tilde{S}_\ell)] +
2 C e^{-c (\log m)^2} \nonumber \\
& \le C (\log n)^{-2} i^{-1/2} \ell^{-1/2} + 2 C e^{-c (\log m)^2}
 \label{eq:bd1}
\end{align}
(using in the latter inequality also \eqref{ee1}, \eqref{gre1} and \eqref{mult} at $p=1$). As
$\log m \ge \epsilon \log n$, the sum of 
the \abbr{rhs} of \eqref{eq:bd1} over $i \le n$ and $\ell \le m$ is at most as specified
(ie. $O(\sqrt{n m} (\log n)^{-2}$). Further, even when $i < n_\alpha$ or
$\ell < m_\alpha$ or $i > n - n_\alpha$ or $\ell > m - m_\alpha$,
we still get the bound  $ C i^{-1/2} \ell^{-1/2}$ on the \abbr{lhs} of \eqref{eq:bd1}.
The sum of $i^{-1/2} \ell^{-1/2}$ subject to any one of the latter four restrictions is 
at most $O(\sqrt{m_\alpha n}) = O(\sqrt{n m} (\log m)^{-\alpha/2})$, which is as 
required (for $\alpha>4$).

Finally, recall that for some $C,C_3$ finite and all $m,n \in \N$,
\[
E Y_{m,n} \le \sum_{i=1}^n \sum_{\ell=1}^{m} E [ G(S_i,\tilde{S}_\ell) ] \le C_3 
\sum_{i=1}^{n} \sum_{\ell=1}^{m} i^{-1/2} \ell^{-1/2}
\le C \sqrt{n m} \,,
\]
as claimed.
\end{proof}

%%%%%%%%%%%%%%%%%%%%%%%%%%%%%%%%%%%%%%%%%%%%%%%%%%%%%%%%%%%%%%%%%%%%

\subsection{The upper bound in the limsup-\abbr{lil}} As in case of the capacity 
limsup-\abbr{lil} lower bound, 
we adapt here the relevant element from the proof of the limsup-\abbr{LIL} of $|\R_n|$ and 
\abbr{srw}  $\Z^2$, namely \cite[Prop. 4.1]{BK02}. To this end, we first establish a key
approximate additivity for $\varphi_n := E \crn$. 
\begin{lem}\label{error} 
There exists $c'$ finite, such that for any $a,b \in \N$, 
\begin{align}\label{eq:bd-lem46}
0 \le \varphi_a + \varphi_b - \varphi_{a+b} \le c' \frac{\alpha^{1/4} \, (a+b)}{(\log (a+b))^2} 
\,, 
\end{align}
where $\varphi_n := E \crn$ and $\alpha:=\min(a,b)/(a+b)$.
\end{lem}
\begin{proof} Starting at the expected value of \eqref{def:Rab}, we get by the same 
reasoning we have used in deriving \eqref{eq:Wjbd}, that 
\begin{align*}
0 \le E [V_{0,a,a+b}] = \varphi_a + \varphi_b - \varphi_{a+b} & \le 
2E \sum_{i=1}^a  \sum_{\ell=1}^b g_{a,\alpha}(i) G(S_i,\tilde{S}_{\ell})\tilde{g}_{b,\alpha}(\ell) = 2 E Y_{a,b} \,.
% \le C a^{1/2}(\log a)^{-1}b^{1/2}(\log b)^{-1} \,.
\end{align*}
Assuming \abbr{wlog} that $a \le b$, it thus suffices to verify that $E Y_{a,b} \le C a^{1/4} b^{3/4}/(\log b)^2$
(yielding \eqref{eq:bd-lem46} for some $c'(C)<\infty$).
Indeed, for $a \ge \sqrt{b}$ this follows from \eqref{moment*+}, whereas if $a < \sqrt{b}$ then
even the bound $E Y_{a,b} \le C \sqrt{a b}$ which we have from Lemma \ref{moment*}, suffices.
\end{proof}

\blue{Recall \eqref{def:Rab} that $R_{a+b} - R_a - R_{b} \circ \theta_a = -V_{0,a,a+b} \le 0$
for any $a,b \ge 0$. This implies a \emph{non-random bound} on the difference of such centered variables, 
yielding in terms of the non-random $c'$ of Lemma \ref{error} the upper bound}
\begin{align}\label{eq:45BK}
\overline{R}_{a+b} - \overline{R}_a - \overline{R}_{b} \circ \theta_a &\le \blue{c'} \Big(\frac{\min(a,b)}{a+b}  \Big)^{1/4} \frac{(a+b)}{(\log(a+b))^2}\,.
\end{align} 
Utilizing \blue{\eqref{eq:45BK}, we next} establish sharp tail estimates for $\max_{j \le n} \{\overline{R}_j\}$
(in particular, improving upon \cite[Lemma 2.5]{HS20}).
\begin{lem}\label{bkk}
%Let $G_j=G_j^{(m)}:=\frac{(\log m)^2}{m} \overline{R}_j$ for $j \le m$. 
%(a) There exists $C<\infty$ such that
%\begin{align*}
%P(\max_{1\le j\le m} \{G_j\} >C)<1/2.
%\end{align*}
%(b)
For some $c>0$, $C<\infty$ and all $n$,
\begin{align}\label{eq:mgf-bd}
E[e^{c D^{(n)}}] \le C \,, \qquad \qquad D^{(n)}:=\frac{(\log n)^2}{n} \max_{0 \le j\le n} \{\overline{R}_j\} \,.
\end{align}
\end{lem}
\begin{proof} \blue{Note that \eqref{eq:mgf-bd} matches the statement of \cite[(4.4)]{BK02}}
for $G_j:=G_j^{n}:=\frac{(\log n)^2}{n} \overline{R}_j$,  $j \le n$. \blue{It is easy to check that
the proof of \cite[(4.3) \& (4.4)]{BK02} applies verbatim for any variables $\{G_j\}$ that satisfy
\cite[(4.5) \& (4.6)]{BK02}, and furthermore, that their} 
argument applies even if the power $\alpha^{1/2}$ on the right-most term in 
\cite[(4.5)]{BK02} is replaced by $\alpha^{1/4}$. \blue{Indeed, this is what we have here, with
\eqref{eq:45BK} yielding that for some non-random $c_1<\infty$ 
% x/(log x)^2 increasing for x \ge e^2.
and all $a \le j \le n$,
\[
G^n_{j} - G^n_a  \le  G^n_{j-a} \circ \theta_a + c_1 \Big(\frac{a}{j} \wedge \frac{j-a}{j} \Big)^{1/4} \,.
\]
To finish the proof note that} for some $c_2 < \infty$ and all $a,b \ge 0$,
\blue{we get from} \cite[Cor. 1.5]{As5} that 
\begin{align}
E [ (\overline{R}_b \circ \theta_a)^2 ] &\le \frac{c_2 b^2}{(\log b)^4}\,, \qquad \text{hence} 
\qquad E [ (G^n_j \circ \theta_a)^2 ] \le c_2 \frac{j^2 (\log n)^4}{n^2 (\log j)^4}\,,
\label{eq:46BK}
\end{align} 
\blue{which is precisely \cite[(4.6)]{BK02}.}
% We essentially repeat the proof of \cite[Lemma 4.3]{BK02}.
% Note that \cite[Lemma 4.3]{BK02} holds even if they replace $\alpha^{1/2}$ in \cite[Lemma 4.2]{BK02} into $\alpha^{1/4}$. 
% If we set $F_j=R_j-\pi^2/8 \log j$,we can not repeat  \cite[(4.6)]{BK02} since we do not have the estimation of second order of $\varphi_n$ (see \cite[Corollary 1.4]{As3}) while they obtain the second order of $E|\R_n|$ (see \cite[(2.6)]{BK02}). 
%Hence, we set $F_j=R_j-\varphi_j$ and repeat the proof of \cite[Prop. 4.1]{BK02} by using  
% instead of \cite[Lemma 4.2 and (2.7)]{BK02}. 
\end{proof}

For the limsup-\abbr{lil} upper bound, by Borel-Cantelli it suffices to show that for any $q>1$, 
$\gamma>0$ and $\varepsilon>0$, 
\begin{equation}\label{eq:summable}
\sum_i P\bigg( \frac{(\log m)^2}{m} \max_{r_{i-1} < \ell \le r_i} \{ \overline{R}_\ell \} \ge 
\big(\frac{\pi^2}{8}+ 2 \varepsilon\big) k \log k
\bigg) < \infty \,,
\end{equation}
where $r_i = q^i$, $k=2^p$ for $p=[(\log \gamma + \log_3 r_i)/\log 2]$ and $m=\lceil r_i/k \rceil$. Now,
considering \eqref{eq:decomp4} for $n_j=j m$, $j<k'$ and $n_{k'}=\ell$ it follows from 
\eqref{eq:bd-lem46} that 
\begin{equation}\label{eq:dec-ub}
\frac{(\log m)^2}{m} \max_{(k'-1)m < \ell \le k' m} \{ \overline{R}_{\ell} \}
% \le \sum_{j=1}^k D_j^{(m)} + k' \varphi_m - \varphi_{k' m}
% + (\varphi_{\ell-(k-1) m} + \varphi_{km-\ell} - \varphi_m)
% + (\varphi_{km} - \varphi_\ell - \varphi_{km-\ell}) 
\le \sum_{j=1}^{k'} D_j^{(m)} + \frac{(\log m)^2}{m} \Big( k' \varphi_m - \varphi_{k' m} \Big)
+ c'  \,,
\end{equation}
where $D^{(m)}_j$ are i.i.d. copies of $D^{(m)}$ of Lemma \ref{bkk}. With 
$k' \mapsto (k' \varphi_m - \varphi_{k' m})$ non-decreasing and $D^{(m)} \ge 0$, the
maximum over $k' \le k$ of the \abbr{rhs} of \eqref{eq:dec-ub}, is attained at $k'=k$.
Further, en route to \eqref{log3}, we showed that as $p = o(\log m) \to \infty$,
\begin{equation}\label{log3*}
\frac{1}{k \log k} \frac{(\log m)^2}{m} \Big( k \varphi_m - \varphi_{k m} \Big) \to \frac{\pi^2}{8} \,.
\end{equation}
Thus, noting that \eqref{eq:mgf-bd} results with 
 \[
P\Big(\sum_{j=1}^k D_j^{(m)} \ge \varepsilon k \log k \Big) \le C^k e^{-\varepsilon c k (\log k)} \,,
 \]
which is summable over $i$ for our choice of $k=\gamma \log i$, we have established \eqref{eq:summable}
and thereby completed the proof of the limsup-\abbr{lil}.

% \end{proof}

%%%%%%%%%%%%%%%%%%%%%%%%%%%%%%%%%%%%%%%%%%%%%%%%%%%%%%%%%%%%%%%%%%%%%%%%%%%

\subsection{Non-random and positive liminf-\abbr{LIL}} Setting hereafter
$\blue{\barh}_4(n) := n (\log_2 n)/(\log n)^2$, we first show that the 
$[-\infty,\infty]$-valued, 
\[
c_\star := - \liminf_{r \to \infty} \bigg\{ \frac{\overline{R}_r}{\barh_4(r)} \bigg\} \,,
\]
is non-random. Indeed, recall \eqref{for1} that $\ca(\R_r) - \ca (\R[k,r]) \in [0,k]$. Thus,  
for any $k$ finite, changing $\R_k$ without altering $S_k$ yields at most a difference
of $k$ in  the value of $\ca(\R_r)$, implying by the Hewitt-Savage zero-one law
that $c_\star$ is non-random. 

\blue{Turning next to show that $c_\star>0$,} it suffices to establish this 
for the sub-sequence $r_{j+1}=r_j + 2 n_j$, where $n_j:=2^{j^2}$, $r_0:=0$
%\[
%n_j := \sum_{i=1}^{j-1} 2^{i^2+1} = 2 \sum_{i=1}^{j-1} r_i \,.
%\]
\blue{and we proceed to show that infinitely many  
$-\overline{R}_{r_{j+1}}$ are at least of $O(\barh_4(n_j))$, due to heavy tails
of the non-negative variables $V_{0,n,2n}/\barh_4(n)$. Specifically,}
setting $Q_n :=|\R(0,n] \cap \R (n,2n]|$ we have from \cite[Prop. 1.6]{As5} that 
\begin{equation}\label{bd:Vn}
V_{0,n,2n} \ge 2 \inf_{j,\ell \in [1,2n]} \{ G(S_j,S_\ell) \} \, R_n \, R_{n,2n} - Q_n \,.
\end{equation}
Recall from \cite[Section 3.4]{LA91} that $E Q_n \le C_0 \log n$ for some $C_0 <\infty$ and all $n$. 
Hence, by Markov's inequality
\begin{align}\label{bd:Dn}
P(\hat{A}_n^c) := P( Q_n \ge (\log n)^3 )
% \le  E Q_n  (\log n)^{-3}
\le  C_0 (\log n)^{-2}. 
\end{align} 
Further, recall \cite[(1.4) and Cor. 1.5]{As5}, that $E R_n \ge n/(\log n)$ 
and Var$(R_n) \le C_1 n^2/(\log n)^4$ for some $C_1 <\infty$ and all $n$ large, 
in which case by Markov's inequality
\begin{align*}
P\Big(R_n \le \frac{n}{2 \log n}\Big) \le \Big(\frac{2 \log n}{n}\Big)^2 {\rm Var}(R_n) \le \frac{4 C_1}{(\log n)^2} \,.
\end{align*}
Consequently, by the union bound,
\begin{align}\label{bd:An}
P(A_n^c) := P\Big( \min( R_{n}, R_{n,2n}) \le \frac{n}{2 \log n} \Big) \le 8 C_1 (\log n)^{-2} .
\end{align}
Next, from \eqref{gre1} we have that 
\[
F_{k,m} := \big\{ \max_{j \le 2km} |S_j| \le \sqrt{m} \big\} \quad \Longrightarrow \quad 
\inf_{j,\ell \in [1,2km]} \{ G(S_j,S_\ell) \} \ge (4 C m)^{-1} \,.
\] 
Setting $c=1/(10 C)$ it thus follows from \eqref{bd:Vn} that for all 
$n=k m \ge n'(C)$, on the event $G_n := F_{k,m} \cap \hat{A}_n \cap A_n$, 
\[
V_{0,n,2n} \ge \frac{2}{4 C m} \Big(\frac{n}{2 \log n} \Big)^2  - (\log n)^3 \ge 
\frac{c n k}{(\log n)^2} \,.
\]
Similarly to our derivation of the \abbr{lhs} of \eqref{eq:Fk}, it follows from the invariance principle
that $P(F_{k,m}) \ge c_2^k$ for some $c_2>0$ and any $k,m \ge 1$. By \eqref{bd:Dn}-\eqref{bd:An}
this implies in turn that $P(G_n) \ge \frac{1}{3} c_2^k$ for $k=2^p=[\gamma \log_2 n]$, provided
$\gamma' := \gamma \log (1/c_2) < 2$ and $n \ge n'$. To summarize, we have that for 
$c'=c \gamma>0$ and all $n \ge n'$,  
\[
P\Big( V_{0,n,2n} \ge c' \barh_4(n) \Big) \ge (\log n)^{-\gamma'} \,.
\]
The same applies for the mutually independent $\{V_{r_j,r_j+n_j,r_{j+1}} \}$, hence upon
fixing $\gamma'<1/2$ we get by the second Borel-Cantelli lemma, that a.s.
\begin{equation}\label{eq:bc2}
\limsup_{j \to \infty} \Big\{ \barh_4(n_j)^{-1} V_{r_j,r_j+n_j, r_{j+1}} \Big\} \ge c' > 0\,.
\end{equation}
Now, as in \eqref{eq:decom}, for any $r \ge 0$, $n \ge 1$,
\begin{align}\label{eq:dcm-liminf}
\overline{R}_{r,r+2n} &  
=2 \varphi_{n} - \varphi_{2n} + \overline{R}_{r,r+n}
  + \overline{R}_{r+n,r+2n} - V_{r,r+n,r+2n}  \,.
\end{align}
Considering \eqref{eq:23As5} for $p=1$ (that is, $k=2$), 
%(at $2n$), 
we see that as $n \to \infty$,  
\begin{align}\label{eq:exp-lim}
 \frac{(\log n)^2}{n} \big[ 2 \varphi_{n}- \varphi_{2n} \big] \le 
 \frac{(\log n)^2}{n} E[\chi_{2n}(1,1)] \to \frac{\pi^2 \log 2}{4}  \,.
\end{align} 
Further, recall \eqref{eq:summable} that for any $\delta>0$, 
\begin{align}\label{eq:bc1}
\sum_j P\big( \overline{R}_{n_j}  \ge (1+\delta) h_4(n_j) \big) < \infty \,.
\end{align}
The same applies of course also for $\overline{R}_{n_j} \circ \theta^{r_j}$ and 
$\overline{R}_{n_j} \circ \theta^{r_j+n_j}$, so with $h_4(n)/\barh_4(n) \to 0$,
we deduce from \eqref{eq:bc2}-\eqref{eq:bc1}
(at $n=n_j$ and $r=r_j$), that a.s.
\[
\liminf_{j \to \infty}  \Big\{ \frac{\overline{R}_{r_j,r_{j+1}}}{\barh_4(n_j)} \Big\}
\le - \limsup_{j \to \infty} \Big\{ \barh_4(n_j)^{-1} \, V_{r_j,r_j+n_j, r_{j+1}} \Big\} 
\le - c' \,.
\]
Now from \eqref{eq:45BK} (at $a=r_j$, $b=2n_j$), 
\[
\overline{R}_{r_{j+1}} \le \overline{R}_{r_j} + \overline{R}_{r_j,r_{j+1}}  
+ 
%\Big(\frac{r_j}{r_{j+1}}\Big)^{1/4}  
\frac{c_1 r_{j+1}}{(\log r_{j+1})^2} \,,
\]
and since $r_{j+1} \le 3 n_j$, dividing by $\barh_4(n_j)$ and taking limits, yields that
\[
- 3 c_\star \le \liminf_{j \to \infty}  \Big\{ \barh_4(n_j)^{-1} \, \overline{R}_{r_{j+1}} \Big\}
\le -c' + \limsup_{j \to \infty}  \Big\{  \barh_4(n_j)^{-1} \, \overline{R}_{r_j} \Big\} \,.
\] 
The last term is a.s. zero (as \eqref{eq:bc1} applies also for $\{r_j\}$ instead of $\{n_j\}$
and $h_4(n)/\barh_4(n) \to 0$), so we conclude that $c_\star \ge c'/3>0$.

%%%%%%%%%%%%%%%%%%%%%%%%%%%%%%%%%%%%%%%%%%%%5
\subsection{Finiteness of the liminf-\abbr{LIL}}
We show that $c_\star \le c_o < \infty$ by following the 
proof of the upper bound of \cite[Thm. 1.7]{BCR} (on the liminf \abbr{LIL} of $|\R_n|$ in $\Z^2$), 
while replacing \cite[Thm. 1.5]{BCR} and \cite[Lemma 10.3]{BCR}, respectively, by 
\begin{align}\label{eq:sd2}
\sup_n \Big\{ (\log n) (\log_2 n)^2 P\Big( - \overline{R}_n > c_o \, \barh_4(n) \, \Big) \Big\} 
& < \infty \,, \\ 
\label{eq:sd3}
\sup_n  \Big\{ (\log n)^2 P\Big( \max_{n/q_0 \le k \le n}(\overline{R}_n-\overline{R}_k )
>\epsilon \, \barh_4(n) \Big) \Big\} & < \infty \,,
\end{align}
holding for some $c_o < \infty$, any $\epsilon>0$ and some $q_0(\epsilon)>1$. 

Similarly to $|\R_n|$, the capacity is sub-additive (see \eqref{def:Rab}),
and upon centering satisfies \eqref{eq:45BK}, which is the analog of \cite[(10.2)]{BCR}. 
Thus, the bound \eqref{eq:sd3} follows as in \cite[Proof of Lemma 10.3]{BCR},
now using \eqref{eq:mgf-bd} to arrive at \cite[(10.14)]{BCR} and to bound the 
\abbr{rhs} of \cite[(10.15)]{BCR}). 

Since $|\overline{R}_{n} - \overline{R}_{n'}| \le |n-n'|$, it suffices to prove 
\eqref{eq:sd2} \blue{only for some $\{n_i\}$ such that $n_{i+1}-n_i = o(\barh_4(n_i))$.
We take here all integers of the form} $n=m k$, $k=2^p$, $p=[(\log_2 n)/\log 2]$
\blue{(thus with gaps of size $k=O(\log n)=o(\barh_4(n))$).} Setting such values and
$n_u':=2^{-u} n$ for $1 \le u \le p$, we have 
as in \eqref{eq:decom}, now using an alternative expression for $\Delta_{n,k}$ of \eqref{eq:decomp4}, that  
\begin{align}
-\overline{R}_n &=  \overline{\Delta}_{n,k} - \sum_{j=1}^k \overline{U}_j \,, \qquad 
\Delta_{n,k} 
% :=\sum_{j=1}^{k-1} V_{0,jm,(j+1)m} 
=\sum_{u=1}^p \sum_{j=1}^{2^{u-1}} V_{(2j-2) n_u', (2j-1) n_u' ,2j n_u'} \,,
\label{aw1}
\end{align}
with $k$ i.i.d. copies $\{\overline{U}_j\}$ of $\overline{R}_m$ and the i.i.d. variables
$\{ V_{(2j-2) n_u', (2j-1) n_u' ,2j n_u'}\}$ per fixed $u \ge 1$.
Since Var($\overline{R}_m) \le C_1  m^2/(\log m)^4$ it follows by Markov's inequality, that 
\begin{align}\label{aw2}
P\Big(- \sum_{j=1}^k \overline{U}_j \ge  \epsilon c_o \barh_4(n) \Big)
\le  \big(\epsilon c_o \barh_4(n) \big)^{-2} \frac{C_1 k m^2}{(\log m)^4} 
\le \frac{2 C_1}{(\epsilon c_o)^2 k (\log_2 n)^2}  \,. 
\end{align}
Setting $c_o>(1-\epsilon)^{-1} c_1^{-1}$, we arrive at \eqref{eq:sd2} out of (\ref{aw1}), (\ref{aw2}) and 
the following lemma. 
\begin{lem}\label{sd1}
For some $c_1>0$ and any $\lambda>0$, 
\begin{align}\label{md-Delta}
\limsup_{n\to \infty}
\frac{1}{\log_2 n}
\log P\big( \overline{\Delta}_{n,k} \ge \lambda \barh_4(n) \big)
\le -c_1 \lambda \,,
\end{align}
where $\Delta_{n,k}$ are as in \eqref{def:Delta} for $k=2^p$ and $p=[(\log_2 n)/\log 2]$. 
\end{lem}

We note in passing \cite[Lemma 2.6]{HS20}, which is somewhat related to Lemma \ref{sd1}.
The proof of Lemma \ref{sd1} relies in turn on our next result. 
 \begin{lem}\label{sd2}
Set $p=[(\log_2 n) \blue{/ \log 2}]$ and for any $r \in \N$, the partition $I_i^{(r)} := ((i-1)r,i r]$ of $\N$. Consider for 
$n'_u:=2^{-u} n$ and each $0 \le u \le p$, the i.i.d. variables 
\begin{align*}
\alpha^{(n'_u)}_{j} := \frac{1}{n} \sum_{i \in I^{(n'_u)}_{2j-1}} \sum_{\ell \in I^{(n'_u)}_{2 j}}
G(S_i , S_\ell ),    \qquad  1 \le j \le 2^{u-1}\,.  \qquad  
\end{align*}
Then, for some $c_2>0$,
\begin{align}\label{mgf-Theta}
\sup_{n \in \N} E[e^{c_2 \overline{\Theta}_n}] < \infty, \qquad \qquad 
\Theta_n :=\sum_{u=1}^p \sum_{j=1}^{2^{u-1}} \alpha^{(n'_u)}_{j}\,.
\end{align}
\end{lem}
\begin{proof} With the \abbr{srw} having independent and symmetric increments, one easily verifies that 
$\alpha^{(n'_u)}_{1} \stackrel{(d)}{=} 2^{-u} X_{n'_u}$, for $X_n$ of Lemma \ref{extail:lem}. Consequently, from \eqref{exptail*} and \eqref{exptail} we know that 
\begin{align}\label{bd:varphi}
\varphi(\lambda)
:= \sup_{u , n \in \N} E[ \exp(\lambda \, 2^u \, \overline{\alpha}^{(n'_u)}_{1}) ] \le 1 + c \lambda^2 < \infty \,, 
\end{align}
for some $c<\infty$ and all $\lambda>0$ small enough. 
%By a Tayor expansion to second order of the \abbr{mgf} of 
% $2^u \overline{\alpha}^{(n)}_{u,1}$, using the uniform moment estimates of \eqref{exptail*}
The uniform \abbr{mgf} bound of \eqref{mgf-Theta} then follows as in 
\cite[Page 177, Proof of Thm. 1]{LG94}, upon setting $\alpha_0=\alpha_{1}^{(n'_0)}$, $c_2=b_\infty>0$ of \cite{LG94}, 
and noting that \eqref{bd:varphi} suffices in lieu of both \cite[Lemma 2]{LG94} and 
the scale invariance of \cite[property (ii)]{LG94}.
\end{proof}

\begin{proof}[Proof of Lemma \ref{sd1}] For $u \le p$, 
consider the i.i.d. variables $W_{u,j} := W^{(n_u')} \circ \theta_{(2j-2) n'_u}$, with 
\begin{align*}
W^{(m)} &:= \sum_{i,\ell \in [1,m]} 
g_{m,\alpha}(i) G(S_i,S_{\ell+m}) g_{m,\alpha}(\ell) \circ \theta_{m} 
\end{align*}
having the law of $W_1$ of \eqref{eq:Wjbd} and their (i.i.d.) approximations
$\underline{W}_{u,j} := n \ovg_{n'_u,\alpha}^2 \alpha^{(n'_u)}_j$
(for $g_{m,\alpha}(\cdot)$ and $\ovg_{m,\alpha}$ of Lemma \ref{moment*}).
Setting  
\[
Z_u := \sum_{j=1}^{2^{u-1}} W_{u,j} \,, \qquad \underline{Z}_u := \sum_{j=1}^{2^{u-1}} \underline{W}_{u,j} \,,
\]
it follows by Cauchy-Schwarz and \eqref{moment*++} that for some $C<\infty$, any $u \le p$ and all $n$,
\begin{equation}\label{bd:Zu-uZu}
E \Big[ (Z_u - \underline{Z}_u)^2 \Big] \le 2^{2(u-1)} E [ (W_{u,1} - \underline{W}_{u,1})^2] \le 
C n^2 (\log n)^{-\alpha/2} \,.
\end{equation}
In particular, taking $\alpha > 8 + 2 c_1 \lambda$ yields, by the union bound and Markov's inequality that 
for any $\epsilon>0$,
\begin{align}
P\Big( \sum_{u=1}^p (Z_u-\underline{Z}_u) \ge \epsilon \barh_4(n) \Big) &\le 
\sum_{u=1}^p P(Z_u - \underline{Z}_u \ge \epsilon p^{-1} \barh_4(n) ) \nonumber \\ 
& \le C p^3 ( \epsilon \barh_4(n) )^{-2} n^2 (\log n)^{-\alpha/2} 
% \le \frac{C}{\epsilon^{2} (\log 2)^3} (\log_2 n) (\log n)^{4-\alpha/2}  
\le C' \epsilon^{-2} (\log n)^{-c_1 \lambda}.
\label{eq:Zu-uZu}
\end{align}
We also find for our choice of $k$, see \eqref{log3*}, that
\[
E \Delta_{n,k} =  \frac{\pi^2}{8} \barh_4(n) (1+o(1)) \,.
\]
Similarly to \eqref{eq:Wjbd}, we have that $V_{(2j-2) n_u', (2j-1) n_u' ,2j n_u'} \le 2 W_{u,j}$ 
for any $u,j$. In view of \eqref{aw1} and \eqref{eq:Zu-uZu} it thus
suffices to establish \eqref{md-Delta} with $\overline{\Delta}_{n,k}$
replaced by 
\begin{align*}
2 \sum_{u=1}^p \underline{Z}_u - \frac{\pi^2}{8} \barh_4(n) \,,
\end{align*}
which in view of \eqref{ee1} and the definition of $\Theta_n$, can be further replaced by 
\begin{align*}
\frac{\pi^2}{8} \frac{n}{(\log n)^2} \Big( \Big( \frac{\pi^2}{4}+o(1) \Big) \Theta_n - \log_2 n \Big) \,.
\end{align*}
Moreover, by \eqref{bd:Zu-uZu}, 
\[
\Big| E \sum_{u=1}^p Z_u - E \sum_{u=1}^p \underline{Z}_u \Big| 
 \le O\Big(\frac{p n}{(\log n)^{\alpha/4}}\Big) = o(\barh_4(n)) \,,
\]
and combining \cite[Lemma 6.5]{As5} with the considerations 
as in \eqref{eq:rough-tail} and after \eqref{eq:bd1}, we deduce that for large $m$,
\[
2 E [W^{(m)}] = (1+o(1)) E [V_{0,m,2m}] \,.
\] 
It then follows that 
\[
2 \sum_{u=1}^p E \underline{Z}_u = (1+o(1)) \frac{\pi^2}{8} \barh_4(n) 
\]
and consequently, 
\[
\frac{\pi^2}{4} E \Theta_n = (1+o(1)) \log_2 n \,.
\]
Thus it suffices to establish \eqref{md-Delta} with $\overline{\Delta}_{n,k}$ replaced by 
\begin{align*}
\frac{\pi^4}{32} \frac{n}{(\log n)^2} \overline{\Theta}_n  \,,
\end{align*}
which in turns follows from \eqref{mgf-Theta}, upon setting $c_1=32 \pi^{-4} c_2>0$.
\end{proof}

%%%%%%%%%%%%%%%%%%%%%%%%%%%%%%%%%%%%%%%%%%5

\section{\abbr{LIL} for \abbr{srw} on $\Z^d$, $d \ge 5$: Proof of Theorem \ref{m3}}\label{sec:m2-dge5}
Since $\crn$ for the \abbr{srw} on $\Z^{d}$, $d \ge 5$, has similar
structural properties to the size of the range of the \abbr{srw} on $\Z^{d-2}$,
we establish Theorem \ref{m3} by adapting the proof in \cite[\blue{Section 3}]{BK02} 
for the \abbr{LIL} of the latter sequence. Specifically, setting 
$\rho_n := \sqrt{n \log n}$ when $d=5$ and otherwise $\rho_n := \sqrt{n}$, 
we have from \cite[Thm. A]{S20} in case $d=5$, and from \cite[Lemma 3.3]{As3} when
$d \ge 6$, that
\begin{equation}\label{hjm6}
\lim_{n \to \infty} \frac{1}{\rho_n} \| \ocrn \|_2 = \sigma_d \,.
\end{equation}
\blue{Next, recalling for integers $0 \le a \le b \le c$, the notations of \eqref{def:Rab},
\begin{equation}\label{def:Rab2}
R_{a,b} := \ca(\R(a,b]) \,, \qquad V_{a,b,c} := R_{a,b} + R_{b,c} - R_{a,c} \ge 0 \,,
\end{equation}
we} proceed with the following variant of \cite[Lemma 3.2]{As3}.
\begin{lem}\label{vv}
% Recall the notations of \eqref{def:Rab} and 
For any $0 \le a < b$, set
\begin{align*}
\widetilde{V}_{a,b}  := \sup_{t \ge b} \{ V_{a,b,t} \} \,, \qquad 
\hat{V}_{a,b}  := \sup_{s \le a} \,\{ V_{s,a,b} \} \,.
\end{align*}
Then, for some $C_{d,\ell}$ finite, any $\ell \ge 1$ and all $a<b$,
\begin{align}\label{eq:Vab-bd}
E[\widetilde{V}_{a,b}^\ell]\le C_{d,\ell} \, f_d(b-a)^{\ell}\,,  \qquad
E[\hat{V}_{a,b}^\ell]\le C_{d,\ell} \, f_d(b-a)^{\ell}\,,
\end{align}
where 
\begin{align*}
f_5(n)=\sqrt{n}, \quad f_6(n)=\log n, \quad 
%\text{ and } \quad 
f_d(n)=1 \quad \forall d\ge7.  
\end{align*}
\end{lem}
\begin{proof} By the shift invariance of the \abbr{srw}, we may \abbr{wlog} set $a=0$. 
Further, in view of \cite[(2.9) \& (2.11)]{As5},  for a fixed set $A$, the function 
\[ 
B \mapsto \ca(A)+\ca(B)-\ca(A \cup B)
\]
is non-decreasing (and bounded above by $\ca(A)$). In particular, 
the value of $\widetilde{V}_{0,n}$ is attained for $t \to \infty$. 
Thus, from \cite[Prop. 1.2]{As3} 
we arrive at 
\begin{equation}\label{eq:Vbd} 
\widetilde{V}_{0,n} \le 2 \sum_{x \in \R_n}\sum_{y \in \R(n,\infty)} G(x,y) \stackrel{d}{=} 
2 \sum_{x \in \R_n}\sum_{y \in \hat{\R}_\infty} G(x,y), \,
\end{equation}
where $\hat{\R}_\infty$, denotes the range of an independent \abbr{srw}.  
Similarly, the value of $\hat{V}_{0,n}$ is attained at $s \to -\infty$, 
with the right-side of \eqref{eq:Vbd} also bounding $\hat{V}_{0,n}$ 
(we then have $\R(-\infty,0]$ instead of $\R(n,\infty)$ in \eqref{eq:Vbd}).
Thereafter, adapting \blue{\cite[Section 3.1]{As3}} 
%the proofs of \cite[Lemma 3.1 \& 3.2]{As3} 
yields \eqref{eq:Vab-bd}. Indeed, 
with $p_{2k}(x,y):=P^x (S_{2k}=y)$ the 
square of a transition probability, we have as in the proof of \cite[Lemma 3.1]{As3}, 
that for \emph{even} $k \ge 0$ and any $a \in \Z^d$,
\begin{equation*}
s_a := \sum_{x,y \in \Z^d} G(0,x) G(0,y) p_k (x,y+a)  \le s_0 \,.
\end{equation*}
In case of a lazy \abbr{srw}, this applies for any $k \ge 0$, so  
summing over $k \le n$ yields that 
\begin{align}\label{Lem31-As3}
\max_{a \in \Z^d} \Big\{
\sum_{x,y \in \Z^d} G(0,x) G(0,y) G_n(x,y+a) \Big\} &= \sum_{x,y \in \Z^d} G(0,x) G(0,y) G_n(x,y) 
\\
&= \sum_{x,y \in \Z^d} G_n(0,x) G(0,y) G (x,y) \le C_{d} \, f_d(n) \,,
\nonumber 
\end{align}
where we have utilized \cite[(3.4)]{As3} for the latter inequality. Further, as in \cite{As3}, 
up to an increase of $C_d$ value, \eqref{Lem31-As3} extends to the original \abbr{srw}. Now, 
similarly to \cite[(3.5)]{As3}, it follows from \eqref{eq:Vbd} that 
\[
E[\widetilde{V}_{0,n}^\ell] \le 2^\ell \sum_{\underline{x},\underline{y}} 
E\Big[ \prod_{i=1}^\ell L_n(x_i) \Big] E \Big[\prod_{i=1}^\ell L_\infty(y_i)\Big]
\prod_{i=1}^\ell G(x_i,y_i) \,,
\]
where $L_n(x)$ denotes the total \abbr{srw} local time at $x \in \Z^d$, during time interval $[1,n]$
(and the same bound applies for $E[\hat{V}_{0,n}^\ell]$).
For $\ell=1$ we thus get \eqref{eq:Vab-bd} out of \eqref{Lem31-As3} 
(as $E [L_n(x)] = G_n(0,x)$ and $E [L_\infty(y)] = G(0,y)$). The general case then follows by an
inductive argument, as in the proof of \cite[Lemma 3.2]{As3}, utilizing also that $a=0$
is optimal in \eqref{Lem31-As3}.
\end{proof}

Utilizing Lemma \ref{vv}, we next establish 
the analog of \cite[Lemma 3.3]{BK02} for $\crn$.
\begin{lem}\label{l-moment}
For any $d \ge 5$, $m \ge 3$, there exists $c_m$ finite, such that for all $b > a \ge 0$
\begin{equation}\label{eq:BK-33a}
\| \, \overline{R}_b -\overline{R}_a  \, \|_m  \le c_m \, \rho_{b-a} \,.
\end{equation}
Further, for some ${\bar c}_m$ finite and any $\lambda>0$, $b > a \ge 0$,
\begin{equation}\label{eq:BK-33b}
P (\max_{n \in [a,b]} \{|\ocrn -\overline{R}_a| \} > \lambda \rho_{b-a}) \le {\bar c}_m \lambda^{-m} \,.
\end{equation}
\end{lem}
\begin{proof} \blue{From \eqref{def:Rab2} we see that}
\[
V_{0,a,b} = R_a + R_{a,b} - R_{b} \in [0,\hat{V}_{a,b}] \,,
\]
for $\hat{V}_{a,b}$ of Lemma \ref{vv}. In particular, for any $m \ge 3$, $b>a$,
\[
0 \le E[V_{0,a,b}]^m \le E[V_{0,a,b}^m] \le E[\hat{V}_{a,b}^m] \le C_{d,m} f_d(b-a)^{m} 
= o (\rho_{b-a}^m)\,.
\]
We can thus replace $R_b-R_a$ in 
%the bound 
\eqref{eq:BK-33a} by $R_{a,b}$ and thereby 
due to the shift invariance of the increments, set \abbr{wlog} $a=0$ (whereupon $R_a=0$). 
Hence, analogously to \cite[(3.34)]{BK02}, it suffices for \eqref{eq:BK-33a} 
to show inductively over $\ell \ge 1$,
that $\sup_n \{L_{n,2\ell}\}$ is finite for 
\[
L_{n,\ell} := \frac{1}{\rho_n} \| \ocrn \|_{\ell} \,.
\]
The induction basis $\ell=1$ is merely \eqref{hjm6}. Further, with 
\[
\lim_{n \to \infty} 
\rho_{2n}^{-1} \sup_{a < 2n} \| \overline{V}_{0,a,2n} \|_{2\ell}  = 0 \,,
\]
by the preceding decomposition, we can and shall replace $R_{2n}$ in the induction step, by 
\[
\crn + R_{n,2n} \stackrel{d}{=} \crn + \hat{R}_n \,,
\]
where $\hat{R}_n$ denotes the capacity of the range of an independent second \abbr{srw}. 
For any $\ell \ge 2$, by the induction hypothesis $\sup_n \{L_{n,k}\}$ are finite for all $k \le 2 (\ell-1)$, 
hence
\[
\sup_{n,n'} \sum_{k=2}^{2\ell-2} {2\ell \choose k} \,
L_{n,k}^k \, L_{n',2\ell-k}^{2\ell-k} := c_\ell < \infty \,.
\]
Recalling that $\rho_{n} \le 2^{-1/2} \rho_{2n}$, we thus get similarly to \cite[(3.37) \& (3.38)]{BK02} that 
\[
L_{2n,2\ell} \le o_n(1) + \Big(  2^{-(\ell-1)} L_{n,2\ell}^{2\ell}  + 2^{-\ell} c_\ell \Big)^{1/(2\ell)}\,,
\]
from which it follows as in \cite{BK02} that $\sup_j \{L_{2^j,2\ell}\}$ is finite. Finally, for any 
$n \in [2^{j-1},2^j)$, $j \ge 2$, we have as in the preceding that
\[
R_{2^j} \stackrel{d}{=} R_n + \hat{R}_{2^j-n} - V_{0,n,2^j}.
\]
Upon centering, taking the $2\ell$-th power and isolating the $2\ell$-th power of 
$\ocrn$, the preceding identity results with
\[
[L_{2^j,2\ell} + o_j(1)]^{2\ell} + c_\ell  \ge   (\rho_n/\rho_{2^j})^{2\ell} L_{n,2\ell}^{2\ell}  \ge 
4^{-\ell} L_{n,2\ell}^{2\ell} \,.
\]
Thus, $\sup_n \{L_{n,2\ell}\}$ is finite as well, completing the induction step and thereby 
establishing \eqref{eq:BK-33a}. Finally, we get \eqref{eq:BK-33b} out of \eqref{eq:BK-33a}
precisely as in deriving \cite[(3.39)]{BK02} out of \cite[(3.40)]{BK02}.
\end{proof}

\blue{Recall the decomposition \eqref{eq:decomp4}-\eqref{def:Delta},
%Note that by shift invariance $R_{a,b} \stackrel{d}{=} R_{0,b-a}=R_{b-a}$ and
%$R_{a,b}$ is independent of $R_{b,c}$ (due to the independence of increments).  In particular, 
for the independent variables $U_j :=R_{n_{j-1},n_j}$ and any increasing $\{n_k\}$ starting at $n_0=0$, 
%one has the decomposition 
\begin{equation}\label{eq:decDelta}
R_{n_k} = \sum_{j=1}^k U_j - \Delta_{n_k,k} \,, \qquad 
\Delta_{n_k,k} := \sum_{j=1}^{k-1} V_{n_{j-1},n_j,n_k} 
%= \sum_{j=1}^{k-1} V_{0,n_j,n_{j+1}} 
\,.
\end{equation}
Centering the random variables of the preceding identity} 
%of \eqref{eq:decomp4} 
we arrive at
\begin{equation}\label{eq:decomp}
\overline{R}_{n_k} 
%= \sum_{j=1}^k \overline{R}_{n_{j-1},n_j} - \sum_{j=1}^{k-1} \overline{V}_{n_{j-1},n_j,n_k} 
= \sum_{j=1}^k \overline{U}_j - \overline{\Delta}_{n_k,k}  \,,
\end{equation}
with 
% $\Delta_{n_k,k}$ of \eqref{def:Delta} and the
zero-mean, independent variables 
$\overline{U}_j$.
% :=\overline{R}_{n_{j-1},n_j}$. 
Proceeding to show that $\overline{\Delta}_{n_k,k}$
has a negligible effect on $\overline{R}_{n_k}$, first recall from \eqref{hjm6} that
% $\{\overline{U}_j\}$ satisfy 
\begin{align}\label{hjm*}
\lim_{j\to \infty}\frac{E[\overline{U}_j^2]} {\rho_{n_j-n_{j-1}}^2} = \sigma_d^2 \,,
\end{align}
whereas \eqref{eq:BK-33a} at $a=n_{j-1}$, $b=n_j$, amounts to 
\begin{align}\label{hjm}
E[|\overline{U}_j|^m] \le (c_m \, \rho_{n_j-n_{j-1}} )^{m}\,.
\end{align}

In case $d=5$ we take the same values of $\alpha$, $\beta<1/2$ and $\{n_k\}$ 
as in the proof of \cite[Thm. 2.1]{BK02}. Lemma \ref{vv} at $\ell=4$ is then 
the analog of \cite[(3.2)]{BK02}, and utilizing it at $a=n_{j-1}$, $b=n_j$, $j<k$, we 
find by following verbatim, the derivation of \cite[(3.9)]{BK02}, that for some $c$ finite
\begin{align}\label{limv}
\limsup_{k \to \infty} 
\frac{|\overline{\Delta}_{n_k,k}|}{\sqrt{n_k} (\log n_k)^\beta} \le c
\quad \text{a.s.}
\end{align}
Thereafter,  
substituting (\ref{limv}) for \cite[(3.9)]{BK02} and \eqref{hjm} to get \cite[(3.16)]{BK02},
by the same reasoning as in the proof of \cite[Thm. 2.1]{BK02},  we find that a.s. 
\begin{equation}\label{eq:LIL-ni}
\lim_{k \to \infty} h_d(n_k)^{-1} \big[ \overline{R}_{n_k} - \sigma_d B_{\rho_{n_k}^2} \big] = 0 \,,
\end{equation}
for some one-dimensional standard Brownian motion $(B_t, t \ge 0)$. As shown after 
\cite[(3.17)]{BK02} (apart from replacing \cite[Lemma 3.3(b)]{BK02} by \eqref{eq:BK-33b}),
the stated \abbr{lil} is then a direct consequence of Kinchin's \abbr{LIL} for the latter Brownian motion.
 
In case $d\ge 6$, we take $\{n_k\}$ again as in the proof of \cite[Thm. 2.1]{BK02}, 
except that now this is done for the choice of $\alpha=1$. 
Then, by Lemma \ref{vv}, for $C=C_{d,2}$ and any $1 \le j < k$,
\[
{\rm Var} (V_{n_{j-1},n_j,n_k}) \le C ( \log (n_j-n_{j-1}) )^2 \le C (\log n_k)^2 \,.
\]
Thus, for any $\beta >0$ and all $k$, by Markov's inequality and \blue{the definition of 
$\Delta_{n_k,k}$ (see \eqref{eq:decDelta}),}
\begin{align*}
P\Big( \big|\overline{\Delta}_{n_k,k} \big|\ge n_k^\beta\Big) 
\le n_k^{-2 \beta} \mathrm{Var} \Big( \Delta_{n_k,k} \Big) 
\le C n_k^{-2 \beta} k^2 (\log n_k)^2 \,.
\end{align*}
Since $\big| \{n_k\} \cap [2^{\ell},2^{\ell+1}) \big| \le \ell$ for any $\ell \ge 1$, eventually 
$k \le (\log n_k)^2$.
Hence, by the first Borel-Cantelli lemma, we have that for any $\beta>0$, 
\begin{align}\label{limv6}
\limsup_{k \to \infty} 
n_k^{-\beta} \big| \overline{\Delta}_{n_k,k} \big|  \le 1
\quad \text{a.s.}
\end{align}
We then get \eqref{eq:LIL-ni} by following, as for $d=5$,
the proof of \cite[Thm. 2.1]{BK02}, utilizing again \eqref{hjm*}-\eqref{hjm}, 
while having now, via \eqref{limv6} at $\beta<1/2$, a negligible contribution 
at scale $\rho_n=\sqrt{n}$ (instead of \eqref{limv} and the scale $\sqrt{n \log n}$
throughout \cite[(3.13)-(3.17)]{BK02}). Finally, recall that $n_{k+1} -n_k \le n_k/\ell$
whenever $n_k \in [2^\ell,2^{\ell+1})$. Hence, in view of \eqref{eq:BK-33b} at $m=6$ 
and $\lambda=\varepsilon h_d(n_k)/\sqrt{n_{k+1}-n_k}$, we have that for 
some $c_{\varepsilon}$ finite, any $\varepsilon>0$ and $n_k \in [2^\ell,2^{\ell+1})$, 
\[
P \Big( \max_{n \in (n_k,n_{k+1})} \big\{ |\ocrn - \overline{R}_{n_k} | \big\} > \varepsilon h_d(n_k) \Big) 
\le c_{\varepsilon}\, \ell^{-3} \,.
\]
With at most $\ell$ values of such $n_k$, by the first Borel-Cantelli lemma, the events on the \abbr{lhs}
a.s. occur only for finitely many values of $k$ and the stated \abbr{LIL} thus 
follows, as before, from \eqref{eq:LIL-ni}.

\end{document}